\documentclass[a4paper,11pt]{article}

\usepackage[T1]{fontenc}
\usepackage{amsmath}
\usepackage{amsthm,ifthen,amssymb}
\usepackage{url,enumitem}
\usepackage[normalem]{ulem}
\usepackage[pdftex]{graphicx}
\usepackage[curve,knot]{xypic}
\usepackage[colorinlistoftodos]{todonotes}
\usepackage{upgreek}
\usepackage{sectionbreak}
\usepackage{mathrsfs}

\usepackage{enumitem}
\setlist[enumerate,1]{label=\textup{(\alph*)}}

\usepackage{yhmath}

\usepackage{tikz}
\usepackage{tikz-cd}
\usetikzlibrary{calc,math}
\usepgflibrary{shapes.geometric}
\usepgflibrary{shapes.misc}
\usetikzlibrary{positioning}
\usetikzlibrary{decorations}
\usetikzlibrary{arrows}

\usepackage[a4paper,left=1in,width=458pt,textheight=650pt,top=1.5in]{geometry}

\usepackage{bbm}

\usepackage{hyperref}

\hypersetup{
	pdftitle={Abstract Cluster Structures},
	pdfauthor={Jan E. Grabowski and Sira Gratz},
	pdfstartview=FitH
}

\DeclareMathOperator{\fpmod}{mod}
\DeclareMathOperator{\Mod}{Mod}
\DeclareMathOperator{\fd}{fd}

\tikzcdset{
	infl/.style={rightarrowtail},
	defl/.style={twoheadrightarrow},
	confl/.style={dashed},
	exists/.style={dotted}}
	
\newcommand{\cindspace}{}

\newcommand{\Abcat}{\mathrm{Ab}}
\newcommand{\addcat}[1]{{#1}^{\textup{\text{add}}}}

\newcommand{\adj}[1]{{#1}^{\dagger}}

\newcommand{\bigdsum}{\bigoplus}

\newcommand{\blank}{-}

\newcommand{\card}[1]{\lvert #1 \rvert}

\let\chisave\chi
\renewcommand{\chi}{{%
		\mathchoice{\raisebox{0.25ex}{$\displaystyle\chisave$}}
		{\raisebox{0.2ex}{$\textstyle\chisave$}}
		{\raisebox{0.2ex}{$\scriptstyle\chisave$}}
		{\raisebox{0.1ex}{$\scriptscriptstyle\chisave$}}}}

\newcommand{\coind}[2]{\text{\textup{coind}}_{#1}^{#2}}

\newcommand{\coprodsml}{\amalg}

\newcommand{\cross}{\times}
\newcommand{\curly}[1]{\ensuremath{\mathcal{#1}}}

\newcommand{\defeq}{:=}
\let\degsave\deg
\renewcommand{\deg}{\underline{\degsave}}

\newcommand{\disjointunion}{\sqcup}
\newcommand{\dsum}{\ensuremath{ \oplus}}
\newcommand{\dual}[1]{\ensuremath {#1}^{*}}

\renewcommand{\epsilon}{\varepsilon}
\newcommand{\Ext}[4]{\ensuremath \operatorname{Ext}^{#1}_{#2}(#3,#4)}

\newcommand{\fabcat}{\mathrm{fAb}}

\newcommand{\from}{\leftarrow}

\newcommand{\Hom}[3]{\ensuremath \operatorname{Hom}_{#1}(#2,#3)}

\newcommand{\id}{\ensuremath \textup{id}}
\renewcommand{\iff}{\ensuremath \Longleftrightarrow}

\newcommand{\Image}[1]{\ensuremath \mbox{Im}\>{#1}}

\newcommand{\ind}[2]{\text{\textup{ind}}_{#1}^{#2}}

\newcommand{\inj}{\hookrightarrow}
\newcommand{\integ}{\ensuremath{\mathbb{Z}}}
\newcommand{\intersection}{\mathrel{\cap}}
\newcommand{\ip}[2]{\ensuremath \langle #1,#2 \rangle}
\newcommand{\Pip}[2]{\ensuremath \{ #1,#2 \}}

\newcommand{\iso}{\ensuremath \cong}

\newcommand{\Ker}{\operatorname{Ker}}

\newcommand{\Kgp}[1]{\mathrm{K}_{0}(#1)}

\newcommand{\Mor}[3]{\text{Mor}_{#1}(#2 , #3 )}

\newcommand{\nat}{\ensuremath \mathbb{N}}

\newcommand{\Obj}[1]{\text{Obj}(#1) }

\newcommand{\onto}{\twoheadrightarrow}
\newcommand{\op}[1]{{#1}^{\mbox{\scriptsize \textup{op}}}}

\newcommand{\opevform}[2]{\ip{#1}{#2}_{\textup{ev}}^{\textup{op}}}

\renewcommand{\phi}{\varphi}
\newcommand{\prodsml}{\mathop{\Pi}}

\newcommand{\rank}{\mathop{\mathrm{rank}}}
\newcommand{\real}{\ensuremath \mathbb{R}}

\newcommand{\stab}[1]{\underline{#1}}

\newcommand{\tensor}{\ensuremath \otimes}

\newcommand{\union}{\mathrel{\cup}}

\renewcommand{\leq}{\leqslant}

\newcommand{\stabind}[2]{{\underline{\text{\textup{ind}}}\cindspace}_{#1}^{#2}}
\newcommand{\stabindbar}[2]{{\underline{\overline{\text{\textup{ind}}}}\cindspace}_{#1}^{#2}}
\newcommand{\stabcoind}[2]{{\underline{\text{\textup{coind}}}\cindspace}_{#1}^{#2}}
\newcommand{\stabcoindbar}[2]{{\underline{\overline{\text{\textup{coind}}}}\cindspace}_{#1}^{#2}}

\newcommand{\abgp}{A}

\newcommand{\ex}{\mathbf{ex}}

\newcommand{\Spec}[1]{\operatorname{Spec} #1} 
\newcommand{\Sch}[1]{\operatorname{Sch}\!/#1}

\newcommand{\cf}{cf.\ }
\newcommand{\ie}{i.e.\ }

\newcommand{\bF}{\mathbf{F}}
\newcommand{\cA}{\mathcal{A}}
\newcommand{\cC}{\mathcal{C}}
\newcommand{\cE}{\mathcal{E}}

\newcommand{\cT}{\mathcal{T}}
\newcommand{\cX}{\mathcal{X}}
\newcommand{\uk}{\underline{k}}
\newcommand{\sX}{\underline{\mathcal{X}}}
\newcommand{\sZ}{\underline{\integ}}

\newcommand{\Ark}[1]{\cA\textup{-rank}(#1)}
\newcommand{\Xrk}[1]{\cX\textup{-rank}(#1)}

\newcommand{\trk}[1]{\mathrm{t.rk}(#1)}
\newcommand{\mrk}[1]{\mathrm{m.rk}(#1)}
\newcommand{\frk}[1]{\mathrm{f.rk}(#1)}

\newcommand{\Coker}{\operatorname{Coker}}

\newcommand{\signededge}[2]{#1 \stackrel{\pm}{\to} #2}

\DeclareMathOperator{\lfd}{lfd}

\newcommand{\genericACS}{(\cE,\cX,\beta,\cA,\ip{-}{-})}
\newcommand{\genericAQCS}{(\cE,\cX,\beta,\lambda,\cA,\ip{-}{-})}

\theoremstyle{plain}
\newtheorem{theorem}{Theorem}[section]
\newtheorem*{theorem*}{Theorem}
\newtheorem{proposition}[theorem]{Proposition}
\newtheorem{lemma}[theorem]{Lemma}
\newtheorem{corollary}[theorem]{Corollary}

\theoremstyle{definition}
\newtheorem{definition}[theorem]{Definition}
\newtheorem*{definition*}{Definition}

\newtheorem*{question*}{Question}
\newtheorem*{convention*}{Convention}
\theoremstyle{remark}
\newtheorem{remark}[theorem]{Remark}

\newtheorem{example}[theorem]{Example}

\newtheorem*{example*}{Example}
\newtheorem*{examplectd*}{Example (continued)}
\renewcommand{\tau}{\uptau}
\renewcommand{\gamma}{\upgamma}
\let\betasave\beta
\newcommand{\itbeta}{\betasave}
\renewcommand{\beta}{\upbeta}
\renewcommand{\lambda}{\uplambda}

\newcommand{\cB}{\mathcal{B}}
\newcommand{\cU}{\mathcal{U}}

\newcommand{\bK}{\mathbb{K}}

\newcommand{\bigintersection}{\bigcap}
\newcommand{\bigunion}{\bigcup}

\renewcommand{\epsilon}{\varepsilon}

\renewcommand{\iff}{\ensuremath \Longleftrightarrow}

\newcommand{\qtorus}[3]{\mathbb{T}_{#1}^{#2}({#3})}
\newcommand{\qtorusaffine}[3]{\mathbb{TA}_{#1}^{#2}({#3})}

\newcommand{\units}[1]{{#1}^{\cross}}

\newcommand{\ctsubcat}{\mathrel{\subseteq_{\mathrm{ct}\,}}}

\renewcommand{\leq}{\leqslant}

\newcommand{\invfrozen}{\mathbf{inv}}

\newcommand{\leftapp}[2]{L_{#1}{#2}}
\newcommand{\leftcok}[2]{C_{#1}{#2}}
\newcommand{\rightapp}[2]{R_{#1}{#2}}
\newcommand{\rightker}[2]{K_{#1}{#2}}

\newcommand{\bplus}[2]{[#1]_{+}^{#2}}
\newcommand{\bminus}[2]{[#1]_{-}^{#2}}
\newcommand{\bplusminus}[2]{[#1]_{\pm}^{#2}}
\newcommand{\bminusplus}[2]{[#1]_{\mp}^{#2}}

\newcommand{\canform}[3]{\ip{#1}{#2}_{#3}}

\newcommand{\evform}[2]{\ip{#1}{#2}_{\textup{ev}}}

\renewcommand{\emptyset}{\varnothing}

\newcommand{\Setcat}{\textbf{\textup{Set}}}

\newcommand{\Catcat}{\textbf{\textup{Cat}}}

\newcommand{\AQCScat}{\textbf{\textup{AQCS}}}
\newcommand{\ACScat}{\textbf{\textup{ACS}}}

\DeclareMathOperator{\proj}{proj}

\title{Abstract cluster structures}
\author{Jan E. Grabowski\footnotemark[2] 
	\\ \small{\textit{School of Mathematical Sciences, Lancaster University,}}
	\\ \small{\textit{Lancaster, LA1 4YF, United Kingdom}}
	\and Sira Gratz\footnotemark[3]
	\\ \small{\textit{Department of Mathematics, Aarhus University}}
	\\ \small{\textit{Ny Munkegade 118, 8000 Aarhus C, Denmark}}
}
\date{14th May 2026}

\begin{document}
	
	\maketitle
	
	\renewcommand{\thefootnote}{\fnsymbol{footnote}}
	\footnotetext[2]{Email: \url{j.grabowski@lancaster.ac.uk}.  Website: \url{http://www.maths.lancs.ac.uk/~grabowsj/}}
		\footnotetext[3]{Email: \url{sira@math.au.dk}. Website: \url{https://sites.google.com/view/siragratz}}
\renewcommand{\thefootnote}{\arabic{footnote}}
\setcounter{footnote}{0}

\begin{abstract}
We describe a framework for encoding cluster combinatorics using categorical methods.  We give a definition of an \emph{abstract cluster structure}, which captures the essence of cluster mutation at a tropical level and show that cluster algebras, cluster varieties, cluster categories and surface models all have associated abstract cluster structures.  For the first two classes, we also show that they can be constructed from abstract cluster structures.

By defining a suitable notion of morphism of abstract cluster structures, we introduce a category of these and show that it has several desirable properties, such as initial and terminal objects and finite products and coproducts.  We also prove that rooted cluster morphisms of cluster algebras give rise to morphisms of the associated abstract cluster structures, so that our framework includes a version of the extant category of cluster algebras.  

We can do more, however, because we can relate different types of representation of abstract cluster structures (cluster algebra, varieties, categories) directly via morphisms of their associated abstract cluster structures, even though no direct map from e.g.\ a cluster category to the associated cluster algebra is possible.
	
In fact, we do much of the above in the setting of abstract quantum cluster structures, with some analysis of the difference between the category of these and that of the unquantized version.  In order to show the relationship between abstract quantum cluster structures and quantum cluster algebras, we reformulate the usual construction of the latter in a way that is more amenable to our purposes and which we expect will be of independent interest and use.
	
	\vspace{1em}
\end{abstract}

\vfill
\pagebreak

\setcounter{tocdepth}{2}
\tableofcontents

\sectionbreak
\section*{Introduction}\label{s:intro}
\addcontentsline{toc}{section}{Introduction}

The goal of this work is to introduce a framework in which we can formalise the notion of ``having cluster combinatorics''.  Cluster combinatorics appear in a number of different settings: cluster algebras, cluster varieties, cluster categories, geometric models and potentially others.  It is difficult to handle these all together algebraically, since no one category naturally holds them all.

It is also well-known that there are technical issues with trying to define a category of cluster algebras.  At its heart, this boils down to the fact that admitting a cluster algebra structure means having a presentation by generators and relations of a particular form, and homomorphisms of algebras will rarely respect the presentations.  As a result, very few of the standard constructions of sub-objects, quotients, sums etc., work as one would like.

We will distil out the cluster combinatorics to define a notion of an \emph{abstract cluster structure}, keeping some common features from the different classes of examples above but disregarding others that are individual to a particular type of realisation.  Then, crucially, abstract cluster structures will all be of the same type---in fact, they will be categories with certain additional data---and indeed we can define a category of abstract cluster structures. Having done so, we then have access to the standard constructions, but in a different and better-behaved category. We develop this theory in Part~\ref{P:ACS}.

Philosophically, we regard cluster algebras, cluster varieties and other classes of examples as \emph{representations} of abstract cluster structures. In Part~\ref{P:reps}, we treat each of the different settings for cluster theory and show that they have associated abstract cluster structures.  For cluster algebras and cluster varieties, we can show that this process is reversible: to each abstract cluster structure, we may associate a cluster algebra or variety.  The corresponding claims for cluster categories or geometric models are beyond the scope of this work, however.

Having done this, we also able to build bridges between the various settings.  We can formalise the notion of ``type'' by saying that two cluster-theoretic objects having the same cluster type means having isomorphic abstract cluster structures.  Our lowest-level piece of data in an abstract cluster structure will turn out to be a version of the exchange graph, so this is indeed a plausible approach to take.

For cluster categories admitting decategorifications via cluster characters, we expect that the cluster category and its associated cluster algebra have isomorphic abstract cluster structures.  Similarly, we would hope to link cluster algebras and varieties, or cluster algebras and their geometric models, via isomorphisms of abstract cluster structures.  Much of the detail is left to future work, but we illustrate the principle via an example of the latter type.

Furthermore, we can explore other natural relationships arise from morphisms of abstract cluster structures that are not isomorphisms, for example, products or quotients.  This opens up the possibility of insight into ``cluster algebra'' (as opposed to just ``cluster theory''), through the study of the category of abstract cluster structures. We discuss initial and terminal objects and products and coproducts in detail.

We also treat the quantization of abstract cluster structures and associated objects alongside their classical counterparts.  In particular, we devote a portion of Section~\ref{s:linear-reps} to the construction of quantum cluster algebras in a way that is closely aligned with the approach encouraged by abstract cluster structures.

The two Parts of this work are intimately related.  The examples contained in the second provide motivation and justification for the abstract definitions in the first.  However, the formal statements of the second require the framework of the first.  As such, we recommend a non-linear reading of this paper.  

We anticipate that most readers will be familiar with one or other of the examples in Part~\ref{P:reps} and we suggest keeping one such in mind when first reading Part~\ref{P:ACS}.  Indeed, by looking in the relevant section in Part~\ref{P:reps}, you will find the dictionary between your favourite example and the abstract version.

We now briefly indicate the main definitions and results to be found in each section, avoiding notation or technicality where possible:

\begin{description}
	\item[Part~\ref{P:ACS} Abstracting cluster combinatorics] {\ }
	\begin{description}
		\item[\S\ref{s:pairings}] We give a definition of a pairing between two functors having values in Abelian groups and discuss non-degeneracy and left and right radicals.
		\item[\S\ref{s:ACS} Abstract cluster structures] {\ }
		\begin{description}
			\item[\S\ref{ss:def-of-ACS}] We give our main definition, that of an abstract cluster structure over a simple directed graph $E$.  To do so, we construct a signed path category $\cE$ over the graph and define an abstract cluster structure in terms of two functors $\cX\colon \op{\cE}\to \Abcat$ and $\cA\colon \cE\to \Abcat$, a factorization $\beta\colon \cX \to \cA$ and a right non-degenerate pairing $\canform{\blank}{\blank}{}\colon \cA \tensor_{\integ} \cX \to \sZ$.
			\item[\S\ref{ss:connectedness}] We discuss the various types of connectedness that $E$ could have and their impact on an associated abstract cluster structure.  We also introduce several notions of rank.
			\item[\S\ref{ss:principal-part}] We show that by taking the quotient with respect to the left radical of the pairing, we may obtain an abstract cluster structure on the same underlying graph that behaves as the ``principal part'', i.e.\ corresponds to deleting frozen variables.
			\item[\S\ref{ss:forms-skew-sym}] We discuss skew-symmetrizability of abstract cluster structures, which is not part of the definition, and give several equivalent formulations.
			\item[\S\ref{ss:quantum-str}] We introduce the additional datum $\lambda$ needed to define a quantum abstract cluster structure and show that every quantum cluster structure arises from a choice of retraction of $\beta$.
		\end{description}
		\item[\S\ref{s:cat-of-ACS} The category of abstract (quantum) cluster structures] {\ }
		\begin{description}
			\item[\S\ref{ss:morphisms-in-ACS}] We give a natural definition of morphism of abstract quantum cluster structures, with defining data a functor $F\colon \cE_{1}\to \cE_{2}$ of the signed path categories and $\chi\colon \cX_{1}\to \cX_{2}\op{F}$ and $\alpha\colon \cA_{1}\to \cA_{2}F$ natural transformations with some compatibility conditions.
			\item[\S\ref{ss:cat-of-ACS}] Then the category of abstract quantum cluster structures $\AQCScat$ is defined to have the obvious objects and the aforementioned morphisms.  We also define $\ACScat$, noting that there is a forgetful functor $\mathcal{F}$ from the former to the latter but this lacks several desirable properties.
			\item[\S\ref{ss:monic-epic-iso-in-ACS}] 
			We identify the isomorphisms in $\ACScat$ and $\AQCScat$ as those morphisms whose components are all isomorphisms and comment on why this seemingly strong set of conditions is appropriate.
			\item[\S\ref{ss:initial-terminal}] We show that $\ACScat$ has initial and terminal objects and that the initial object of $\ACScat$ is also initial in $\AQCScat$. However, the quantization of the terminal object of $\ACScat$ is not terminal in $\AQCScat$, due to the necessary morphism failing to exist.  Also, we note that even in $\ACScat$, the initial and terminal objects are not isomorphic.
			\item[\S\ref{ss:props-of-ACS-cat}] We show that $\ACScat$ has all finite products and coproducts, corresponding to a direct product construction and disjoint union respectively, but these are not biproducts.   Similarly, the category $\AQCScat$ has all finite coproducts.
		\end{description}
	\end{description}
	\item[Part~\ref{P:reps} Representations of cluster combinatorics] {\ }
	\begin{description}
		\item[\S\ref{s:linear-reps} Linear representations] This section considers the relationship between (quantum) cluster algebras and abstract (quantum) cluster structures.
		\begin{description}
			\item[\S\ref{ss:quantum-tori}] We begin by reconstructing quantum cluster algebras in a different (but equivalent) way to the usual definition, using a $\integ$-linear map $\beta$ rather than a matrix. Doing so has several technical advantages, not least that it makes the results of \S\ref{ss:AQCS-from-QCA} and \S\ref{ss:mor-from-RCM} possible.
			\item[\S\ref{ss:toric-frames}] Here, we establish the key notions of mutation of the labelling sets for our clusters and of the free Abelian groups over these, where the latter mutations are exactly the \emph{tropical mutation} maps introduced by \cite{FockGoncharov}. We prove the key propositions needed, e.g.\ involutivity, compatibility of mutated data etc.
			\item[\S\ref{ss:QCAs}] We define mutation of quantum cluster variables, show that this is involutive and define quantum cluster algebras. 
			\item[\S\ref{ss:AQCS-from-QCA}] We show that every quantum cluster algebra gives rise to an abstract quantum cluster structure, by ``tropicalization'' or ``taking logarirthms''.
			\item[\S\ref{ss:QCA-from-AQCS}] Conversely, every abstract quantum cluster structure gives rise to a quantum cluster algebra, by ``exponentiation'', in the form of the passage from a free Abelian group $\integ[\cB]$ to a quantum torus $\qtorus{q}{\lambda}{\cB}=(\bK \dual{\integ[\cB]})^{\Omega_{q}^{\lambda}}$ obtained as a cocycle twist of an associated group algebra.
		\end{description}
		\item[\S\ref{s:geom-reps} Geometric representations] {\ }
		\begin{description}
			\item[\S\ref{ss:cluster-varieties}] This section discusses cluster varieties and establishes that Poisson cluster varieties yield abstract quantum cluster structures and vice versa. Indeed, from an abstract quantum cluster structure, we obtain preschemes $\mathbb{A}$ and $\mathbb{X}$ by gluing the tori $\qtorus{\cA}{}{c}=\Hom{\Abcat}{\cA c}{\mathbb{G}_{m}}$ and $\qtorus{\cX}{}{c}=\Hom{\Abcat}{\cX c}{\mathbb{G}_{m}}$ using the morphisms of tori $A\alpha^{+}\colon \qtorus{\cA}{}{c}\to \qtorus{\cA}{}{d}$ (for $\alpha^{+}\colon c\to d$) and $X\alpha^{+}$, respectively. Then, with some assumptions in place, the quantum datum $\lambda$ endows a Poisson structure on $\mathbb{A}$.
			\item[\S\ref{ss:triangulations}] We show that the usual approach to mutating of arcs and  triangulations of marked surfaces can be enhanced to include a notion of mutating quadrilaterals, and hence from a marked surface we may obtain an abstract cluster structure.
		\end{description}
		\item[\S\ref{s:cat-reps} Categorical representations] {\ }
		\begin{description}
			\item[\S\ref{ss:cluster-cats}] We summarize the work of the first author and Pressland in \cite{CSFRT}, which enables us to see that we obtain a natural abstract cluster structure from any cluster category of finite rank.
		\end{description}
		\item[\S\ref{s:morphisms-of-reps} Morphisms of representations] This section is an initial exploration of morphisms in $\ACScat$ associated to the above classes of representations (i.e.\ abstract cluster structures associated to cluster algebra, varieties and categories and triangulations).
		\begin{description}
			\item[\S\ref{ss:mor-of-exch-tree}] We set up some necessary technicalities for the following section.  In particular, we establish a notion of morphism between the exchange trees of two cluster algebras that permits specialisation of variables.  This is done by associating to the tree a functor $\mu_{\cB}\colon \cE(\ex,\cB)\to \Setcat$ taking a mutation sequence $\uk$ to $\mu_{\uk}(\cB)^{0}$ \linebreak (an indexing set for the cluster variables, together with $0$) and arrows in the exchange tree to bijections of these.  Then an $\ex$-admissible map $\phi\colon \cB_{1}^{0}\to \cB_{2}^{0}$ of the initial indexing set induces a functor $\bF$ sending admissible mutation sequences for one cluster algebra to admissible mutation sequences for a second and also induces a natural transformation $\mu_{\cB_{1}}\to \mu_{\cB_{2}}\circ \bF$.
			\item[\S\ref{ss:mor-from-RCM}] The main theorem demonstrating the relationship between our framework and the existing approaches is proved here.  Namely, we show that by $\integ$-linearizing the natural transformation of the previous section, we have that a (consistently positive) rooted cluster morphism of cluster algebras gives rise to a natural morphism of the associated abstract cluster structures.
			\item[\S\ref{ss:classifying-mor}] We finish by explaining, by means of a worked example, how one might construct morphisms of abstract cluster structures between two different types of representation.  Specifically, we look at the cluster algebra structure on $\curly{O}(\mathrm{Gr}(2,6))$, the homogeneous coordinate ring of the $(2,6)$ Grassmannian, and the well-known model for its cluster combinatorics in terms of triangulations of a hexagon.  We see that there is a natural isomorphism of their respective abstract cluster structures, as one would hope.
		\end{description}
	\end{description}
\end{description}

\subsection*{Acknowledgements}

We are particularly grateful to Matthew Pressland for helpful conversations, notably in relation to \cite{CSFRT}, and also to colleagues at seminar and conference talks where preliminary versions of some of the constructions and results here were presented and useful questions and comments were received.  The authors would also like to thank Lancaster University for financial support.

\subsection*{Funding declaration}

Sira Gratz was supported by a research grant (VIL42076) from VILLUM FONDEN.

\pagebreak

\part{Abstracting cluster combinatorics}\label{P:ACS}

\sectionbreak
\section{Pairings and dualities}\label{s:pairings}

For any Abelian group $A$, denote by $A^{*}$ the $\integ$-dual $\Hom{\integ}{A}{\integ}$ and let $\evform{-}{-}$ denote the canonical $\integ$-bilinear form $\evform{-}{-}\colon A\cross A^{*}\to \integ$, $\evform{a}{f}=f(a)$.  We will refer to this as the \emph{evaluation form}.  We will abuse notation mildly and also write $\evform{-}{-}$ for the opposite form $\opevform{-}{-}\colon \dual{A}\cross A\to \integ$, $\opevform{f}{a}=f(a)$.

Let $A,B$ be Abelian groups and let $\ip{-}{-}\colon A\cross B \to \integ$ be any $\integ$-bilinear form. Consider the induced maps $\delta_{A}\colon A\to B^{*}$, $\delta_{A}(a)=\ip{a}{-}$, and $\delta_{B}\colon B\to A^{*}$, $\delta_{B}(b)=\ip{-}{b}$. The kernels of $\delta_A$ and of $\delta_B$ are called the \emph{left radical} and the \emph{right radical} of $\ip{-}{-}$ respectively.

A $\integ$-bilinear form $\ip{-}{-}\colon A\cross B \to \integ$ is said to be a \emph{non-degenerate pairing} if it has trivial left and right radicals.  Further, the form is called \emph{perfect} if $\delta_{A}$ and $\delta_{B}$ are isomorphisms.  

Note that the induced maps associated to the evaluation form in particular have the following properties.  Firstly, $\delta_{A^{*}}^{\text{ev}}=\id_{A^{*}}$, and secondly, $\delta_{A}^{\text{ev}}$ is injective if and only if $A$ is free (and an isomorphism if and only if $A$ is free and finitely generated).

Denote by $\Abcat$ the category of Abelian groups and $\fabcat$ its full subcategory whose objects are the finitely generated free Abelian groups together with $\bigdsum_{\nat} \integ$ and $\prod_{\nat} \integ$. Let $\abgp$ be an Abelian group and (by mild abuse) let $\abgp$ also denote the subcategory of $\Abcat$ with the single object $\abgp$ and unique morphism $\id_{\abgp}\colon \abgp \to \abgp$.

For $A,B\in \fabcat$ of finite rank, if these ranks are equal and $\ip{-}{-}\colon A\cross B\to \integ$ is a non-degenerate pairing then the images of $\delta_{A}$ and $\delta_{B}$ are finite-index subgroups of their codomains.  

\begin{convention*} Throughout this work, functors are covariant unless otherwise stated.  Contravariant functors will be expressed as covariant functors from the opposite category of the domain.
\end{convention*}

\begin{definition} For any category $\cC$ and any functor $F\colon \cC\to \Abcat$, we define the dual functor $F^{*}\colon \op{\cC}\to \Abcat$ to be $F^{*}\defeq \Hom{\integ}{F\blank}{\integ}=\Hom{\integ}{-}{\integ}\circ F$.  Explicitly, $F^{*}C=\Hom{\integ}{FC}{\integ}=(FC)^{*}$ and $F^{*}f=\Hom{\integ}{Ff}{\integ}=(Ff)^{*}$.  
\end{definition}

\begin{remark}
	If $F\colon \cC \to \fabcat$ then $\dual{F}$ takes values in $\fabcat$, so that we may regard $\dual{F}$ as a functor $\dual{F}\colon \op{\cC}\to \fabcat$.  For the functor $\Hom{\integ}{-}{\integ}\colon \fabcat \to \fabcat$ is a duality, since the definition of $\fabcat$ ensures it is closed under taking $\integ$-duals \cite{Specker}.  
	
	Many of the functors we consider will have their images in $\fabcat$ but for a mix of technical and stylistic reasons (e.g.\ to enable use of pre-existing terminology) we will mostly choose $\Abcat$ as their codomain.
\end{remark}

Given functors $F \colon \cC \to \Abcat$ and $G \colon \op{\cC} \to \Abcat$ we define the functor $F \otimes_\integ G \colon \cC \times \op{\cC} \to \Abcat$ to be the composition
\[
\xymatrix{\cC \cross \op{\cC} \ar[r]^-{F \cross G} & \Abcat \cross \Abcat \ar[r]^-{\tensor} & \Abcat}
\]
for $\tensor$ giving the standard monoidal structure on $\Abcat\equiv \Mod \integ$.  We may extend this notion to natural transformations $\alpha\colon F\to F'$ and $\itbeta\colon G\to G'$ to obtain $\alpha \tensor_{\integ} \itbeta$ in the obvious way.

Note that even if $F$ and $G$ land in the full subcategory $\fabcat$ of $\Abcat$, the functor $F \otimes_\integ G$ in general takes values in $\Abcat$, since tensoring with the Baer--Specker group $\prod_\mathbb{N}\integ$ may yield an object outside of $\fabcat$.

For any category $\cC$, let $\sZ\colon \cC\to \Abcat$ denote the constant functor with $\sZ c=\integ$ and $\sZ f=\id_{\integ}$ for any object $c$ and any morphism $f$ in $\cC$. 

We can use the above to extend the notion of a $\integ$-bilinear form on a pair of Abelian groups to that of a form on a pair of functors taking values in $\Abcat$.

\begin{definition}\label{d:pairing}
	Let $F \colon \cC \to \Abcat$ and $G \colon \op{\cC} \to \Abcat$ be functors. A \emph{pairing} of $F$ and $G$ is a dinatural transformation 
	\[ \ip{\blank}{\blank}\colon F\tensor_{\integ} G \to \sZ. \]
\end{definition}

A dinatural transformation has a component for each object of $\cC$; we write $\ip{\blank}{\blank}_{c}$ for this.  The following lemma demonstrates what dinaturality means in practice in this situation: it is the data of bilinear forms $\ip{\blank}{\blank}_{c}\colon Fc\tensor Gc\to \integ$ for each $c\in \cC$ in such a way that we have an adjointness relationship between $Ff$ and $Gf$ for any morphism $f$ in $\cC$.

\begin{lemma}\label{l:adjoint} Let $\ip{\blank}{\blank}\colon F\tensor_{\integ} G \to \sZ$ be a pairing of $F$ and $G$ as above.  Then for all morphisms $f\colon c\to d$ in $\cC$, we have that
	\[ \ip{\blank}{Gf(\blank)}_{c}=\ip{Ff(\blank)}{\blank}_{d}. \]
\end{lemma}

\begin{proof} The dinaturality condition means the commuting of the following diagram:
	\[ \begin{tikzcd}
		& (F\tensor G)(c,c) \arrow{r}{\ip{\blank}{\blank}_{c}} & \sZ(c,c) \arrow{rd}{\sZ(f,\id_{c})} & \\
		(F\tensor G)(c,d) \arrow{ru}{(F\tensor G)(\id_{c},f)} \arrow{rd}[below left]{(F\tensor G)(f,\id_{d})} &&& \sZ(d,c)\\
		& (F\tensor G)(d,d) \arrow{r}[below]{\ip{\blank}{\blank}_{d}} & \sZ(d,d) \arrow{ru}[below right]{\sZ(\id_{d},f)} &
	\end{tikzcd}\]
	
	From the definitions of $F\tensor G$ and $\sZ$, we see that this simplifies to
	\[ \begin{tikzcd}
		& Fc\tensor Gc \arrow{rd}{\ip{\blank}{\blank}_{c}} & \\
		Fc\tensor Gd \arrow{ru}{\id_{Fc}\tensor Gf} \arrow{rd}[below left]{Ff\tensor \id_{Gd}} && \integ\\
		& Fd\tensor Gd \arrow{ru}[below right]{\ip{\blank}{\blank}_{d}} & 
	\end{tikzcd}
	\]
	from which we obtain the claim.
\end{proof}

Note that the proof makes clear that the converse also holds.

\begin{definition}\label{d:nondeg} 	Let $F \colon \cC \to \Abcat$ and $G \colon \op{\cC} \to \Abcat$ be functors and let $\ip{\blank}{\blank}\colon F\tensor_{\integ} G \to \sZ$ be a pairing of $F$ and $G$.
	\begin{enumerate}[label=(\alph*)]
		\item\label{d:nondeg-left} Let $\delta_{F}\colon F\to \dual{G}$ be the natural transformation with components $(\delta_{F})_{c}=\delta_{Fc}$ for \linebreak $\delta_{Fc}\colon Fc\to \dual{(Gc)}$, $\delta_{Fc}(a)=\ip{a}{\blank}_{c}$. We say $\ip{\blank}{\blank}$ is \emph{left non-degenerate} if $\Ker \delta_{F}=0$.
		\item\label{d:nondeg-right} Let $\delta_{G}\colon G\to \dual{F}$ be the natural transformation with components $(\delta_{G})_{c}=\delta_{Gc}$ for \linebreak $\delta_{Gc}\colon Gc\to \dual{(Fc)}$, $\delta_{Gc}(b)=\ip{\blank}{b}_{c}$. We say $\ip{\blank}{\blank}$ is \emph{right non-degenerate} if $\Ker \delta_{G}=0$.
		\item\label{d:nondeg-both} We say $\ip{\blank}{\blank}$ is \emph{non-degenerate} if every component $\ip{\blank}{\blank}_{c}$ ($c\in \cC$) is non-degenerate, i.e.\ $\ip{\blank}{\blank}$ is left and right non-degenerate.
		
	\end{enumerate}
\end{definition}

Here, $\Ker \delta_{F}$ refers to the functor with $(\Ker \delta_{F})c=\Ker (\delta_{F})_{c}=\Ker \delta_{Fc}$ and similarly for $\Ker \delta_{G}$; also, $0$ refers to the zero functor with $0c=0\in\Abcat$.

The following is easily checked using Lemma~\ref{l:adjoint}.

\begin{corollary}
	Let $F\colon \cC \to \Abcat$ and $\dual{F}$ its dual.  The evaluation forms $\evform{\blank}{\blank}\colon Fc\cross \dual{F}c \to \integ$ define a right non-degenerate pairing, the \emph{evaluation pairing}, $\evform{\blank}{\blank}\colon F\tensor_{\integ} \dual{F}\to \sZ$. \qed
\end{corollary}

Similarly, we have a pairing $\op{\ip{\blank}{\blank}}_{\mathrm{ev}}\colon \dual{F}\tensor_{\integ} F\to \sZ$ induced by the opposite forms.  As above, we will abuse notation and write $\evform{\blank}{\blank}$ for both the evaluation pairing and its opposite, since it will be clear from the context which is meant.

Recall that a subgroup $L'$ of a free Abelian group $L$ is said to be \emph{saturated} if for all $l\in L$, $\lambda \in \integ\setminus \{ 0\}$ we have $\lambda l\in L'$ implies $l\in L'$.  This is equivalent to the quotient $L/L'$ being torsion-free, i.e. also free Abelian.  

\begin{definition}
	Let $F\colon \cC \to \Abcat$ be a functor whose essential image is contained in $\fabcat$.  We say that a subfunctor $G$ of $F$ is \emph{saturated} if $Gc$ is a saturated subgroup of $Fc$ for all $c\in \cC$.
\end{definition}

The following lemma will be useful later. 

\begin{lemma}\label{l:Ker-delta-saturated} Let $F\colon \cC \to \Abcat$ and $G\colon \op{\cC}\to \Abcat$ be functors whose essential images are contained in $\fabcat$.  Let $\ip{\blank}{\blank}\colon F\tensor_{\integ} G\to \sZ$ be a pairing of $F$ and $G$.  Then the subfunctor $\Ker \delta_{F}$ of $F$ and the subfunctor $\Ker \delta_{G}$ of $G$ are saturated.
\end{lemma}

\begin{proof}
	For all $c\in \cC$, we have that
	\[ \Ker \delta_{Fc} = \{ a\in Fc \mid \ip{a}{Gc}_{c}=0 \}. \]
	Then if $a\in Fc$, $\lambda\in \integ\setminus \{0\}$ and $\lambda a\in \Ker \delta_{Fc}$, we have that $\ip{\lambda a}{b}_{c}=0$ for all $b\in Gc$. Since $\ip{\blank}{\blank}_{c}$ is $\integ$-bilinear and $\lambda\neq 0$, we deduce that $\ip{a}{b}_{c}=0$ for all $b\in Gc$ and hence $a\in \Ker \delta_{Fc}$. As this holds for all $c$, $\Ker \delta_{F}$ is a saturated subfunctor.
	
	The argument for $\Ker \delta_{G}$ is entirely analogous.
\end{proof}

\sectionbreak
\section{Abstract cluster structures}\label{s:ACS}

\subsection{Definition}\label{ss:def-of-ACS}

Let $E$ be a simple\footnote{A directed graph is simple if it has no loops and at most one directed edge having given source and target vertices; $2$-cycles are permitted.} directed graph with vertex set $C$. We assign to it a graph $E_{\pm}$, which is defined to be the (signed\footnote{There are a variety of notions of signed directed graphs in the literature.  The term is usually used to indicate a directed graph in which each arrow is equipped with a sign.  Our graphs are a special case of this, as our construction entails that each signed arrow also has a partner of the opposite sign, which is not required by the general definition.}) graph with the same vertex set $C$, and with arrows defined as follows: for each arrow $\alpha \colon c \to d$ in $E$, we have two arrows $\alpha^+ \colon c \to d$ and $\alpha^- \colon c \to d$ in $E_{\pm}$.  

We denote by $\cE(E)$ the quotient of the path category of $E_{\pm}$ by the relations $\itbeta^- \circ \alpha^+ = \id_c$ and $\itbeta^+ \circ \alpha^- = \id_c$, whenever $\alpha$ and $\itbeta$ are arrows in $E$ with $s(\alpha) = t(\itbeta) = c$ and $s(\itbeta)=t(\alpha)$. Here, $s$ and $t$ denote the functions assigning to an arrow its source and target respectively.  We will call $\cE(E)$ the \emph{signed path category} over $E$.

The graph $E$ plays the role of the exchange graph\footnote{Sometimes it is more convenient to consider the exchange tree, rather than the exchange graph, as we will see later.} from cluster theory, whose vertices are clusters $c\in C$ and edges (one-step) mutations between them, except that we want to allow for situations where mutation is no longer an involution.  This happens, for example, in combinatorial models of mutation via triangulations of surfaces with infinitely many marked points.  So $E$ is directed and when we want to consider the usual exchange graph, when all mutations are involutive, we should think of the `directed double' of the undirected exchange graph (i.e. the graph where every undirected edge is replaced by a pair of oppositely oriented directed edges).

We will continue to write $c\in C$ even when we are thinking of the object $c$ in the signed path category $\cE=\cE(E)$; the role of the (signed path) category here is principally to allow us to speak of functors from $\cE$ since we cannot do the same with just the graph $E$. 

However, another advantage of using $\cE$ is that we may endow $\cE$ with the extra structure of a \emph{site}.  We will follow \cite{MacLaneMoerdijk} in our use of this and associated terminology.  Specifically, we can endow $\cE$ with the indiscrete (also called the trivial or chaotic) Grothendieck topology, by declaring that the unique sieve on $c\in \cE$ is the slice category $\cE/c$ (that is, the category with objects the morphisms in $\cE$ with codomain $c$).

Then a functor $F\colon \op{\cE}\to \Abcat$ is, by definition, a presheaf of Abelian groups on the site $\cE$ and moreover, every presheaf is a sheaf with respect to this topology.

Recall that free Abelian groups are characterized as being obtained from the functor \linebreak $\mathrm{Free}_{\integ}\colon \Setcat \to \Abcat$ that is left adjoint to the forgetful functor $\Abcat \to \Setcat$, with $\mathrm{Free}_{\integ}(S)=\integ[S]$ being the free Abelian group on the set $S$.

We say that a (pre)sheaf $\mathcal{F}\colon \op{\cE}\to \Abcat$ on the site $\cE$ (with the indiscrete topology) is a \emph{free Abelian} sheaf if there exists a (pre)sheaf of sets $\cB\colon \op{\cE}\to \Setcat$ such that $\mathcal{F}=\mathrm{Free}_{\integ}\circ \cB$.  One can show that $\cB \mapsto \mathrm{Free}_{\integ}\circ \cB$ is functorial: indeed, the construction of a free Abelian presheaf from a presheaf of sets is itself adjoint to the forgetful functor from Abelian presheaves to presheaves of sets.

Since we are working with only the indiscrete topology on $\cE$, so that saying ``sheaf'' is equivalent to saying ``functor with domain $\op{\cE}$'', and the same is true for $\op{\cE}$ \emph{mutatis mutandis}, we will also say ``sheaf'' for ``functor with domain $\op{(\op{\cE})}=\cE$'', rather than the more usual ``cosheaf''.

Now we may give the definition of an abstract cluster structure.

\begin{definition}\label{d:ACS} Let $E$ be a simple directed graph and let $\cE=\cE(E)$ be the signed path category over $E$.  An \emph{abstract cluster structure} over $E$ is a tuple $\cC=\genericACS$ where
	\begin{enumerate}[label=(\alph*)]
		\item $\cX\colon \op{\cE}\to \Abcat$ and $\cA\colon \cE\to \Abcat$ are free Abelian sheaves,
		\item\label{beta-factorization} $\beta\colon \cX\to \cA$ is a factorization, i.e.\ a collection of maps $\beta_c$, one for each $c \in \cE$, such that for every morphism $f \colon c \to d$ in $\cE$ the following diagram commutes:
		\[\begin{tikzcd}[column sep=25pt]
			\cX c \arrow[->]{r}[above]{\beta_{c}} & \cA c \arrow[->]{d}[right]{\cA f} \\
			\cX d  \arrow[->]{u}[left]{\cX f} \arrow[->]{r}[below]{\beta_{d}} & \cA d  
		\end{tikzcd}\]

		\item\label{form-adjoint} 
				$\ip{-}{-}\colon \cA \tensor_{\integ} \cX \to \sZ$ is a right non-degenerate pairing of $\cA$ and $\cX$.
	\end{enumerate}
	If $\cC$ is an abstract cluster structure over $E$, we say that $E$ is the \emph{exchange graph} of $\cC$.
\end{definition}

In $\cE$, we have two types of elementary morphism; an arbitrary morphism is a (finite) composition of these.  By construction, any directed edge $c\stackrel{e}{\to} d$ in $E$ gives rise to morphisms $c\stackrel{e^+}{\to} d$ and $c\stackrel{e^-}{\to} d$; let us call these \emph{elementary}.  

To avoid proliferation of superscripts, we will often not explicitly name edges but write $c\to d$ for an edge in $E$ and $\signededge{c}{d}$ as a shorthand for the pair of morphisms $c\stackrel{+}{\to}d$, $c\stackrel{-}{\to}d$ in $\cE$.  Consequently, associated to the pair of morphisms $\signededge{c}{d}$ we have maps $\cA+$, $\cA-$, $\cX+$ and $\cX-$.

Then the factorization condition of the definition means that for all $\signededge{c}{d}$, the following is a superposition of two commuting diagrams (one for each sign):
\[\begin{tikzcd}[column sep=25pt]
	\cX c \arrow[->]{r}[above]{\beta_{c}} & \cA c \arrow[->,xshift=-3pt]{d}[left]{\cA+} \arrow[->,xshift=3pt]{d}[right]{\cA-}   \\
	\cX d  \arrow[->,xshift=3pt]{u}[right]{\cX-} \arrow[->,xshift=-3pt]{u}[left]{\cX+} \arrow[->]{r}[below]{\beta_{d}} & \cA d  
\end{tikzcd}\]

Note that if $(c,d)$ is a signed digon, i.e.\ both $E(c,d)$ and $E(d,c)$ are non-empty, then $\cX$ and $\cA$ being functors (hence respecting the morphism relations in $\cE$) implies that $\cX\pm$ and $\cA\pm$ are invertible.  This corresponds to the situation where the mutation from $c$ to $d$ has an inverse mutation from $d$ back to $c$; in this case, the corresponding maps of Abelian groups are isomorphisms and we see that $\beta_{c}$ determines $\beta_{d}$ and vice versa.

\begin{remark}
	The forms $\ip{-}{-}_{c}$ will play a significant role in what follows.  For later use, we will record a number of equivalent expressions for the values of this form:
	\begin{align*}
		\ip{a}{x}_{c} & = \delta_{\cA c}(a)(x) \\
		& = \evform{\delta_{\cA c}(a)}{x} \\
		& = \delta_{\cX c}(x)(a) \\
		& = \evform{a}{\delta_{\cX c}(x)}
	\end{align*}
	
	We also note that if $\ip{-}{-}_{c}$ is a non-degenerate pairing of free Abelian groups of the same finite rank, then making the canonical identifications of $\cA c^{**}$ with $\cA c$ and $\cX c^{**}$ with $\cX c$, we have $\delta_{\cA c}^{*}=\delta_{\cX c}$ and $\delta_{\cX c}^{*}=\delta_{\cA c}$.  However as we want to give definitions that will work in the infinite rank case, we cannot use this freely.
\end{remark}

We claim that there are abstract cluster structures associated to many instances of mutation, but particularly the most familiar ones of cluster algebras, surface models for these and categorifications.  We will devote Part~\ref{P:reps} to showing this with all the requisite definitions and verifications, but to aid orientation in the general theory that follows in the rest of this section, we sketch very briefly how each of these families gives rise to abstract cluster structures.  Unfamiliar terms will be defined in the respective sections of Part~\ref{P:reps}.

\begin{example}[Cluster algebras,\ \S\ref{ss:AQCS-from-QCA}] {\ }
	
	In Section~\ref{ss:QCAs}, we give a definition of a quantum cluster algebra $\cC_{q}=\cC_{q}(\ex,\cB,\invfrozen,\beta,\lambda)$. Here, roughly, $\cB$ is an indexing set for the cluster variables of an initial cluster with $\ex$ the subset of indices where mutation is allowed. Then $\beta$ and $\lambda$ are linear maps such that with respect to the relevant bases we obtain the usual exchange matrix $B$ and quasi-commutation matrix $L$.  The datum $\invfrozen$ records which indices correspond to invertible frozen variables.
	
	The associated abstract cluster structure is given as follows.
	\begin{enumerate}
		\item The graph $E$ is the usual exchange tree, modified to be bi-directed (i.e.\ with a pair of oppositely oriented arrows between any two vertices) and with vertices indexed by tuples $\uk$ corresponding to mutation sequences from the initial cluster (i.e.\ root vertex);
		\item $\cX\colon \op{\cE}\to \Abcat$ is defined by $\cX \uk=\integ[\mu_{\uk}(\ex)]$ and $\cX {\pm}=\bar{\mu}_{k}^{\pm}$ (where ${\pm}\colon \uk\to (k,\uk)$) for $\bar{\mu}_{k}^{\pm}$ the isomorphisms of Lemma~\ref{l:F-isos}, which are closely related to Fomin--Zelevinsky's $F$-matrices;
				\item $\beta_{\uk}=\mu_{\uk}\beta$ is the linear map corresponding to the exchange matrix at each cluster;
				\item $\lambda_{\uk}=\mu_{\uk}\lambda$  is the linear map corresponding to the quasi-commutation matrix at each cluster;
				\item $\cA\colon \op{\cE}\to \Abcat$ defined by $\cA \uk=\dual{\integ[\mu_{\uk}(\cB)]}$ and $\cA {\pm}=\mu_{k}^{\pm}$ for $\mu_{k}^{\pm}$ the isomorphisms of Lemma~\ref{l:E-isos}, related to the $E$-matrices; and
				\item $\ip{\blank}{\blank}_{\uk}\colon \cA \uk \tensor \cX \uk \to \sZ$ given by $\ip{\dual{b}}{c}_{\uk}=\evform{\dual{b}}{c}$ (for $\dual{b}\in \dual{\mu_{\uk}(\cB)}$ and $c\in \mu_{\uk}(\ex)$) is just standard duality.
			\end{enumerate}		
\end{example}

\begin{example}[Triangulations of oriented surfaces,\ \S\ref{ss:triangulations}] {\ }
	
	The graph $E$ is the graph with vertices the triangulations of an oriented surface with marked points and a pair of oppositely oriented arrows between any two triangulations that differ by a single arc flip.
	
	Given a triangulation $\cT$ of the surface, $\cA \cT$ is the free Abelian group generated by the set of arcs (including boundary arcs).  The triangulation defines a collection of quadrilaterals, each of which has a unique interior arc as a diagonal, and $\cX \cT$ is the free Abelian group generated by these quadrilaterals.  The pairing between $\cA \cT$ and $\cX \cT$ is that which has each quadrilateral paired with the interior arc it contains.
	
	Mutation of arcs via diagonal flip in a quadrilateral gives rise to $\cA\pm$ and similarly mutation of quadrilaterals gives rise to $\cX\pm$.
	
	The component map $\beta_{\cT}$ is a boundary map: it takes a quadrilateral $q\in \cX \cT$ to the signed sum of its edges (an element of $\cA \cT$).
\end{example}

\begin{example}[Cluster categories, following \cite{CSFRT},\ \S\ref{ss:cluster-cats}] {\ }

	From a cluster category\footnote{As described in more detail later, we use the definition of cluster category from \cite{CSFRT}, which includes (but also extends) the cases of 2-Calabi--Yau triangulated or exact categories with cluster-tilting subcategories.} of finite rank $\cC$, we obtain an abstract cluster structure where
	\begin{enumerate}
		\item $E$ is the complete\footnote{It is shown in \cite{CSFRT} that the maps in \ref{ex:cl-cats-cindbar} and \ref{ex:cl-cats-cind} below are defined for any pair of cluster-tilting subcategories $\cT\!,\cU \ctsubcat \cC$. We comment on this further in \S\ref{ss:cluster-cats}.} bi-directed graph on the set of cluster-tilting subcategories of $\cC$;
		\item\label{ex:cl-cats-cindbar} $\cX \cT=\Kgp{\fd \stab{\cT}}$ (the Grothendieck group of finite-dimensional $\cT$-modules) and $\cX {+}=\stabcoindbar{\cU}{\cT}$, $\cX {-}=\stabindbar{\cU}{\cT}$ for ${\pm}\colon \cT \to \cU$, these maps being (restrictions of) adjoints to $\cA \pm$ below;
		\item $\beta_{\cT}=-p_{\cT}$ is (essentially) given by taking projective resolutions (or via \eqref{eq:p-vs-ind-coind});
		\item\label{ex:cl-cats-cind} $\cA \cT=\Kgp{\cT}$ (the Grothendieck group of $\cT$) and $\cA {+}=\ind{\cT}{\cU}$, $\cA {-}=\coind{\cT}{\cU}$ for ${\pm}\colon \cT \to \cU$, where these are the index and coindex maps associated to $\cT$-approximations; and
		\item $\ip{\blank}{\blank}$ has $\cT$-component given by $\ip{[T]}{[M]}_{\cT}=\dim_{\bK} M(T)$.
	\end{enumerate}
\end{example}

\subsection{Connectedness and ranks}\label{ss:connectedness}

The graph $E$ may have various levels of connectedness.  A directed graph is said to be \emph{weakly connected} if the underlying undirected graph associated to $E$ is connected and \emph{strongly connected} if there exists a directed path between any two vertices.  A directed graph is said to be \emph{complete} if for every pair of vertices $c$, $d$, $(c,d)$ is a digon, that is, there exist directed edges from $c$ to $d$ and from $d$ to $c$.  

We will introduce the term \emph{bi-directed} to mean that for every pair of vertices $c,d$, either $E(c,d)=\emptyset$ or $E(c,d)$ is a digon.  Clearly, complete implies strongly connected implies weakly connected.  Also, complete implies bi-directed. We will refer to the connected components with respect to weak (respectively strong) connectedness as weakly (resp.\ strongly) connected components.  Note that if $E$ is bi-directed then it is weakly connected if and only if it is strongly connected, so for bi-directed graphs we will simply say ``connected''.

Furthermore, we will say that a directed graph $E$ is \emph{rootable} if there exists $r\in C$ such that for all $c\in C$, there exists a directed path from $r$ to $c$.  If $E$ is rootable and some such $r$ is chosen, we will say that $E$ is \emph{rooted} with \emph{root} $r$ and that $(E,r)$ is a rooted graph.  Similarly we will say that $E$ is \emph{corootable} if there exists $r^{\vee}\in C$ such that for all $c\in C$, there exists a directed path from $c$ to $r^{\vee}$, and in this situation say that $E$ (or $(E,r^{\vee})$) is \emph{corooted} with \emph{coroot} $r^{\vee}$.

Note that complete implies rootable and corootable, and either implies weakly connected, but (co)rootable does not imply strongly connected.

We will apply these terms to abstract cluster structures: an abstract cluster structure $\cC$ will be said to be weakly connected (respectively strongly connected, complete, (co)rootable, (co)rooted) if its exchange graph $E$ is.

\begin{definition} Let $E$ be a simple directed graph and $\cC=\genericACS$ an abstract cluster structure on $E$. We define the following:
	\begin{enumerate}[label=\textup{(\alph*)}]
		\item the \emph{$\cA$-rank} of $\cC$, $\Ark{\cC}=\max\{ \rank{\cA c} \mid c\in C\}$, being the maximum of the ranks (as free Abelian groups) of the $\cA c$ for any $c\in C$;
		\item the \emph{$\cX$-rank} of $\cC$, $\Xrk{\cC}=\max\{ \rank{\cX c} \mid c\in C\}$, being the maximum of the ranks (as free Abelian groups) of the $\sX c$ for any $c\in C$.
	\end{enumerate}
	
	We say that $\cC$ has \emph{finite rank} if $\Ark{\cC}\in \nat$ and $\Xrk{\cC}\in \nat$, and that $\cC$ has \emph{infinite rank} otherwise.
\end{definition}

\begin{definition} Let $E$ be a simple directed graph and $\cC=\genericACS$ an abstract cluster structure on $E$.  
	\begin{enumerate}[label=(\alph*)] 
		\item We say that $\cC$ has \emph{weakly} (respectively, \emph{strongly}) \emph{locally constant} rank if the functions $c\mapsto \rank{\cA c}$ and $c\mapsto \rank{\cX c}$ are constant on the weakly (resp.\ strongly) connected components of $E$.
		\item If for all $c\in C$ we have $\rank{\cA c}=\Ark{\cC}$ and $\rank{\cX c}=\Xrk{\cC}$, we say that $\cC$ has \emph{constant rank}.  In this situation we call 
		\begin{enumerate}[label=(\roman*)]
			\item $\Ark{\cC}$ the \emph{total rank} of $\cC$, and write $\trk{\cC}$;
			\item $\Xrk{\cC}$ the \emph{mutable rank} of $\cC$, and write $\mrk{\cC}$, and
			\item $\frk{\cC}\defeq \trk{\cC}-\mrk{\cC}$ the \emph{frozen rank} of $\cC$.
		\end{enumerate}
	\end{enumerate}
\end{definition}

\begin{remark} When $\cC$ is finite rank, the frozen rank of $\cC$ is non-negative because $\ip{\blank}{\blank}$ is required to be right non-degenerate.  Imposing this condition clearly breaks some symmetry between the $\cA$ and $\cX$ sides but we believe it is reasonable as a result of the examples. 
\end{remark}

With these properties, some standard features of cluster theory start to fall into place.

\begin{lemma}\label{l:constant-rank} Let $E$ be a connected simple bi-directed graph and $\cC$ an abstract cluster structure over $E$.  Then $\cC$ has constant rank.
\end{lemma}

\begin{proof} Since $E$ is bi-directed, every pair $(c,d)$ of vertices such that $E(c,d)\neq \emptyset$ forms a signed digon and so each of the associated morphisms $\cX+$, $\cX-$, $\cA+$ and $\cA-$ is an isomorphism. So $\cC$ has locally constant rank and hence (since $E$ is connected) constant rank.
\end{proof}

Note too that in this situation, the ranks as $\integ$-linear maps of the $\beta_{c}$ are all equal; when this is defined, we will call this value $\rank \beta$.

\begin{lemma} Let $E$ be a connected simple bi-directed graph and $\cC$ an abstract cluster structure over $E$.  Then there exists an Abelian group $G(\cC)$ such that for all $c\in C$, $\Coker {\beta_{c}}\iso G(\cC)$.  We call $G(\cC)$ the \emph{fundamental group} of $\cC$.
\end{lemma}

\begin{proof} Similarly to the previous lemma, $E$ being connected bi-directed implies that $\Coker {\beta_{c}}\iso \Coker {\beta_{d}}$ for all $c$, $d$.
\end{proof}

\subsection{The principal part}\label{ss:principal-part}

Let $\cC=\genericACS$ be an abstract cluster structure over a simple directed graph $E$.  We show how the pairing $\ip{\blank}{\blank}$ of $\cA$ and $\cX$ enables us to construct the \emph{principal part} of the abstract cluster structure, analogous to the usual notion for an exchange matrix.

From the pairing $\ip{\blank}{\blank}$, we have the associated \emph{left radical} functor $\Ker \delta_{\cA}\colon \cE \to \Abcat$, with $(\Ker \delta_{\cA})c=\Ker \delta_{\cA c}$.  This is a saturated subfunctor of $\cA$ by Lemma~\ref{l:Ker-delta-saturated}.

By definition, 
\[ \Ker \delta_{\cA c}= \{ a\in \cA c \mid \ip{a}{\cX c}_{c}=0 \} \]
so we will write $\cX c^{\perp}\defeq \Ker \delta_{\cA c}$ for this subgroup of $\cA c$.

Now define $\cA^{p} c=\cA c/\cX c^{\perp}$ and let $\pi_{c}^{p}\colon \cA c\to \cA c/\cX c^{\perp}=\cA^{p} c$ be the canonical surjection.  By construction, the induced form $\ip{-}{-}_{c}^{p}\colon \cA^{p} c \cross \cX c \to \integ$, $\ip{a+\cX c^{\perp}}{x}_{c}^{p}=\ip{a}{x}_{c}$ has trivial left radical.  Since $\ip{\blank}{\blank}$ is already right non-degenerate, we see that $\ip{\blank}{\blank}^{p}_{c}$ is non-degenerate.

Denote by $\delta_{\cX^{p} c}\colon \cX c \to (\cA^{p}c)^{*}$ the natural map associated to $\ip{-}{-}^{p}_{c}$, so that we distinguish this from $\delta_{\cX}$ formed with respect to $\ip{-}{-}_{c}$.  

\begin{remark}\label{rmk:equalities-for-principal-part-form} As before, we may infer a number of equivalent expressions for the value of the form $\ip{-}{-}_{c}^{p}$:
	\begin{alignat*}{5}
		&\ip{a+\cX c^{\perp}}{x}_{c}^{p} && = \ip{\pi_{c}^{p}(a)}{x}_{c}^{p} && = \ip{a}{x}_{c} && = \evform{\delta_{\cA c}(a)}{x} && = \evform{a}{\delta_{\cX c}(x)}\\
		& && = \delta_{\cA^{p} c}(\pi_{c}^{p}(a))(x) && && =\evform{\delta_{\cA^{p} c}(\pi_{c}^{p}(a))}{x} && = \delta_{\cA c}(a)(x)\\
		& && = \delta_{\cX c}(x)(\pi_{c}^{p}(a)) && && = \evform{\pi_{c}^{p}(a)}{\delta_{\cX c}(x)} && = \delta_{\cX c}(x)(a)
	\end{alignat*}
	
	Note in particular, as it will be useful later, that for all $a\in \cA c$ and $x\in \cX$ we have
	\begin{align*} ((\pi_{c}^{p})^{*}\circ \delta_{\cX^{p}c})(x)(a) & = \ip{\pi_{c}^{p}(a)}{x}_{c}^{p} \\
		& = \ip{a+\cX c^{\perp}}{x}_{c}^{p} \\
		& = \ip{a}{x}_{c} \\
		& = \delta_{\cX c}(x)(a) 
	\end{align*}
	so that 
	\begin{equation}\label{eq:pi-star-equals-iota} (\pi_{c}^{p})^{*}\circ \delta_{\cX^{p}c}=\delta_{\cX} \end{equation}
	and similarly
	\begin{equation}\label{eq:iota-star-equals-pi} \delta_{\cA c}=\delta_{\cA^{p} c}\circ \pi_{c}^{p}. \end{equation}
\end{remark}

Our next goal is to show that the assignment $c\mapsto \cA^{p} c$ can be enhanced to a functor.

\begin{lemma}\label{l:wmd-props} Let $\cC=\genericACS$ be an abstract cluster structure over a simple directed graph $E$. Then there are well-defined homomorphisms $\overline{\cA +}$ and $\overline{\cA -}$ defined as follows:
	\begin{align*} \overline{\cA +} & \colon \cA c/\cX c^{\perp}\to \cA d/\cX d^{\perp},\ \overline{\cA +}(a+\cX c^{\perp})=(\cA +)(a)+\cX d^{\perp}; \\
		\overline{\cA -} & \colon \cA c/\cX c^{\perp}\to \cA d/\cX d^{\perp},\ \overline{\cA -}(a+\cX c^{\perp})=(\cA -)(a)+\cX d^{\perp}. \\
	\end{align*}
\end{lemma}

\begin{proof} Let $a'=(\cA \pm)(a)\in (\cA \pm)(\cX c^{\perp})$ for some $a\in \cX c^{\perp}$. Then for $x'\in \cX d$ we have that by adjointness of $\cA\pm$ and $\cX\pm$ (Lemma~\ref{l:adjoint}), \[ \ip{a'}{x'}_{d}=\ip{(\cA \pm)(a)}{x'}_{d}=\ip{a}{(\cX \pm)(x')}_{c}=0 \]
	since $a\in \cX c^{\perp}$; hence $a'\in \cX d^{\perp}$. Thus $(\cA \pm)(\cX c^{\perp})\leq \cX d^{\perp}$ and the maps are well-defined.
\end{proof}

This lemma allows us to define the \emph{principal part} of our abstract cluster structure and see that this is itself an abstract cluster structure.

\begin{definition} Let $\cC=\genericACS$ be an abstract cluster structure over a simple directed graph $E$. Define $\cA^{p}\colon \cE \to \Abcat$ by $\cA^{p}c \defeq \cA^{p} c=\cA c/\cX c^{\perp}$ and $\cA^{p}\pm=\overline{\cA \pm}$.
\end{definition}

Since $\cA$ is functorial, so is $\cA^{p}$.  Moreover, as noted above, $\Ker \delta_{\cA}$ is a saturated subfunctor of $\cA$, and hence $\cA^{p}$ is again a free Abelian sheaf, since saturation implies that $\cA^{p}c=\cA c/\Ker \delta_{\cA c}$ is torsion-free for all $c$.

It is also immediate from the definitions that the family $\{ \pi_{c}^{p} \}$ defines a natural transformation $\pi^{p}\colon \cA \to \cA^{p}$.  Define $\beta^{p}=\pi^{p} \circ \beta$.

\begin{proposition}\label{p:prin-part-ACS} Let $\cC=\genericACS$ be an abstract cluster structure over a simple directed graph $E$. Then the tuple $\cC^{p}\defeq (\cE,\cX,\beta^{p},\cA^{p},\ip{-}{-}^{p})$ is an abstract cluster structure over $E$.
\end{proposition}

\begin{proof} That $\cC^{p}$ is an abstract cluster structure now follows immediately from our definitions, notably of $\beta^{p}$ as a composition of the natural transformations $\beta$ and $\pi^{p}$.
\end{proof}

We will refer to $\cC^{p}$ as the \emph{principal part} of $\cC$.

Note that if $\cC$ has constant finite rank, then $\frk{\cC^{p}}=0$; this is the abstract version of the principal part of the exchange matrix corresponding to only the mutable rows and columns.  Similarly, that the principal part of the exchange matrix is obtained by deleting the frozen part is encapsulated by the following lemma.

\begin{lemma}\label{l:prin-beta-ip} We have
	\[ \ip{\beta_{c}^{p}(x)}{y}_{c}^{p}=\ip{\beta_{c}(x)}{y}_{c} \]
	for all $x,y\in \cX c$.
\end{lemma}
	
\begin{proof}
	This is immediate from the equations in Remark~\ref{rmk:equalities-for-principal-part-form}.
\end{proof}
		
	\subsection{Forms and skew-symmetry}\label{ss:forms-skew-sym}
	
	As hinted at by the previous lemma, there is a second natural family of forms associated to an abstract cluster structure.  Unlike the above, these are not in general non-degenerate.
	
	\begin{definition}\label{d:sX-form} Let $\cC=\genericACS$ be an abstract cluster structure over a simple directed graph $E$.  For $c\in C$, define $\ip{-}{-}_{\cX c}\colon \cX c \cross \cX c \to \integ$ by $\ip{x}{y}_{\cX c}=\ip{\beta_{c}(x)}{y}_{c}$.
	\end{definition}
	
	We will refer to the form $\ip{\blank}{\blank}_{\cX c}$ as the \emph{$\cX$-form} at $c$.
	
	By Remark~\ref{rmk:equalities-for-principal-part-form}, we have
	\begin{equation}\label{eq:expressions-for-sX-form} \ip{x}{y}_{\cX c}=\ip{\beta_{c}(x)}{y}_{c}=\evform{(\delta_{\cA c}\circ \beta_{c})(x)}{y}=\evform{\beta_{c}(x)}{\delta_{\cX c}(y)}.
	\end{equation}
	
	\begin{definition}\label{d:skew-symmetrizable} We say that $\cC$ is \emph{skew-symmetrizable} at $c\in C$ if $\ip{-}{-}_{\cX c}$ is skew-symmetric, i.e.\  
		\[ \ip{x}{y}_{\cX c}=-\ip{y}{x}_{\cX c} .\]
		Furthermore, we say that $\cC$ is skew-symmetrizable if it is skew-symmetrizable at every $c$.
	\end{definition}
	
	Note that since, by Lemma~\ref{l:prin-beta-ip}, we have $\ip{\beta_{c}^{p}(x)}{y}_{c}^{p}=\ip{\beta_{c}(x)}{y}_{c}$, $\cC$ and its principal part $\cC^{p}$ define the same form on $\cX c$.  Therefore, the principal part $\cC^{p}$ is skew-symmetrizable if and only if $\cC$ is.  As such, we will simply say ``is skew-symmetrizable'' rather than ``has skew-symmetrizable principal part''.
	
	The apparent disparity between the term ``skew-symmetrizable'' in the terminology and the requirement that the form $\ip{-}{-}_{\cX c}$ be skew-symmetric is explained by the following definition and proposition.  Recall that $\delta_{A}^{\text{ev}}\colon A\to A^{**}$, $a\mapsto (f\mapsto f(a))$ is the canonical evaluation morphism obtained from the evaluation form.
	
	\begin{definition}\label{d:skew-symmetric}
		Let $A$ be a free Abelian group and $\phi\colon A \to \dual{A}$ a homomorphism. We say $\phi$ is skew-symmetric if $\phi^{*}\circ \delta_{A}^{\text{ev}}=-\phi$.
	\end{definition}
	
	Since $\delta_{\dual{A}}^{\text{ev}}=\id_{\dual{A}}$, the condition of the definition is equivalent to $\delta_{\dual{A}}^{\text{ev}}\circ \phi^{*}\circ \delta_{A}^{\text{ev}}=-\phi$.  But $\delta_{\dual{A}}^{\text{ev}}\circ \phi^{*}\circ \delta_{A}^{\text{ev}}$ is exactly the adjoint of $\phi$ with respect to the evaluation form, which we denote $\adj{\phi}$, so that the condition becomes the more familiar $\adj{\phi}=-\phi$ (at the cost of suppressing in the notation which form is being used to take the adjoint).
	
	\begin{remark}
		We can compute the adjoint as follows: for $a,b\in A$,
			\begin{align*} 
			\adj{\phi}(a)(b) & = (\delta_{\dual{A}}^{\text{ev}} \circ \dual{\phi} \circ \delta_{A}^{\text{ev}})(a)(b) \\
			& = (\delta_{\dual{A}}^{\text{ev}} \circ \dual{\phi})(\evform{a}{\blank})(b) \\
			& = \delta_{\dual{A}}^{\text{ev}}(\evform{a}{\phi(\blank)})(b) \\
			& = \evform{b}{\evform{a}{\phi(\blank)}} \\
			&  = \evform{a}{\phi(b)} \\
			& = \phi(b)(a)
		\end{align*}
		and so skew-symmetry (with respect to the evaluation form) is equivalent to $\phi(a)(b)=-\phi(b)(a)$.
	\end{remark}
	
	\begin{proposition}\label{p:ACS-skew-symmetrizable} The following are equivalent:
		\begin{enumerate}
			\item The abstract cluster structure $\cC$ is skew-symmetrizable at $c$.
			\item The principal part of $\cC$, $\cC^{p}$, is skew-symmetrizable at $c$.
			\item The map $\delta_{\cA c}\circ \beta_{c}$ is skew-symmetric.
			\item The map $\delta_{\cA^{p} c}\circ \beta_{c}^{p}$ is skew-symmetric.
		\end{enumerate}
	\end{proposition}
	
	\begin{proof}
		We first show that
		\[ \ip{x}{y}_{\cX c}=-\ip{y}{x}_{\cX c} \iff (\delta_{\cA^{p} c}\circ \beta^{p}_{c})^{*}\circ \delta_{\cA^{p} c}^{\text{ev}}=-(\delta_{\cA^{p} c}\circ \beta^{p}_{c}). 
		\]
		Computing, we have 
		\begin{align*} 
			((\delta_{\cA^{p} c}\circ \beta^{p}_{c})^{*}\circ \delta_{\cA^{p} c}^{\text{ev}})(x)(y) & = \delta_{\cA^{p} c}^{\text{ev}}(x)((\delta_{\cA^{p} c} \circ \beta_{c}^{p})(y)) \\
			& = \delta_{\cA^{p} c}^{\text{ev}}(x)(\ip{\beta_{c}^{p}(y)}{-}_{c}^{p}) \\
			& = \delta_{\cA^{p} c}^{\text{ev}}(x)(\ip{y}{-}_{\cX c}) \\
			& = \ip{y}{x}_{\cX c}
		\end{align*}
		and
		\[ (\delta_{\cA^{p} c}\circ \beta_{c}^{p})(x)(y)=\ip{\beta^{p}_{c}(x)}{y}_{c}^{p}=\ip{x}{y}_{\cX c}, \]
		from which the claim follows immediately. Now, we see that the same argument without the superscript ``$p$'' is also valid, which together with the comments immediately after Definition~\ref{d:skew-symmetrizable} gives the equivalence of the four claims.
	\end{proof}
	
	From this result, we see that $\delta_{\cA c}$ (or equivalently $\delta_{\cA^{p} c}$) plays the role of the skew-symmetrizer matrix.  It is somewhat hidden in Definition~\ref{d:sX-form}, until one looks at the alternative expressions in \eqref{eq:expressions-for-sX-form}, where one sees the corresponding $\delta_{\cA c}\circ \beta$.
	
	\begin{proposition}\label{p:X-form-mut} Let $\cC=\genericACS$ be an abstract cluster structure over a simple directed graph $E$.  Then for $c \to d$, the forms $\ip{-}{-}_{\cX c}$ and $\ip{-}{-}_{\cX d}$ are related by
		\[ \ip{-}{-}_{\cX d}=\ip{-}{-}_{\cX c}\circ (\cX{+} \cross \cX+)=\ip{-}{-}_{\cX c}\circ (\cX{-} \cross \cX-). \]
	\end{proposition}
	
	\begin{proof} Consider $x,y\in \cX d$.  Then
		\begin{align*} \ip{x}{y}_{\cX d} & = \ip{\beta_{d}(x)}{y}_{d} \\
			& = \ip{(\cA \pm \circ \beta_{c} \circ \cX \pm)(x)}{y}_{d} \\
			& = \ip{\beta_{c}((\cX \pm)(x))}{(\cX \pm)(y)}_{c} \\
			& = \ip{(\cX \pm)(x)}{(\cX \pm)(y)}_{\cX c}. \qedhere
		\end{align*}
	\end{proof}

	The interpretation of this proposition is that by requiring the factorization property for the part of the definition of an abstract cluster structure relating to $\beta$, the $\cX$-forms are related by mutation (controlled by the edges of $E$) in a consistent way.
	
	Since $\cX\pm$ are linear, the following is immediate.
	
	\begin{corollary} For $c \to d$, if
		$\ip{-}{-}_{\cX c}$ is skew-symmetric then $\ip{-}{-}_{\cX d}$ is. \qed
	\end{corollary}
	
	\begin{remark} Let $c\to d \to e$ be a path of length two in $E$. Then by functoriality and the previous proposition,
		\[ \ip{-}{-}_{\cX e}=\ip{-}{-}_{\cX c}\circ (\cX(+\circ+) \cross \cX(+\circ +))=\ip{-}{-}_{\cX c}\circ (\cX(-\circ-) \cross \cX(-\circ -)). \]
		Indeed, this extends to paths of arbitrary finite length, i.e.\ to all morphisms in $\cE$ that are sign-coherent.
		
		When the underlying graph $E$ is bi-directed, as in the case of the usual exchange graph, it follows that the value of the $\cX$-form at one $c\in C$ determines all the others and skew-symmetry propagates from one $c$ to all.  Equivalently, one may observe that this true of $\beta$ in the bi-directed setting too.
	\end{remark}
	
	The family of forms $\ip{\blank}{\blank}_{\cX c}$, $c\in C$ fit together as follows.  Consider the natural transformation $\delta_{\cA}\circ \beta\colon \cX \to \dual{\cX}$. This has components $(\delta_{\cA}\circ \beta)_{c}=\delta_{\cA c}\circ \beta_{c}$ for $c\in C$ and as in the proof of Proposition~\ref{p:ACS-skew-symmetrizable}, $(\delta_{\cA c}\circ \beta_{c})(x)(y)=\ip{x}{y}_{\cX c}$.
	
	We may rewrite this as follows.  Let $\cX \tensor_{\integ} \cX\colon \op{\cE} \to \Abcat$ be the functor given by \linebreak $(\cX \tensor_{\integ} \cX)c=\cX c \tensor \cX c$ on objects and $(\cX \tensor_{\integ} \cX)f=f\tensor f$ on morphisms.  Then we may define $\ip{\blank}{\blank}_{\cX}\colon \cX \tensor_{\integ} \cX \to \sZ$ to have components $(\ip{\blank}{\blank}_{\cX})_{c}=\ip{\blank}{\blank}_{\cX c}$.  By Proposition~\ref{p:X-form-mut}, we see that we have $\ip{-}{-}_{\cX d}=\ip{\cX\pm(-)}{\cX\pm(-)}_{\cX c}$.
	
	However, this is not a pairing of $\cX$ with itself, in the sense of Definition~\ref{d:pairing}, for the definition of a pairing takes as input one covariant and one contravariant functor.  Rather, we have a variation on this, with a variance change in one position, akin to the difference between a natural transformation and a factorization. Nevertheless, we see that either via $\ip{\blank}{\blank}_{\cX}$ or $\delta_{\cA} \circ \beta$, we have ``naturality'' of the forms $\ip{\blank}{\blank}_{\cX c}$.
	
	We remarked at the start of this section that the forms $\ip{\blank}{\blank}_{\cX c}$ are typically not non-degenerate.  Indeed, we can see from the above discussion that their non-degeneracy determines and is determined by the corresponding property of $\beta_{c}$.
	
	\begin{proposition}\label{p:inj-iff-left-non-deg-X-form} Assume that $\cC$ is skew-symmetrizable at $c$. The following are equivalent:
		\begin{enumerate}
			\item\label{p:inj-iff-left-non-deg-X-form-inj} The map $\beta_{c}$ is injective.
			\item\label{p:inj-iff-left-non-deg-X-form-left-nd} The form $\ip{\blank}{\blank}_{\cX c}$ is left non-degenerate.
		\end{enumerate}
	\end{proposition}
	
	\begin{proof}
		
		Assume $\beta_{c}$ is injective and consider $\epsilon_{L}\colon \cX c \to \dual{\cX c}$, $\epsilon_{L}(x)=\ip{x}{\blank}_{\cX c}$. Then $\Ker \epsilon_{L}$ is the left radical of $\ip{\blank}{\blank}_{\cX c}$; we use $\epsilon$ rather than $\delta$ to avoid confusion with the maps induced by the form $\ip{\blank}{\blank}$ and use ``$L$'' to distinguish from the other induced map.
		
		Let $x\in \Ker \epsilon_{L}$.  Then for all $y\in \cX$, $\epsilon_{L}(x)(y)=0$, that is $\ip{x}{y}_{\cX c}=\ip{\beta_{c}(x)}{y}_{c}=0$.  Since $\cC$ is skew-symmetrizable at $c$, we deduce that $\ip{y}{\beta_{c}(x)}_{c}=0$ for all $y$. But $\ip{\blank}{\blank}_{c}$ is right non-degenerate, so $\beta_{c}(x)=0$.  Since $\beta_{c}$ is injective, then $x=0$ and $\Ker \epsilon_{L}=0$, so that $\ip{\blank}{\blank}_{\cX c}$ is left non-degenerate.
		
		Conversely, assume that $\ip{\blank}{\blank}_{\cX c}$ is left non-degenerate and let $x,x'\in \cX c$ be such that $\beta_{c}(x)=\beta_{c}(x')$.  Then for all $y\in \cX$, $\ip{x}{y}_{\cX c}=\ip{x'}{y}_{\cX c}$, so $\ip{x-x'}{y}_{\cX c}=0$ for all $y$ and so left non-degeneracy implies $x=x'$ and $\beta_{c}$ injective.
		
	\end{proof}
	
	\begin{proposition}\label{p:surj-iff-right-non-deg-X-form} Let $\cC$ be a finite rank abstract cluster structure.  The following are equivalent:
		\begin{enumerate}
			\item\label{p:surj-iff-right-non-deg-X-form-surj} The map $\beta_{c}$ is surjective.
			\item\label{p:surj-iff-right-non-deg-X-form-left-nd} The form $\ip{\blank}{\blank}_{\cX c}$ is right non-degenerate.
		\end{enumerate}
	\end{proposition}
	
	\begin{proof}
		
		Assume that $\beta_{c}$ is surjective and let $\epsilon_{R}\colon \cX c\to \dual{\cX c}$, $\epsilon_{R}(y)=\ip{\blank}{y}_{\cX c}$.  Then $\Ker \epsilon_{R}$ is the right radical of $\ip{\blank}{\blank}_{\cX c}$.
		
		Let $y\in \Ker \epsilon_{R}$. Then for all $x\in \cX c$, $\epsilon_{R}(y)(x)=0$, that is $\ip{x}{y}_{\cX c}=\ip{\beta_{c}(x)}{y}_{c}=0$. But $\beta_{c}$ is surjective, so $\ip{\beta_{c}(x)}{y}_{c}=0$ for all $x\in \cX c$ if and only if $\ip{a}{y}_{c}=0$ for all $a\in \cA c$. Since $\ip{\blank}{\blank}_{c}$ is right non-degenerate, this implies that $y=0$ and so $\ip{\blank}{\blank}_{\cX c}$ is right non-degenerate.
		
		Conversely, we have
		\[ \Ker \epsilon_{R} = \{ y\in \cX c \mid \ip{\beta_{c}(x)}{y}_{c}=0\ \text{for all}\ x\in \cX \}=(\Image \beta_{c})^{\perp_{c}} \] 
		where we use ``$\perp_{c}$'' to indicate that this is with respect to the form $\ip{\blank}{\blank}_{c}$.	Since $\ip{\blank}{\blank}_{\cX c}$ is right non-degenerate, $(\Image \beta_{c})^{\perp_{c}}=\Ker \epsilon_{R}=0$.  Then since $\cC$ has finite rank, we deduce that $\Image \beta_{c}=0^{\perp_{c}}=\cA_{c}$.
				
	\end{proof}
	
	\begin{corollary}
		Let $\cC$ be a finite rank abstract cluster structure. Then the map $\beta_{c}$ is an isomorphism if and only if $\ip{\blank}{\blank}_{\cX c}$ is non-degenerate. \qed
	\end{corollary}
	
	The conditions that $\beta_{c}$ be injective or surjective correspond to statements about the exchange matrix being ``full rank''.  In the most general situation, when wanting to consider cluster algebras of infinite rank (i.e.\ when our clusters are not necessarily finite), we should be careful to use the correct condition, when e.g.\ rank and nullity comparisons are no longer appropriate. In particular, we see from the above that it is injectivity of $\beta_{c}$ that behaves uniformly between finite and infinite rank, rather than being ``full rank'', i.e. surjective.
	
	\subsection{Quantum structures}\label{ss:quantum-str}
	
	As we will see shortly from the construction, quantum cluster data is dual to the exchange data encoded in $\beta$.  Not every abstract cluster structure admits a quantum structure, although many important examples do.
	
	More precisely, the definition of an abstract cluster structure has an asymmetry: the factorization $\beta$ has domain $\cX$ and codomain $\cA$.  To restore this symmetry, we try to choose a retraction (as a factorization) of $\beta$ , $\rho\colon \cA\to \cX$.  That is, $\rho \circ \beta = \id_{\cX}$.
	
	\begin{remark}
	Asking that $\rho$ is a retraction of $\beta$, as opposed to being an arbitrary factorization with domain $\cA$ and codomain $\cX$, has two aspects that we comment on briefly.  First, it means that we are not fully restoring symmetry: we could try to choose a section rather than a retraction.  But, for reasons that are unclear to us, seeking a retraction is the right thing to do in many examples.
	
	Secondly, choosing some condition---rather than none---is sensible, as we want to have $\beta$ and $\rho$ interact in some way, or else all we have is a system of two abstract cluster structures with no relationship to each other beyond being built on the same underlying data.  The retraction condition is what we choose as our notion of \emph{compatibility}, and we will see several ways to express this.
	\end{remark}
	
	Since we are seeking a retraction, there are conditions that must be fulfilled for this to exist.
	
	\begin{lemma}\label{l:retract-iff-mono} Let $\cC=\genericACS$ be an abstract cluster structure.  The following are equivalent:
		\begin{enumerate}
			\item the factorization $\beta$ admits a retraction $\rho$;
			\item $\beta$ is a split monomorphism; and
			\item every component $\beta_{c}$ of $\beta$ is injective.
		\end{enumerate}
	\end{lemma}
	
	\begin{proof} The first two equivalences are standard and the third is equivalent to those because $\cX$ and $\cA$ are free Abelian sheaves.
	\end{proof}
	
	By Proposition~\ref{p:inj-iff-left-non-deg-X-form}, we could also express this by saying that the forms $\canform{\blank}{\blank}{\cX c}$ are left non-degenerate for all $c$.  As per the discussion around that Proposition, this expresses the usual notion of ``full rank''.
	
	Now, to define a quantum structure in a way that aligns with the usual approach, we should not use $\rho$ directly.  Recall that an abstract cluster structure comes with a right non-degenerate pairing $\canform{\blank}{\blank}{}\colon \cA\tensor_{\integ} \cX \to \sZ$ and that there is an associated natural transformation $\delta_{\cX}\colon \cX \to \dual{\cA}$, which is itself a monomorphism.
	
	Then it is $\lambda=\delta_{\cX}\circ \rho$ that we will put in the data of an abstract quantum cluster structure.
	
	\begin{definition} Let $\cC=\genericACS$ be an abstract cluster structure over a simple directed graph $E$.
		
		A \emph{quantum structure} for $\cC$ is a factorization $\lambda\colon \cA \to \dual{\cA}$ such that $\lambda=\delta_{\cX}\circ \rho$ for some retraction $\rho$ of $\beta$.
		If $\cC$ admits a quantum structure, we will say that the extended tuple $\genericAQCS$ is an \emph{abstract quantum cluster structure}.
	\end{definition}
	
	We will usually call abstract quantum cluster structures $\cC_{q}$, to highlight the additional data.	By Lemma~\ref{l:retract-iff-mono}, $\cC$ admits a quantum structure if and only if $\beta$ is a split monomorphism.
	
	Recalling too that we have a dinatural transformation $\evform{\blank}{\blank}\colon \dual{\cA}\tensor_{\integ} \cA \to \sZ$ and associated natural transformations $\delta_{\cA}^{\text{ev}}$ and $\delta_{\dual{\cA}}^{\text{ev}}=\id_{\dual{\cA}}$, we define the \emph{adjoint} $\adj{\lambda}$ of $\lambda$ to be the factorization $\delta_{\dual{\cA}}^{\text{ev}}\circ \dual{\lambda} \circ \delta_{\cA}^{\text{ev}}$.
	
	\begin{lemma}\label{l:compat-equiv-conds}
		The following are equivalent:
		\begin{enumerate}
			\item\label{l:compat-equiv-conds-q-str} $\lambda=\delta_{\cX}\circ \rho$ is a quantum structure for $\cC$;
			\item\label{l:compat-equiv-conds-retract} $\rho$ is a retraction of $\beta$, i.e.\ $\rho \circ \beta=\id_{\cX}$;
			\item\label{l:compat-equiv-conds-forms} $\evform{\blank}{\blank}\circ (\lambda \tensor_{\integ} \beta) =\canform{\blank}{\blank}{}$;
			\item\label{l:compat-equiv-conds-LTB} $\adj{\lambda}\circ \beta=\delta_{\cX}$; and
			\item\label{l:compat-equiv-conds-BTL} $\adj{\beta} \circ \lambda = \delta_{\cA}$.
		\end{enumerate}
	\end{lemma}
	
	\begin{proof} We have \ref{l:compat-equiv-conds-q-str} if and only if \ref{l:compat-equiv-conds-retract} by definition. Since $\delta_{\cX}\circ \rho\circ \beta=\delta_{\cX}$, these are equivalent to \ref{l:compat-equiv-conds-forms}.	Rewriting the left-hand side of  $\evform{\lambda_{c}(\blank)}{\beta_{c}(\blank)}=\canform{\blank}{\blank}{c}$ as $\evform{\blank}{\adj{\lambda_{c}}\circ \beta_{c}(\blank)}$, we obtain the equivalency with \ref{l:compat-equiv-conds-LTB}. Finally, taking adjoints of both sides of the latter, we obtain \ref{l:compat-equiv-conds-BTL}.		
	\end{proof}
	
	We will refer to any of these equivalent conditions as the \emph{compatibility} of $\beta$ and $\lambda$ in an abstract quantum cluster structure.
	
	\begin{remark} As we progress through the list of equivalent conditions of the lemma, we come closer to expressions that will be familiar to those who have seen the original definition of a quantum cluster algebra.  Namely, for an exchange matrix $B$ and quasi-commutation matrix $L$, the compatibility condition is expressed by saying that $B^{T}L$ should (up to reordering columns) have the form of a block matrix consisting of a diagonal block with positive integer entries and a zero block.  Furthermore, the positive integers appearing are closely related to the skew-symmetrizer of $B$.  
		
	Now $\adj{\ }$ corresponds to taking the matrix transpose and, by Proposition~\ref{p:ACS-skew-symmetrizable}, $\delta_{\cA c}\circ \beta_{c}$ is skew-symmetric, i.e.\ $\delta_{\cA c}$ is a skew-symmetrizer for $\beta_{c}$.  So condition~\ref{l:compat-equiv-conds-BTL} matches the usual compatibility condition.
	\end{remark}

	A further (seeming) departure from the standard approach is that we have not required $\lambda$ to be skew-symmetric, which is important for the construction of quantum cluster algebras.  Specifically, the skew-symmetry of $\lambda$ enables us to twist commutative Laurent polynomial algebras by the associated bicharacter and obtain an algebra where the generating variables quasi-commute.
	
	However, we have avoided insisting on more than is needed, until the crucial point.  The following tidies this up and also introduces the analogous notions and theory to that in Section~\ref{ss:forms-skew-sym} for $\beta$.
	
	\begin{definition}\label{d:A-form}
		Let $\cC_{q}=\genericAQCS$ be an abstract quantum cluster structure over a simple directed graph $E$.  For $c\in C$, define $\canform{\blank}{\blank}{\cA c}\colon \cA c \cross \cA c \to \integ$ by $\canform{a}{b}{\cA c}=\evform{\lambda_{c}(a)}{b}$.
	\end{definition}
	
	We will call this form the \emph{$\cA$-form} at $c$, by analogy with the $\cX$-form, $\canform{\blank}{\blank}{\cX c}$.
	
	\begin{lemma}
		The $\cA$-form $\canform{\blank}{\blank}{\cA c}$ is the pullback of $\canform{\blank}{\blank}{\cX c}$ along $\rho_{c}$, i.e.\
		\[ \canform{\blank}{\blank}{\cA c}=\canform{\blank}{\blank}{\cX c}\circ (\rho_{c} \cross \rho_{c}). \]
	\end{lemma}
	
	\begin{proof}
		We have 
		\begin{align*} \canform{a}{b}{\cA c} & = \evform{\lambda_{c}(a)}{b} \\
		 & = \lambda_{c}(a)(b) \\
		& = (\delta_{\cX c}\circ \rho_{c})(a)(b) \\
		& = \canform{b}{\rho_{c}(a)}{c} \\
		& \stackrel{\ref{l:compat-equiv-conds}\ref{l:compat-equiv-conds-forms}}{=} \evform{\lambda_{c}(b)}{(\beta_{c} \circ \rho_{c})(a)} \\
		& = \evform{(\delta_{\cX c}\circ \rho_{c})(b)}{(\beta_{c} \circ \rho_{c})(a)} \\
		& = \canform{\beta_{c}(\rho_{c}(a))}{\rho_{c}(b)}{c} \\
		& = \canform{\rho_{c}(a)}{\rho_{c}(b)}{\cX c}. \qedhere
		\end{align*}
	\end{proof}
	
	Then, if a compatible $\lambda$ exists for $\beta$, i.e.\ there is a retraction $\rho$ of $\beta$, we also see that
	\[ \canform{\blank}{\blank}{\cX c}=\canform{\blank}{\blank}{\cX c}\circ ((\rho_{c}\circ \beta_{c}) \cross (\rho_{c}\circ \beta_{c}))=\canform{\blank}{\blank}{\cA c}\circ (\beta_{c}\cross \beta_{c}) \]
	so that the $\cX$-form is the pullback of the $\cA$-form along $\beta$.
	
	\begin{corollary} Let $\cC_{q}=\genericAQCS$ be an abstract quantum cluster structure.
		
		The following are equivalent:
		\begin{enumerate}
			\item $\cC_{q}$ is skew-symmetrizable;
			\item $\canform{\blank}{\blank}{\cX}$ is skew-symmetric;
			\item $\canform{\blank}{\blank}{\cA}$ is skew-symmetric;
			\item $\delta_{\cA}\circ \beta$ is skew-symmetric; and 
			\item $\lambda=\delta_{\cX} \circ \rho$ is skew-symmetric.
		\end{enumerate}
	\end{corollary}
		
	\begin{proof} This now follows from the above and Proposition~\ref{p:ACS-skew-symmetrizable}.
	\end{proof}
		
	Explicitly, the factorization condition means that $\lambda$ is a collection of maps $\lambda_c$, one for each $c \in \cE$, such that for every morphism $f \colon c \to d$ in $\cE$ the following diagram commutes:
	\[\begin{tikzcd}[column sep=25pt]
		\cA c \arrow[->]{r}[above]{\lambda_{c}} \arrow[->]{d}[left]{\cA f} & \dual{\cA} c  \\
		\cA d  \arrow[->]{r}[below]{\lambda_{d}} & \dual{\cA} d  \arrow[->]{u}[right]{\dual{\cA} f}
	\end{tikzcd}\]

	We have the analogous claim to Proposition~\ref{p:X-form-mut}.
	
	\begin{proposition}
		Let $\cC_{q}=\genericAQCS$ be an abstract quantum cluster structure over a simple directed graph $E$.  Then for $c\to d$, the forms $\ip{\blank}{\blank}_{\cA c}$ and $\ip{\blank}{\blank}_{\cA d}$ are related by
		\[ \ip{\blank}{\blank}_{\cA c}=\ip{\blank}{\blank}_{\cA d}\circ (\cA + \cross \cA +)=\ip{\blank}{\blank}_{\cA d}\circ (\cA- \cross \cA -) . \]
	\end{proposition}
	
	\begin{proof}
		Consider $a,b\in \cA c$.  Then
		\begin{align*} \ip{a}{b}_{\cA c} = & \evform{\lambda_{c}(a)}{b} \\
			 & = \lambda_{c}(a)(b) \\
			 & = (\dual{\cA}\!\pm \circ \lambda_{d} \circ \cA \pm)(a)(b) \\
			 & = \lambda_{d}((\cA \pm)(a))((\cA \pm)(b)) \\
			 & = \evform{(\lambda_{d} \circ \cA \pm)(a)}{(\cA \pm)(b)} \\
			 & = \canform{(\cA \pm)(a)}{(\cA \pm)(b)}{\cA d}. \qedhere
		\end{align*} 
	\end{proof}
	
	Then, again as for the $\cX$-form, it follows that if $E$ is bi-directed then the corresponding transformation $\canform{\blank}{\blank}{\cA}\colon \cA\tensor_{\integ} \cA \to \sZ$ is determined by a choice of one form $\canform{\blank}{\blank}{\cA c}$ for some $c$.  Again, this is already implicit in $\lambda$ being a factorization.
	
	\sectionbreak
		\section{The category of abstract (quantum) cluster structures}\label{s:cat-of-ACS}

	In order to achieve the benefits we claim for our abstract approach, we want the collection of abstract cluster structures to form a category.
	
	Rather than repeat ourselves unnecessarily, we will give the definitions in this section for abstract \emph{quantum} cluster structures and indicate where one should omit additional data or conditions for the non-quantum case.
		
		\subsection{Morphisms}\label{ss:morphisms-in-ACS}

	\begin{definition}\label{d:morphism-in-ACS}
		Let $\cC_{1}=(\cE_{1},\cX_{1},\beta_{1},\lambda_{1},\cA_{1},\ip{\blank}{\blank}^{1})$, $\cC_{2}=(\cE_{2},\cX_{2},\beta_{2},\lambda_{2},\cA_{2},\ip{\blank}{\blank}^{2})$ be a pair of abstract quantum cluster structures.  We define a \emph{morphism} of abstract quantum cluster structures to be a tuple $(F,\chi,\alpha)$ such that
		\begin{enumerate}[label=(\textup{\alph*})]
			\item $F\colon \cE_{1}\to \cE_{2}$ is a functor;
			\item\label{chi-nt} $\chi\colon \cX_{1}\to \cX_{2}\op{F}$ is a natural transformation;
			\item\label{alpha-nt} $\alpha\colon \cA_{1}\to \cA_{2}F$ is a natural transformation;
			\item\label{beta-squares}\setcounter{enumi}{16}  $\alpha\circ \beta_1=\beta_{2}\op{F}\circ \chi$;
			\item\label{lambda-squares} $\lambda_{1}=\dual{\alpha}\circ \lambda_{2}F\circ \alpha$.
		\end{enumerate}
	\end{definition}
	
	\begin{remark}
	Note that:
	\begin{itemize}
		\item conditions \ref{chi-nt} and \ref{alpha-nt} are precisely saying that $\chi$ and $\alpha$ are morphisms of the given sheaves;
		\item the conditions in \ref{chi-nt}, \ref{alpha-nt} and \ref{beta-squares} ensure that for all $c\to d$, the following cube commutes:
		\[
		\xymatrix{ \cX_{1}c \ar[rr]^-{(\beta_{1})_{c}} \ar[rd]_{\chi_{c}} & &  \cA_{1}c \ar[dd]^(0.66){\cA_{1}\pm} \ar[rd]^-{\alpha_{c}} & \\
			& \cX_{2}\op{F}c \ar[rr]^(0.33){(\beta_{2})_{Fc}} & & \cA_{2}Fc \ar[dd]^-{\cA_{2}F\pm} \\
			\cX_{1}d \ar[uu]^-{\cX_{1}\pm} \ar[rd]_{\chi_d} \ar[rr]_(0.66){(\beta_{1})_{d}} & & \cA_{1}d \ar[rd]^-{\alpha_{d}} & \\
			& \cX_{2}\op{F}d \ar[uu]^(0.66){\cX_{2}\op{F}\pm} \ar[rr]_-{(\beta_{2})_{Fd}} & & \cA_{2}Fd}
		\]
		\item in the quantum case, conditions \ref{beta-squares} and \ref{lambda-squares} imply the following:
		\begin{equation}
			\label{adjoint-pres} \canform{\blank}{\blank}{}^{1}=\canform{F\blank}{\op{F}\blank}{}^{2}\circ (\alpha \cross \chi)
		\end{equation}
		since 
		\begin{align*}
			\canform{\blank}{\blank}{}^{1} & =\evform{\blank}{\blank}\circ (\lambda_{1} \tensor_{\integ} \beta_{1}) \\
			& =\evform{\blank}{\blank}\circ (\dual{\alpha}\circ \lambda_{2}F\circ \alpha \tensor_{\integ} \beta_{1}) \\
			& = \evform{(\lambda_{2}F\circ \alpha)(\blank)}{(\alpha\circ \beta_{1})(\blank)} \\
			& = \evform{(\lambda_{2}F\circ \alpha)(\blank)}{(\beta_{2}\op{F}\circ \chi)(\blank)} \\
			& = \canform{F\blank}{\op{F}\blank}{}^{2}\circ (\alpha \cross \chi).
		\end{align*}
		\item for the non-quantum case, condition~\ref{lambda-squares} should be omitted.  We do not ask for \eqref{adjoint-pres} instead, despite this being a natural condition on the preservation of the associated forms.  This choice has potential implications for categorical properties of $\ACScat$; with the current proto-theory of abstract cluster structures, it is unclear what the correct definition should be, so we have opted for the weakest for the time being. 
	\end{itemize} 
	\end{remark}
	
	\subsection{The category of abstract quantum cluster structures}\label{ss:cat-of-ACS}
	
	\begin{theorem}
		The collection of abstract quantum cluster structures over simple directed graphs together with morphisms of these forms a category.
	\end{theorem}
	
	\begin{proof}
		Composition of morphisms is defined as follows.  For $\curly{F}=(F,\chi,\alpha)\colon \cC_{1}\to \cC_{2}$ and $\curly{G}=(G,\chi',\alpha')\colon \cC_{2}\to \cC_{3}$, we set $\curly{G}\circ \curly{F}=(G\circ F,\chi'\circ \chi,\alpha'\circ \alpha)$, where the composition of natural transformations $\chi' \circ \chi$ and $\alpha' \circ \alpha$ are respectively 
		\[ \chi'\circ \chi\colon \cX_{1}\to \cX_{3}\op{(G\circ F)},\ (\chi'\circ \chi)_{c}=(\chi'\op{F})_{Fc}\circ \chi_{c} \] and \[ \alpha'\circ \alpha\colon \cA_{1}\to \cA_{3}(G\circ F),\ (\alpha'\circ \alpha)_{c}=(\alpha'F)_{Fc}\circ \alpha_{c}.\]  It is straightforward to check that this is again a morphism of abstract quantum cluster structures $\curly{G}\circ \curly{F}\colon \cC_{1}\to \cC_{3}$, by concatenating the relevant diagrams, and moreover that this composition is associative, similarly.
		
		Since we also have the identity morphism of abstract quantum cluster structures given by $\id_{\cC}=(\id_{\cE},\id_{\cX},\id_{\cA})$, for which the functors and (components of) the natural transformations are the respective identities, and this indeed satisfies $\curly{F}\circ \id_{\cC_{1}}=\curly{F}=\id_{\cC_{2}}\circ \curly{F}$ for all $\curly{F}\colon \cC_{1}\to \cC_{2}$, we have a category as claimed.
	\end{proof}
	
	\begin{corollary}
		The collection of abstract cluster structures over simple directed graphs together with morphisms of these forms a category. \qed
	\end{corollary}
		
	\begin{definition}
		Let $\AQCScat$ denote the category of abstract quantum cluster structures and $\ACScat$ denote the category of abstract cluster structures.
	\end{definition}
	
	There is a forgetful functor $\mathscr{F}\colon \AQCScat\to \ACScat$ sending $(\cE,\cX,\beta,\lambda,\cA,\ip{\blank}{\blank})$ to \linebreak $(\cE,\cX,\beta,\cA,\ip{\blank}{\blank})$ and equal to the identity on morphisms. Note, however, that $\mathscr{F}$ is not (even essentially) surjective: since a $\lambda$ compatible with $\beta$ can only exist if $\beta$ is a split monomorphism, there are many abstract cluster structures that are not in the image of $\mathscr{F}$.  
	
	In what follows, we will see that when we prove results about abstract quantum cluster structures, we may sometimes obtain the same claim for abstract cluster structures (i.e.\ with no injectivity assumption on $\beta$) by omitting the no-longer relevant parts of the proof.  As such, we will ``deduce'' the corresponding claim in the non-quantum case from the quantum one as a corollary, but of the proof rather than directly of the statement.  In this way, we will avoid having to repeat ourselves excessively.
	
	However, some caution is required: since the definition of morphism in $\AQCScat$ is stronger than that in $\ACScat$, $\mathscr{F}$ is not full.  That is, there are abstract cluster structures for which quantizations exists and morphisms in $\ACScat$ between them, but where there is no corresponding morphism in $\AQCScat$ between their quantizations.  The consequence of this is that the two categories $\ACScat$ and $\AQCScat$ have different categorical properties, as we shall see in Section~\ref{ss:initial-terminal}.
	
	\subsection{Isomorphisms}\label{ss:monic-epic-iso-in-ACS}
	
	The isomorphisms in $\ACScat$ and $\AQCScat$ are as follows.
	
	\begin{proposition}
		A morphism of abstract cluster structures (respectively, abstract quantum cluster structures) $\curly{F}=(F,\chi,\alpha)$ is an isomorphism in $\ACScat$ (respectively, $\AQCScat$) if and only if $F$ is an isomorphism of categories and $\chi$ and $\alpha$ are natural isomorphisms.
	\end{proposition}
	
	\begin{proof} If $\curly{F}$ has a two-sided inverse $\curly{G}=(G,\chi',\alpha')$, then by the definitions of composition and of the identity morphism in $\ACScat$ or $\AQCScat$, each of $F$, $\chi$ and $\alpha$ is invertible, so that they have the claimed properties.  The converse is also clear.
	\end{proof}
	
	Each component being an isomorphism, especially the functor $F$, might initially seem like a very strict condition (including with respect to the usual notion of strictness in category theory).
	However, the following phenomenon suggests that we should take care not to be too weak in what we ask for.  
	
	For two abstract cluster structures---which are designed precisely to capture cluster combin\-atorics---to be isomorphic, we want that they have the same pattern of clusters and mutations, i.e.\ the same exchange graphs.  In our setting, the exchange graph is the input data, $E$, a signed directed graph.  Now if $E$ and $E'$ are signed bi-drected graphs (which is in particular the case for involutive mutations), we have that $\curly{E}(E)$ and $\curly{E}(E')$ are groupoids and hence equivalent as categories to a bunch of copies of the 1-object category, one copy for each connected component.  So it is far too easy for $\curly{E}(E)$ to be equivalent to $\curly{E}(E')$.
		
	On the other hand, since we see in an isomorphism of abstract cluster structures $(F,\chi,\alpha)$ that the functor $F$ is an isomorphism of categories, we immediately deduce that the underlying exchange graphs are isomorphic.
	
	\begin{remark}
		
	A natural next question is to enquire about monic and epic morphisms.  We make the following opening observations in response.
		
	A morphism $\curly{F}=(F,\chi,\alpha)$ in $\AQCScat$ will be monic, respectively epic, if its components are, in their natural categories.  In particular, $\curly{F}$ is monic if and only if $F$ is an embedding of categories (i.e.\ injective-on-objects and faithful) and $\chi$ and $\alpha$ have injective components\footnote{That we have such a simple characterization is due to choosing to work with (free) Abelian sheaves: it is the fact that $\Abcat$ is Abelian that means that natural transformations of functors having values there are monic if and only if they have injective components, and dually for epic.}.  
				
	If $F$ is surjective-on-objects and full and $\chi$ and $\alpha$ have surjective components, then $\curly{F}$ is epic; we would be interested to know whether the converse holds and in particular, whether the general characterization of when functors are epic given by Isbell's zig-zag theorem \cite{Isbell} simplifies given the rather specific form of $\cE$.
	\end{remark}
	
	\subsection{Initial and terminal objects}\label{ss:initial-terminal}
	
	The category $\ACScat$ has initial and terminal objects, and the initial object of $\ACScat$ is also initial in $\AQCScat$, as we show now.  
	
	\begin{definition} Let $\curly{I}$ be the abstract quantum cluster structure with
		\[ \curly{I}=(\cE(\emptyset), \cX=\emptyset_{\op{\cE(\emptyset)}},\beta=\emptyset,\lambda=\emptyset,\cA=\emptyset_{\cE(\emptyset)},\ip{\blank}{\blank}=\emptyset). \]
	\end{definition}
	
	Here, we start with $E=\emptyset$, the empty graph, whose signed path category $\cE(\emptyset)$ is the empty category.  Then $\cX\colon \op{\cE(\emptyset)}\to \Abcat$ is the unique functor $\emptyset_{\op{\cE(\emptyset)}}$ arising from the empty category being initial in $\mathbf{Cat}$, and similarly for $\cA$.  All of $\beta$, $\lambda$ and $\ip{\blank}{\blank}$ are empty because the indexing set for their components is empty.
	
	\begin{proposition}
		$\curly{I}$ is initial in the category $\AQCScat$.
	\end{proposition}
	
	\begin{proof} Given an abstract quantum cluster structure $\cC$, the unique morphism in $\AQCScat$ from $\curly{I}$ to $\cC$ is $(\emptyset_{\cE},\chi,\alpha)$ with $\emptyset_{\cE}\colon \emptyset \to \cE$ the unique such functor, $\chi\colon \emptyset_{\Abcat}\to \chi \op{\emptyset_{\cE}}=\emptyset_{\Abcat}$ the unique such natural transformation and similarly for $\alpha\colon \emptyset_{\Abcat}\to \alpha \emptyset_{\cE}=\emptyset_{\Abcat}$.	 
	\end{proof}
	
	Note that the image $\mathscr{F}\curly{I}$ of $\curly{I}$ under the forgetful functor from $\AQCScat$ to $\ACScat$ is initial in $\ACScat$, by an identical argument\footnote{There is a small subtlety here: to prove the claim in $\ACScat$, we have more objects to check the claims for. However, since the proofs make no reference to $\beta$, its injectivity or otherwise is moot.}.
	
	Let $\mathbbm{1}$ be the terminal object in $\mathbf{Cat}$, so that $\mathbbm{1}$ has one object $*$ and one morphism $\id_{*}$.  Let $E=*$ be the graph with one vertex and no arrows, so that $\cE(*)=\mathbbm{1}$.
	
	\begin{definition}
		Let $\cT$ be the abstract cluster structure with
		\[ \cT=(\mathbbm{1},\cX=\op{\mathbf{0}},\beta=\id_{\mathbf{0}},\cA=\mathbf{0},\ip{\blank}{\blank}_{\mathbf{0}}) \]
		where
		\begin{itemize}
			\item $\op{\mathbf{0}}\colon \op{\mathbbm{1}}\to \Abcat$, $\op{\mathbf{0}}(*)=\{0\}$, $\op{\mathbf{0}}(\id_{*})=\id_{\{0\}}$;
			\item $\mathbf{0}\colon \mathbbm{1}\to \Abcat$, $\mathbf{0}(*)=\{0\}$, $\mathbf{0}(\id_{*})=\id_{\{0\}}$;
			\item $\beta=\id_{\mathbf{0}}\defeq \{ \beta_{*}=\id_{\{0\}} \}$; and
			\item $\ip{\blank}{\blank}_{\mathbf{0}}\defeq \{ \ip{\blank}{\blank}_{*}\colon \{0\}\cross \{0\}\to \{0\},\ \ip{0}{0}_{*}=0 \}$.
		\end{itemize}
	\end{definition}
	
	\noindent This is an abstract cluster structure: 
	\begin{itemize}
		\item $\cX$ and $\cA$ are free Abelian sheaves on $\mathbbm{1}=\op{\mathbbm{1}}$;
		\item the factorization condition for $\beta$ is trivial since $\mathbbm{1}$ has only one object and the identity morphism; 
		\item $\ip{\blank}{\blank}_{\mathbf{0}}$ has trivial right radical and hence is right non-degenerate.
	
	\end{itemize}
	
	\begin{proposition}
		$\curly{T}$ is terminal in the category $\ACScat$.
	\end{proposition}
	
	\begin{proof} There exists a unique functor $F=\mathbbm{1}_{\curly{E}}\colon \curly{E}\to \mathbbm{1}$, since $\mathbbm{1}$ is terminal in $\mathbf{Cat}$.
		
		There exists a unique $\chi\colon \cX\to \mathbf{0}\op{\mathbbm{1}}_{\curly{E}}$ whose components $\chi_{c}\colon \cX c\to \{0\}$ are all equal to the unique such surjective map, and similarly for $\alpha$.  One may then readily check that $(F,\chi,\alpha)$ is a morphism in $\ACScat$.		
	\end{proof}

	\begin{remark} The abstract cluster structure $\cT$ has a quantization $\cT_{q}$, with \hfill \hfill \linebreak $\lambda_{\cT}=\{ \lambda_{*}\colon \{0\} \to \dual{\{0\}} \}$ for $\lambda_{*}$ the unique such map, with $\lambda_{*}(0)=\hat{0}$ for $\hat{0}\colon \{0\}\to \integ$, $\hat{0}(0)=0$.  That this yields an abstract quantum cluster structure follows since the factorization condition for $\lambda$ is trivial and $\lambda_{*}$ is skew-symmetric and compatible with $\beta$ since all maps involved are zero.
		
	However, $\cT_{q}$ is not terminal in $\AQCScat$, since the necessary morphism as in the above proof is not a morphism in $\AQCScat$.  For condition~\ref{lambda-squares} requires that $\lambda_{1}=\dual{\alpha}\circ \lambda_{2}F\circ \alpha$, but the right-hand side is zero whereas the left is not.
	
	As this choice of $\lambda_{\cT}$ is the only possible quantization of $\cT$, we expect that $\AQCScat$ in fact has no terminal object.
	
	This is the previously signposted example of how the categorical properties of $\AQCScat$ differ from those of $\ACScat$, due to the forgetful functor not being full.
	\end{remark}

	Note that since the empty graph is not isomorphic to the graph $*$ with one vertex and no arrows, $\curly{I}$ is not isomorphic to $\cT$ in $\ACScat$.  Consequently, the category $\ACScat$ does not have a zero object and so cannot be additive\footnote{Noting that the first piece of data in an abstract (quantum) cluster structure is a (signed) path category $\cE$ over a graph and the first piece of data in a morphism of abstract (quantum) cluster structures is a functor of these, it is in some sense inevitable that $\ACScat$ behaves somewhat like a category of graphs. The discussion at \url{https://ncatlab.org/nlab/show/category+of+simple+graphs} is particularly relevant.}, and similarly for $\AQCScat$.
	
	\subsection{Products and coproducts}\label{ss:props-of-ACS-cat}

	We now show that $\ACScat$ has all finite products and coproducts, and that $\AQCScat$ has all finite coproducts. However, $\ACScat$ has no zero object (as its initial and terminal objects are not isomorphic) so that it does not have all finite biproducts.  Similarly, the (expected) lack of a terminal object in $\AQCScat$ prevents the existence even of all finite products.
	
	The construction for products we give is that of a direct sum, and bootstraps from the product (i.e.\ Cartesian product) in $\Catcat$ and the biproduct (i.e.\ direct sum) in $\Abcat$.  As with (the special case of) terminal objects, the failure for this construction to be a biproduct in $\ACScat$ (or $\AQCScat$) is due more to the lack of suitable morphisms rather than objects.
	
	For a simple directed graph $E$ with vertex set $C$, recall that we denote by $E(c,d)$ the set of arrows $c\to d$ in $E$, for $c,d\in C$.

	Let $E_{1}$ and $E_{2}$ be simple directed graphs with vertex sets $C_{1}$ and $C_{2}$.  Then let $E_{1}\prodsml E_{2}$ denote the simple directed graph with vertex set $C_{1}\cross C_{2}$ and arrows \[ (E_{1}\prodsml E_{2})((c_{1},c_{2}),(d_{1},d_{2}))=E_{1}(c_{1},d_{1})\union E_{2}(c_{2},d_{2}). \]
	This models taking the (disjoint) union of clusters in two exchange graphs and mutations in the new exchange graph being those associated to the first cluster together with those associated to the second.  (Note, though, that this is not any of the common graph products, such as the Cartesian or Kronecker/tensor product.)
	
	We see that the signed path category $\cE(E_{1}\prodsml E_{2})$ has objects $C_{1}\cross C_{2}$ and morphisms (i.e.\ paths) generated by $(\alpha,e)$ and $(e,\itbeta)$ where $\alpha$ and $\itbeta$ stand for arrows originating from $E_{1}$ and $E_{2}$ respectively and $e$ stands for the trivial path at a vertex.  Then since $(\alpha,e)\circ(e,\beta)=(\alpha,\beta)$, in the (signed) path category, we obtain that \[ \Mor{\cE(E_{1}\prodsml E_{2})}{(c_{1},c_{2})}{(d_{1},d_{2})}=\Mor{\cE_{1}}{c_{1}}{d_{1}})\cross \Mor{\cE_{2}}{c_{2}}{d_{2}} \] so that $\cE(E_{1}\prodsml E_{2})=\cE_{1}\cross \cE_{2}$, the Cartesian product of categories.
	
	For $F_{1}$ and $F_{2}$ presheaves of Abelian groups over categories $\cC_{1}$ and $\cC_{2}$ respectively, we have a product presheaf $F_{1}\dsum F_{2}$ on $\cC_{1}\cross \cC_{2}$ given by $F_{1}\dsum F_{2}\colon \op{(\cC_{1}\cross \cC_{2})}\to \Abcat$, where $(F_{1}\dsum F_{2})(c_{1},c_{2})=F_{1}c_{1}\dsum F_{2}c_{2}$ on objects and $(F_{1}\dsum F_{2})(f_{1},f_{2})=F_{1}f_{1}\dsum F_{2}f_{2}$ on morphisms.  Here, $F_{1}c_{1}\dsum F_{2}c_{2}$ and $F_{1}f_{1}\dsum F_{2}f_{2}$ are the usual constructions in $\Abcat$, namely direct sums of Abelian groups and ``block diagonal'' sum of homomorphisms, arising from the existence of finite products and coproducts in $\Abcat$.  
	
	We also have a direct sum of natural transformations or factorizations: if $\alpha_{1}\colon F_{1}\to G_{1}$ and $\alpha_{2}\colon F_{2}\to G_{2}$ are natural transformations of presheaves of Abelian groups $F_{1},G_{1}\colon \op{\cC_{1}}\to \Abcat$ and $F_{2},G_{2}\colon \op{\cC_{2}}\to \Abcat$, we have $\alpha_{1}\dsum \alpha_{2}\colon F_{1}\dsum F_{2}\to G_{1}\dsum G_{2}$ induced by $(\alpha_{1}\dsum \alpha_{2})_{(c_{1},c_{2})}=(\alpha_{1})_{c_{1}}\dsum (\alpha_{2})_{c_{2}}$.
	
	Similarly, for dinatural transformations of the form $\ip{\blank}{\blank}^{1}\colon F_{1}\tensor_{\integ} G_{1}\to \sZ$ and $\ip{\blank}{\blank}^{2}\colon F_{2}\tensor_{\integ} G_{2}\to \sZ$, we have an induced dinatural transformation 
	\[ \ip{\blank}{\blank}^{\dsum}\colon (F_{1}\dsum F_{1})\tensor_{\integ} (G_{1}\dsum G_{2})\to \sZ \]
	defined by $\ip{\blank}{\blank}^{\dsum}_{(c_{1},c_{2})}=\ip{\blank}{\blank}^{1}_{c_{1}}+\ip{\blank}{\blank}^{2}_{c_{2}}$.
	
	We claim that there is a natural abstract quantum cluster structure over $\cE(E_{1}\prodsml E_{2})$.
	
	\begin{proposition} Let $E_{1}$ and $E_{2}$ be simple directed graphs and let $\cC_{1}=(\cE_{1},\cX_{1},\beta_{1},\lambda_{1},\cA_{1},\ip{\blank}{\blank}^{1})$ and  $\cC_{2}=(\cE_{2},\cX_{2},\beta_{2},\lambda_{2},\cA_{2},\ip{\blank}{\blank}^{2})$ be abstract quantum cluster structures over $E_{1}$ and $E_{2}$ respectively.
	
	Then \[ \cC_{1}\prodsml \cC_{2}=(\cE(E_{1}\prodsml E_{2}),\cX_{1}\dsum \cX_{2},\beta_{1}\dsum \beta_{2},\lambda_{1}\dsum \lambda_{2},\cA_{1}\dsum \cA_{2},\ip{\blank}{\blank}^{\dsum}) \]
	is an abstract quantum cluster structure.
	\end{proposition}
	
	\begin{proof} Since the direct sum of free Abelian groups is free Abelian (on the union of the respective bases) $\cX_{1}\dsum \cX_{2}$ and $\cA_{1}\dsum \cA_{2}$ are free Abelian sheaves on $\cE(E_{1}\prodsml E_{2})=\cE_{1}\cross \cE_{2}$.  It is also straightforward to see that $\beta_{1}\dsum \beta_{2}$ is a factorization, by taking the direct sum of the respective factorization diagrams, and similarly for $\lambda_{1}\dsum \lambda_{2}$.  Skew-symmetrizability is preserved under the direct sum construction too.
		
		Since kernels are computed locally, the right non-degeneracy follows from that for the constituent parts of the direct sum form.
	\end{proof}
	
	\begin{corollary}
	Let $E_{1}$ and $E_{2}$ be simple directed graphs and let $\cC_{1}=(\cE_{1},\cX_{1},\beta_{1},\cA_{1},\ip{\blank}{\blank}^{1})$ and  $\cC_{2}=(\cE_{2},\cX_{2},\beta_{2},\cA_{2},\ip{\blank}{\blank}^{2})$ be abstract cluster structures over $E_{1}$ and $E_{2}$ respectively.
		
		Then \[ \cC_{1}\prodsml \cC_{2}=(\cE(E_{1}\prodsml E_{2}),\cX_{1}\dsum \cX_{2},\beta_{1}\dsum \beta_{2},\cA_{1}\dsum \cA_{2},\ip{\blank}{\blank}^{\dsum}) \]
		is an abstract cluster structure. \qed
	\end{corollary}

	\begin{theorem}\label{p:ACS-has-products} 
				The abstract cluster structure $\cC_{1}\prodsml \cC_{2}$ is a categorical product of $\cC_{1}$ and $\cC_{2}$ in $\ACScat$. Hence, $\ACScat$ has finite products.
	\end{theorem}
	
	\begin{proof}
			Keeping the notation of the previous discussion and Proposition, we claim that there are morphisms of abstract cluster structures 
			\begin{equation*}
				\begin{tikzcd} \cC_{1} & \cC_{1}\prodsml \cC_{2} \arrow{l}[above]{\curly{P}_{1}} \arrow{r}{\curly{P}_{2}} & \cC_{2} \end{tikzcd}
			\end{equation*}
			such that for any $\cC\in \ACScat$ and morphisms $\curly{F}_{i}\colon \cC\to \cC_{i}$ for $i=1,2$, we have a unique morphism $\curly{F}\colon \cC\to \cC_{1}\prodsml \cC_{2}$ such that $\curly{P}_{i}\circ \curly{F} = \curly{F}_{i}$.
			That is,
			\begin{equation*}
				\begin{tikzcd} & \cC \arrow{dl}[above left]{\curly{F}_{1}} \arrow{dr}{\curly{F}_{2}} \arrow[dashed]{d}{\curly{F}} & \\ \cC_{1} & \cC_{1}\prodsml \cC_{2} \arrow{l}[below]{\curly{P}_{1}} \arrow{r}[below]{\curly{P}_{2}} & \cC_{2}  \end{tikzcd}
			\end{equation*}
			
			Given $\cC_{1}=(\cE_{1},\cX_{1},\beta_{1},\cA_{1},\ip{\blank}{\blank}^{1})$,  $\cC_{2}=(\cE_{2},\cX_{2},\beta_{2},\cA_{2},\ip{\blank}{\blank}^{2})$ abstract cluster structures over $E_{1}$ and $E_{2}$ respectively, as above we have \[ \cC_{1}\prodsml \cC_{2}=(\cE(E_{1}\prodsml E_{2}),\cX_{1}\dsum \cX_{2},\beta_{1}\dsum \beta_{2},\cA_{1}\dsum \cA_{2},\ip{\blank}{\blank}^{\dsum}) \]
			
			Now, the Cartesian product of categories is the product in $\Catcat$.  So, we may let $P_{i}$, $i=1,2$ be the natural functors associated with the Cartesian product decomposition of $\cE(E_{1}\prodsml E_{2})=\cE_{1}\cross \cE_{2}$, i.e.\ 
			\begin{equation*}
				\begin{tikzcd} \cE_{1} & \cE_{1}\cross \cE_{2} \arrow{l}[above]{{P}_{1}} \arrow{r}{{P}_{2}} & \cE_{2} \end{tikzcd}
			\end{equation*}
			
			Recall that $\cX_{1}\dsum \cX_{2}\colon \op{(\cE_{1}\cross \cE_{2})}\to \Abcat$ is given by $(\cX_{1}\dsum \cX_{2})(c_{1},c_{2})=\cX_{1} c_{1} \dsum \cX_{2} c_{2}$ on objects and by $(\cX_{1}\dsum \cX_{2})(f_{1},f_{2})=\cX_{1} f_{1}\dsum \cX_{2} f_{2}$ on morphisms.  Since $\dsum$ is the biproduct in $\Abcat$, for all $(c_{1},c_{2})\in C_{1}\cross C_{2}$ and $i=1,2$, there exist homomorphisms 
			\[ (\pi^{\chi}_{i})(c_{1},c_{2})\colon \cX_{1}c_{1} \dsum \cX_{2}c_{2} \to \cX_{i}c_{i} \]
			and
			\[ (\iota^{\chi}_{i})(c_{1},c_{2})\colon \cX_{i}c_{i} \to \cX_{1}c_{1} \dsum \cX_{2}c_{2} \]
			such that $(\pi^{\chi}_{i})(c_{1},c_{2})$ and $(\iota^{\chi}_{i})(c_{1},c_{2})$ satisfy the universal properties for products and coproducts respectively, $\pi^{\chi}_{i}(c_{1},c_{2})\circ \iota^{\chi}_{i}(c_{1},c_{2})=\id_{\cX_{i}c_{i}}$ and \[ \iota^{\chi}_{1}(c_{1},c_{2}) \circ \pi^{\chi}_{1}(c_{1},c_{2})\circ \iota^{\chi}_{2}(c_{1},c_{2}) \circ \pi^{\chi}_{2}(c_{1},c_{2}) = \iota^{\chi}_{2}(c_{1},c_{2}) \circ \pi^{\chi}_{2}(c_{1},c_{2}) \circ \iota^{\chi}_{1}(c_{1},c_{2}) \circ \pi^{\chi}_{1}(c_{1},c_{2}) \]
			That is, we have
			\begin{equation*}
				\begin{tikzcd}[column sep=40pt] \cX_{1}c_{1} \arrow[yshift=-0.5em]{r}[below]{(\iota^{\chi}_{1})(c_{1},c_{2})} & \cX_{1}c_{1} \dsum \cX_{2}c_{2} \arrow[yshift=0.5em]{l}[above]{(\pi^{\chi}_{1})(c_{1},c_{2})} \arrow[yshift=0.5em]{r}{(\pi^{\chi}_{2})(c_{1},c_{2})} & \cX_{2}c_{2} \arrow[yshift=-0.5em]{l}{(\iota^{\chi}_{2})(c_{1},c_{2})} \end{tikzcd}
			\end{equation*}
			
			For $i=1,2$, define $\pi^{\chi}_{i}\colon \cX_{1}\dsum \cX_{2}\to \cX_{i}$ to be the natural transformation with components $(\pi^{\chi}_{i})(c_{1},c_{2})$.  Naturality is the commuting of the square
			\[\begin{tikzcd}[column sep=40pt]
				(\cX_{1} \dsum \cX_{2})(c_{1},c_{2}) \arrow[->]{d}[left]{(\cX_{1}\dsum \cX_{2})(f_{1},f_{2})} \arrow[->]{r}[above]{(\pi^{\chi}_{i})(c_{1},c_{2})} & \cX_{i} c_{i} \arrow[->]{d}[right]{\cX_{i}f_{i}} \\
				(\cX_{1}\dsum \cX_{2})(d_{1},d_{2})  \arrow[->]{r}[below]{(\pi^{\chi}_{i})(d_{1},d_{2})} & \cX_{i}d_{i} 
			\end{tikzcd}\]
			for all $(f_{1},f_{2})\colon (c_{1},c_{2})\to (d_{1},d_{2})$, which follows since $(\cX_{1}\dsum \cX_{2})(f_{1},f_{2})=\cX_{1} f_{1}\dsum \cX_{2} f_{2}$.
			
			Arguing similarly with respect to $\cA_{1}\dsum \cA_{2}$, we obtain natural transformations \hfill \hfill \linebreak $\pi^{\alpha}_{i}\colon \cA_{1}\dsum \cA_{2} \to \cA_{i}$, also defined by projection onto the respective component.
			
			We claim that $\curly{P}_{i}\defeq (P_{i},\pi^{\chi}_{i},\pi^{\alpha}_{i})$ ($i=1,2$) are morphisms of abstract cluster structures.  To show this, it remains to check that $\pi^{\alpha}_{i}\circ (\beta_{1}\dsum \beta_{2})=\beta_{i}\op{P_{i}}\circ \pi^{\chi}_{i}$. But this follows immediately since $(\beta_{1}\dsum \beta_{2})_{(c_{1},c_{2})}=(\beta_{1})_{c_{1}}\dsum (\beta_{2})_{c_{2}}$, which is compatible with projection onto the relevant component.
			
			We then need to show the universality of $(\cC_{1}\prodsml \cC_{2},\curly{P}_{1},\curly{P}_{2})$.  Let $\cC=\genericACS$ be any abstract cluster structure and assume that $\curly{F}_{i}=(F_{i},\chi_{i},\alpha_{i})\colon \cC \to \cC_{i}$ ($i=1,2$) are morphisms of abstract cluster structures.
			
			Hence, since we have functors $F_{i}\colon \cE \to \cE_{i}$ as part of the data of $\curly{F}_{i}$, there exists a unique functor $F\colon \cE \to \cE_{1}\cross \cE_{2}$ such that $P_{i}\circ F=F_{i}$ for $i=1,2$:
			\begin{equation*}
				\begin{tikzcd} & \cE \arrow{dl}[above left]{F_{1}} \arrow{dr}{F_{2}} \arrow[dashed]{d}{F} & \\ \cE_{1} & \cE_{1}\cross \cE_{2} \arrow{l}[below]{P_{1}} \arrow{r}[below]{P_{2}} & \cE_{2}  \end{tikzcd}
			\end{equation*}
			In particular, on objects, writing $Fc=(c_{1},c_{2})$ for objects $c_{1}\in \cE_{1}$ and $c_{2}\in \cE_{2}$, we have $F_{i}c=(P_{i}\circ F)c=c_{i}$.
			
			Next, consider the natural transformations $\chi_{i}\colon \cX \to \cX_{i}\op{F_{i}}$ and $\alpha_{i}\colon \cA \to \cA_{i}F$ that we have as the remaining parts of the data of the $\curly{F}_{i}$.  Concentrating on one component $c\in C$, we have $\integ$-linear maps $(\chi_{i})_{c}\colon \cX c \to \cX_{i}\op{F_{i}}c$ and $(\alpha_{i})_{c}\colon \cA c \to \cA_{i}F_{i}c$.  
			
			As $\dsum$ is the biproduct in $\Abcat$, there exist 
			\[ p_{i}\colon \cX_{1}\op{F_{1}}c\dsum \cX_{2}\op{F_{2}}c \to \cX_{i}\op{F_{i}}c \]
			and (a unique) $\chi_{c}\colon \cX c \to \cX_{1}\op{F_{1}}c\dsum \cX_{2}\op{F_{2}}c$ such that $p_{i}\circ \chi_{c}=(\chi_{i})_{c}$:
			\begin{equation*}
				\begin{tikzcd} & \cX c \arrow{dl}[above left]{(\chi_{1})_{c}} \arrow{dr}{(\chi_{2})_{c}} \arrow[dashed]{d}{\chi_{c}} & \\ \cX_{1}\op{F_{1}}c & \cX_{1}\op{F_{1}}c \dsum \cX_{2}\op{F_{2}}c \arrow{l}[below]{p_{1}} \arrow{r}[below]{p_{2}} & \cX_{2}\op{F_{2}}c  \end{tikzcd}
			\end{equation*}
			Similarly, there exist 
			\[ q_{i}\colon \cA_{1}F_{1}c\dsum \cA_{2}F_{2}c \to \cA_{i}F_{i}c \]
			and (a unique) $\alpha_{c}\colon \cA c \to \cA_{1}F_{1}c\dsum \cA_{2}F_{2}c$ such that $q_{i}\circ \alpha_{c}=(\alpha_{i})_{c}$:
			\begin{equation*}
				\begin{tikzcd} & \cA c \arrow{dl}[above left]{(\alpha_{1})_{c}} \arrow{dr}{(\alpha_{2})_{c}} \arrow[dashed]{d}{\alpha_{c}} & \\ \cA_{1}F_{1}c & \cA_{1}F_{1}c \dsum \cA_{2}F_{2}c \arrow{l}[below]{q_{1}} \arrow{r}[below]{q_{2}} & \cA_{2}F_{2}c  \end{tikzcd}
			\end{equation*}
			
			Then
			\[ \chi\colon \cX\to (\cX_{1}\dsum \cX_{2})\op{F},\ \chi_{c}\colon \cX c\to \cX_{1}\op{F_{1}}c\dsum \cX_{2}\op{F_{2}}c=(\cX_{1}\dsum \cX_{2})_{\op{F}c} \]
			and
			\[ \alpha\colon \cA\to (\cA_{1}\dsum \cA_{2})F,\ \alpha_{c}\colon \cA c\to \cA_{1}F_{1}c\dsum \cA_{2}F_{2}c=(\cA_{1}\dsum \cA_{2})_{Fc} \]			
			defined by the above diagrams are natural transformations, by bifunctoriality of $\dsum$, and we have $\pi^{\chi}_{i} \circ \chi=\chi_{i}$ and $\pi^{\alpha}_{i} \circ \alpha=\alpha_{i}$ (where the $\op{F_{i}}$ and $F_{i}$ are notationally absorbed in the definition of composition of natural transformations).
			
			Define $\curly{F}=(F,\chi,\alpha)$.  We claim that $\curly{F}$ is a morphism of abstract cluster structures \linebreak $\curly{F}\colon \cC \to \cC_{1}\prodsml \cC_{2}$ such that $\curly{P}_{i}\circ \curly{F}=\curly{F}_{i}$.
	
			The chosen data has the required functoriality and naturality properties, so what remains is to check that $\alpha\circ \beta=(\beta_{1}\dsum \beta_{2})\op{F} \circ \chi$.  But tracking through the above definitions, we see that this follows from the $\curly{F}_{i}$ being morphisms of abstract cluster structures, so that $\alpha_{i}\circ \beta=\beta_{i}\op{F_{i}}\circ \chi_{i}$.
						
			Finally, 
			\[ \curly{P}_{i}\circ \curly{F}=(P_{i}\circ F,\pi^{\chi}_{i}\circ \chi,\pi^{\alpha}_{i}\circ \alpha)= (F_{i},\chi_{i},\alpha_{i})=\curly{F}_{i} \]
			as required.
	\end{proof}
	
	\begin{remark}
		As discussed previously, we expect that $\AQCScat$ has no terminal object and hence cannot have all finite products. The question of when $\cC_{1}\prodsml \cC_{2}$ is a product of $\cC_{1},\cC_{2}\in \AQCScat$ remains open.
	\end{remark}
	
	Let us now turn our attention to coproducts.
	
		Note that the Cartesian product is not the coproduct in $\Catcat$.  Rather the coproduct $\cE_{1} \coprodsml \cE_{2}$ of $\cE_{1}$ and $\cE_{2}$ is the ``disjoint union''.  That is, $\cE_{1}\coprodsml \cE_{2}$ is the category with objects 
		\[ \Obj{\cE_{1}\coprodsml \cE_{2}}=\Obj{\cE_{1}}\disjointunion \Obj{\cE_{2}} \] the disjoint union of the objects of $\cE_{1}$ and $\cE_{2}$, with morphisms sets empty unless the two objects come from the same factor, in which case the morphism set is that from the relevant factor:
		\[ \Hom{\cE_{1}\coprodsml \cE_{2}}{c}{d} = \begin{cases} \Hom{\cE_{i}}{c}{d} & \text{if}\ c,d\in \cE_{i} \\ \emptyset & \text{otherwise} \end{cases} . \]
		
		If $E_{1}$ and $E_{2}$ are simple directed graphs and $\cE_{1}=\cE(E_{1})$, $\cE_{2}=\cE(E_{2})$ their signed path categories, we see that the coproduct $\cE_{1}\coprodsml \cE_{2}$ is exactly $\cE(E_{1}\disjointunion E_{2})$, the signed path category of the disjoint union (also called sum) of the graphs $E_{1}$ and $E_{2}$.
		
		Now, given $\cX_{i}\colon \op{\cE_{i}}\to \Abcat$, for $i=1,2$, there is a functor $\cX_{1}\coprodsml \cX_{2}\colon \op{(\cE_{1}\coprodsml \cE_{2})}\to \Abcat$ given on objects by 
		\[ (\cX_{1}\coprodsml \cX_{2})(c) = \begin{cases} \cX_{1}c & \text{if}\ c\in \cE_{1} \\ \cX_{2}c & \text{if}\ c\in \cE_{2} \end{cases} \]
		and on morphisms $f\colon c\to d$ by
		\[ (\cX_{1}\coprodsml \cX_{2})(f) = \begin{cases} \cX_{1}f & \text{if}\ c,d\in \cE_{1} \\ \cX_{2}f & \text{if}\ c,d\in \cE_{2} \end{cases}. \]
		Indeed, $\cX_{1}\coprodsml \cX_{2}$ is a free Abelian sheaf on $\cE_{1}\coprodsml \cE_{2}$ and we may define $\cA_{1}\coprodsml \cA_{2}$ in exactly the same way.
		
		Similarly, given factorizations $\beta_{i}\colon \cX_{i}\to \cA_{i}$, $i=1,2$, there is a factorization $\beta_{1}\coprodsml \beta_{2}$ with component
		\[ (\beta_{1}\coprodsml \beta_{2})_{c} = \begin{cases} (\beta_{1})_{c} & \text{if}\ c\in \cE_{1} \\ (\beta_{2})_{c} & \text{if}\ c\in \cE_{2} \end{cases}. \]
		In the same way, we may define $\lambda_{1}\coprodsml \lambda_{2}$ and $\canform{\blank}{\blank}{}^{\coprodsml}$.
					
			\begin{proposition} Let $E_{1}$ and $E_{2}$ be simple directed graphs and let $\cC_{1}=(\cE_{1},\cX_{1},\beta_{1},\lambda_{1},\cA_{1},\ip{\blank}{\blank}^{1})$ and  $\cC_{2}=(\cE_{2},\cX_{2},\beta_{2},\lambda_{2},\cA_{2},\ip{\blank}{\blank}^{2})$ be abstract quantum cluster structures over $E_{1}$ and $E_{2}$ respectively.
			
			Then \[ \cC_{1}\coprodsml \cC_{2}=(\cE_{1}\coprodsml \cE_{2},\cX_{1}\coprodsml \cX_{2},\beta_{1}\coprodsml \beta_{2},\lambda_{1}\coprodsml \lambda_{2},\cA_{1}\coprodsml \cA_{2},\ip{\blank}{\blank}^{\coprodsml}) \]
			is an abstract quantum cluster structure.
		\end{proposition}
		
		\begin{proof} Beyond the discussion above, it remains to check that $\canform{\blank}{\blank}{}^{\coprodsml}$ is right non-degenerate and that $\beta_{1}\coprodsml\beta_{2}$ and $\lambda_{1}\coprodsml\lambda_{2}$ are compatible.
			
		For the former, this follows from observing that we check right non-degeneracy component-wise.  For the latter, we see that for $\rho_{1}$ and $\rho_{2}$ retractions of $\beta_{1}$ and $\beta_{2}$ respectively, $\lambda_{1}\coprodsml\lambda_{2}=\delta_{\cX_{1}\coprodsml\cX_{2}}\circ (\rho_{1}\coprodsml \rho_{2})$, as required, also by working component-wise.
		\end{proof}
		
		\begin{theorem} {\ }
			\begin{enumerate}
			\item The abstract cluster structure $\cC_{1}\coprodsml  \cC_{2}$ is a categorical coproduct of $\cC_{1}$ and $\cC_{2}$ in $\ACScat$. Hence, $\ACScat$ has finite coproducts.
			\item The abstract quantum cluster structure $\cC_{1}\coprodsml \cC_{2}$ is a categorical coproduct of $\cC_{1}$ and $\cC_{2}$ in $\AQCScat$. Hence, $\AQCScat$ has finite coproducts.
			\end{enumerate}
		\end{theorem}
		
		\begin{proof} Since the proof of this proposition is very similar in spirit to Theorem~\ref{p:ACS-has-products}, mainly consisting of categorically dual claims, we will give fewer details.
			\begin{enumerate}
			\item We claim that there are morphisms of abstract cluster structures 
		\begin{equation*}
			\begin{tikzcd} \cC_{1} \arrow{r}[above]{\curly{I}_{1}} & \cC_{1}\coprodsml \cC_{2}  & \cC_{2} \arrow{l}[above]{\curly{I}_{2}} \end{tikzcd}
		\end{equation*}
		such that for any $\cC\in \ACScat$ and morphisms $\curly{G}_{i}\colon \cC_{i}\to \cC$ for $i=1,2$, we have a unique morphism $\curly{G}\colon \cC_{1}\coprodsml \cC_{2} \to \cC$ such that $\curly{G}\circ \curly{I}_{i}=\curly{G}_{i}$.
		That is,
		\begin{equation*}
			\begin{tikzcd} \cC_{1} \arrow{r}[above]{\curly{I}_{1}} \arrow{dr}[below left]{\curly{G}_{1}} & \cC_{1}\coprodsml \cC_{2} \arrow[dashed]{d}{\curly{G}} & \cC_{2} \arrow{l}[above]{\curly{I}_{2}} \arrow{dl}{\curly{G}_{2}} \\ & \cC  & \end{tikzcd}
		\end{equation*}
		
		We need to find $\curly{I}_{i}=(I_{i},\iota_{i}^{\chi},\iota_{i}^{\alpha})$, $i=1,2$, morphisms of abstract cluster structures.  (Note that $\iota_{i}^{\chi}$ and $\iota_{i}^{\alpha}$ will not be the same maps as in the proof of Theorem~\ref{p:ACS-has-products}.)  First, we may obtain $I_{i}\colon \cE_{i}\to \cE_{1}\coprodsml \cE_{2}$ as the functors associated with the coproduct of categories.
		
		Next, $\iota_{i}^{\chi}\colon \cX_{1}\to (\cX_{1}\coprodsml\cX_{2})\op{I}_{i}$ should be a natural transformation: in fact, its components are the relevant identity map, since $((\cX_{1}\coprodsml \cX_{2})\op{I}_{i}){c}=\cX_{i}c$.  Similarly, $\iota_{i}^{\alpha}\colon \cA_{i}\to (\cA_{1}\coprodsml \cA_{2})\op{I}_{i}$ is a natural transformation with identity components.  
		
		Then the condition $\iota_{i}^{\alpha}\circ \beta_{i}=(\beta_{1}\coprodsml\beta_{2})\op{I}_{i}\circ \iota_{i}^{\chi}$ is also immediate: $(\beta_{1}\coprodsml\beta_{2})\op{I}_{i}=\beta_{i}$, as one sees by looking component-wise.
		
		This gives us candidate morphisms in $\ACScat$ and it remains to check that they have the required universal property, as in the above diagram.  Constructing $\curly{G}$ as $\curly{G}_{1}\coprodsml \curly{G}_{2}$ in the way suggested by the previous discussion, we see that this holds.
		
		\item The additional requirement in the quantum case is that the morphisms of abstract quantum cluster structures $\curly{I}_{i}$ should also satisfy $\lambda_{i}=\dual{(\iota_{i}^{\alpha})}\circ (\lambda_{1}\coprodsml \lambda_{2})\op{I}_{i}\circ \iota_{i}^{\alpha}$.  But similarly to the above, $(\lambda_{1}\coprodsml \lambda_{2})\op{I}_{i}=\lambda_{i}$ and $\iota_{i}^{\alpha}$ and its dual have identity components, so this holds.
		\qedhere
		\end{enumerate}
		\end{proof}
	
\sectionbreak
\pagebreak
\part{Representations of cluster combinatorics}\label{P:reps}

The philosophy of this part is to think of the various settings in which cluster combinatorics appear as being \emph{representations} of abstract cluster structures.  This is closest to being a formal construction for cluster varieties, where as we will see, the main idea is to form the torus associated to each lattice $\cA c$ and $\cX c$ and glue using the maps $\cA f$ and $\cX f$.  

The construction for cluster algebras is analogous to this but more complicated, yet still intuitive at heart: one passes between groups and algebras by ``exponentiation'' and ``taking logs''.  For an individual $c\in \cE$, we would use the group algebra functor and its adjoint, the group of units functor, to take us between $\cA c$ (respectively $\cX c$) and a Laurent polynomial ring $\bK[\cA c]$ (resp.\ $\bK[\cX c]$).  The technicality of course comes in the gluing step, i.e.\ connecting the different Laurent polynomial rings via mutation.  An additional nuance is that we want to cover the quantum case, so need to use twisted group algebras.

For cluster categories, the picture is incomplete, however.  We will see that to a cluster category one can associate a very natural abstract cluster structure, but the converse is not possible with current technology.  Essentially, this boils down to the question of finding categories with a suitable collection of subcategories where the Grothendieck groups of these match some pre-specified groups.

Similarly, we will see that triangulations of marked surfaces give rise to abstract cluster structures, again in a very natural way, but are not able to give a reverse construction.

In the final section, we will discuss morphisms between the different abstract cluster structures arising from the above classes of examples.

\sectionbreak
\section{Linear representations}\label{s:linear-reps}

In this section, we will consider linear representations of abstract cluster structures.  In fact, we will concentrate on some very specific representations, where the vector spaces in the image of our functor are actually algebras---even more specifically, they are (quantum) tori.

This is because our main goal here will be to show that (quantum) cluster algebras give rise to (quantum) abstract cluster structures, and vice versa.

To do this in as efficient a manner as possible, we will adopt an approach to constructing quantum cluster algebras that is somewhat different in style to the original definition, but uses a philosophy close to the geometric construction of cluster varieties and associated tropical spaces.  Of course, this is also closely aligned to the approach of Part~\ref{P:ACS} but since it is a less familiar setup, we will give full details.

Rather than treat commutative cluster algebras separately from the quantum case, we will simply start with the quantum version and point out at the relevant times how to recover the commutative setting if one prefers.  As with abstract cluster structures, this essentially boils down to simply forgetting one piece of the data.

\subsection{Quantum tori}\label{ss:quantum-tori}

For $\cB$ a finite or countably infinite set, we will write $\integ[\cB]$ for the free Abelian group on $\cB$, so that $\cB$ is a $\integ$-basis for $\integ[\cB]$.  Note that we do not assume $\cB$ has been enumerated or ordered at this point, so that although there exist isomorphisms $\integ[\cB]\iso \integ^{\card{\cB}}$, we have not chosen a preferred one.  We will refer to groups of the form $\integ[\cB]$ as lattices and say that $\integ[\cB]$ is the lattice generated by $\cB$.  For brevity, let us call such a set $\cB$ simply ``countable''.

In what follows, we write $\dual{A}=\Hom{\integ}{A}{\integ}$ for the $\integ$-linear dual of an Abelian group $A$. Note that for $A$ a finitely generated free Abelian group, $\dual{A}$ is again a finitely generated free Abelian group of the same rank as $A$. 

When $\cB$ is finite, we will canonically identify the double dual of a lattice $\integ[\cB]$, $\integ[\cB]^{**}$, with $\integ[\cB]$; that is, we will suppress the $\integ$-linear isomorphism $\delta\colon \integ[\cB]\to \integ[\cB]^{**}$, $\delta(b)=(f\mapsto f(b))$ from notation.  Since for $\cB$ countably infinite this map is not an isomorphism, at some points, extra vigilance will be needed in this case.

We will write $\evform{\blank}{\blank}\colon \dual{A}\cross A\to \integ$ for the \emph{evaluation form} $\evform{f}{a}=f(a)$ and abuse notation by also writing $\evform{\blank}{\blank}$ for the opposite form $\evform{\blank}{\blank}\colon A\cross \dual{A}\to \integ$, $\evform{a}{f}=f(a)$, since the context will always make it clear which is meant.  So, in particular, $\delta(b)=\evform{\blank}{b}$.

Let $\cB$ be a countable set and let $\lambda\colon \dual{\integ[\cB]}\to \integ[\cB]$ be a homomorphism\footnote{The choice of the dual of $\integ[\cB]$ as the domain for $\lambda$ may look surprising; we do it to more cleanly line up with conventions for the exchange matrix, which will be introduced in the following subsection.}.  We say $\lambda$ is \emph{skew-symmetric} if $\dual{\lambda}=-\delta\circ\lambda$. 

Associated to $\lambda$ is a skew-symmetric $\integ$-bilinear form $\canform{\blank}{\blank}{\lambda}\colon \dual{\integ[\cB]}\cross \dual{\integ[\cB]}\to \integ$, $\canform{\blank}{\blank}{\lambda}\defeq \evform{\lambda(\blank)}{\blank}$.  For $\cB$ finite, let $\dual{\cB}=\{ \dual{b} \mid b\in \cB \}$ be a dual basis for $\dual{\integ[\cB]}$, so that $\evform{\dual{b}}{c}=\delta_{bc}=\delta_{\dual{b}\dual{c}}$.

The Gram matrix of this form with respect to the basis $\dual{\cB}$ is equal to the matrix of $\lambda$ with respect to $\dual{\cB}$ and $\cB$: we define $\lambda_{\dual{b},c}$ by the equation $\lambda(\dual{b})=\sum_{c\in \cB} \lambda_{\dual{b},c}c$ and hence have $\canform{\dual{b}}{\dual{c}}{\lambda}= \lambda(\dual{b})(\dual{c})=\lambda_{\dual{b},c}$. Note that the sum defining $\lambda_{\dual{b},c}$ is finite since $\lambda$ is a homomorphism to $\integ[\cB]$, which consists of \emph{finite} sums of elements of $\cB$.  These matrices are skew-symmetric in the usual sense.

Recall that a group bicharacter $\Omega\colon G\cross G \to \units{\bK}$ is a map from $G\cross G$ to the units of $\bK$ satisfying
\[ \Omega(gh,k)=\Omega(g,k)\Omega(h,k)\ \text{and}\ \Omega(g,hk)=\Omega(g,h)\Omega(g,k) . \]
It is skew-symmetric if $\Omega(h,g)=\Omega(g,h)^{-1}$.

Now, any skew-symmetric bilinear form gives rise to a family of skew-symmetric bicharacters, by exponentiation.  Specifically, let $q^{\frac{1}{2}}\in \units{\bK}$ and define $\Omega_{q}^{\lambda}\colon \dual{\integ[\cB]}\cross \dual{\integ[\cB]}\to \units{\bK}$ by
\begin{equation}
	\label{eq:Omega-def}
	\Omega_{q}^{\lambda}(v,w)=q^{\frac{1}{2}\canform{v}{w}{\lambda}}.
\end{equation}
Here, in common with other authors, we write $q^{\frac{1}{2}}\in \units{\bK}$ as a shorthand for the choice of a pair $(v,q)$ of elements of $\units{\bK}$ such that $v^{2}=q$; that is, we have $(q^{\frac{1}{2}})^{2}=q$.

From $\lambda\colon \dual{\integ[\cB]} \to \integ[\cB]$ skew-symmetric, $q^{\frac{1}{2}}\in \units{\bK}$ and the associated skew-symmetric bicharacter $\Omega_{q}^{\lambda}$, we may construct a quantum torus, as the twisted group algebra.

\begin{definition}\label{d:q-torus}
	Let $\qtorus{q}{\lambda}{\cB}\defeq (\bK \dual{\integ[\cB]})^{\Omega_{q}^{\lambda}}$ be the $\bK$-algebra with underlying vector space 
	\[ \bK \dual{\integ[\cB]}=\operatorname{span}_{\bK}\{ x^{v} \mid v\in \dual{\integ[\cB]}\} \]
	and multiplication defined on basis elements by
	\[ x^{v}x^{w}=\Omega_{q}^{\lambda}(v,w)x^{v+w}, \]
	extended linearly.
\end{definition}

We will refer to the basis $\{ x^{v} \}$ as the canonical basis of $\qtorus{q}{\lambda}{\cB}$.  As is well-known, the bicharacter property ensures that $\qtorus{q}{\lambda}{\cB}$ is an associative algebra.  Note that 
\[ x^{v}x^{-v}=\Omega_{q}^{\lambda}(v,-v)x^{0}=1=\Omega_{q}^{\lambda}(-v,v)x^{0}=x^{-v}x^{v} \]
so that $x^{v}$ is invertible with inverse $x^{-v}$.

It follows from the skew-symmetry of $\Omega_{q}^{\lambda}$ that
\[ x^{w}x^{v}=\Omega_{q}^{\lambda}(w,v)x^{w+v}=\Omega_{q}^{\lambda}(v,w)^{-1}x^{v+w}=\Omega_{q}^{\lambda}(v,w)^{-2}x^{v}x^{w} \]
or equivalently
\[ x^{v}x^{w}=\Omega_{q}^{\lambda}(v,w)^{2}x^{w}x^{v} \]
so that the elements $x^{v}$ and $x^{w}$ quasi-commute. Comparing to \eqref{eq:Omega-def}, we see that the power of $q$ appearing in this quasi-commutation relation is $\canform{v}{w}{\lambda}$, without any factor of $\frac{1}{2}$. That is, the appearance of a square root of $q$ is only for the convenience of having our form $\canform{\blank}{\blank}{\lambda}$ exactly encode our quasi-commutation rules.

Combining the above observations, we see that $\qtorus{q}{\lambda}{\cB}$ is isomorphic to a quantum torus algebra in $\card{\cB}$ variables and its specialization at $q^{\frac{1}{2}}=1$, $\qtorus{1}{\lambda}{\cB}$, is isomorphic to a (commutative) Laurent polynomial algebra.

As we will see shortly, the canonical basis $\{x^{v}\}$ has a favourable invariance property.  Before explaining this, we make a link to the more common approach in the literature---notably \cite{GoodearlYakimovQCA}---of constructing based quantum tori. In order to do this, we will explicitly choose an indexing of $\cB$ and compare ordered monomials with respect to this with the canonical basis.

Given a bijection $\epsilon\colon \{1,\dotsc, \card{\cB}\} \to \dual{\cB}$, we may equip $\qtorus{q}{\lambda}{\cB}$ with a $\bK$-basis $\{ x^{\underline{v}} \mid \underline{v}\in \dual{\integ[\cB]} \}$
of standard monomials with respect to the enumeration $\epsilon$, where for $\underline{v}\in \dual{\integ[\cB]}$ expressed as $\sum_{i=1}^{\card{\cB}} v_{i}\epsilon(i)$ ($v_{i}\in \integ$) we define $x^{\underline{v}}\defeq \overrightarrow{\prod}_{i=1}^{\card{\cB}} x^{v_{i}\epsilon(i)}$; here the arrow reminds us that this is an ordered product. The following lemma tells us how to express $x^v$ in this basis.

\begin{lemma} For $v=\sum_{i} v_{i}\epsilon(i)$, we have
	\[ x^{v}=\left(\prod_{i<j} \Omega_{q}^{\lambda}(\epsilon(i),\epsilon(j))^{-v_{i}v_{j}} \right)(x^{\epsilon(1)})^{v_{1}}\dotsm (x^{\epsilon(n)})^{v_{n}}=\left(\prod_{i<j} \Omega_{q}^{\lambda}(\epsilon(i),\epsilon(j))^{-v_{i}v_{j}} \right)x^{\underline{v}} \]
\end{lemma}

\begin{proof} This follows by repeated application of the following ``expansion formula'' approach:
	\begin{align*} x^{v} & = x^{(v-v_{n}\epsilon(n))+v_{n}\epsilon(n)} \\
		& = \Omega_{q}^{\lambda}(v-v_{n}\epsilon(n),v_{n}\epsilon(n))^{-1}x^{v-v_{n}\epsilon(n)}x^{v_{n}\epsilon(n)} \\
		& = \Omega_{q}^{\lambda}(v_{n}\epsilon(n),v-v_{n}\epsilon(n))x^{v-v_{n}\epsilon(n)}(x^{\epsilon(n)})^{v_{n}}.\qedhere
	\end{align*}
\end{proof}

The coefficient appearing here is that referred to in \cite[\S 2.2]{GoodearlYakimovQCA} as the ``symmetrization scalar''; for future use, we define
\begin{equation}\label{eq:def-of-symm-scalar} \mathcal{S}_{q}^{\lambda,\epsilon}(v)=\prod_{i<j} \Omega_{q}^{\lambda}(\epsilon(i),\epsilon(j))^{-v_{i}v_{j}}
\end{equation}
for $v\in \dual{\integ[\cB]}$ decomposed as $v=\sum_{i} v_{i}\epsilon(i)$, noting that this depends on both the bicharacter $\Omega_{q}^{\lambda}$ and a choice of enumeration $\epsilon$.

Now, continuing to follow \cite[\S 2.2]{GoodearlYakimovQCA}, we have the invariance of the basis $\{ x^{v} \}$ referred to above.

\begin{lemma}\label{l:can-basis-invar} Let $\cB$, $\cC$ be countable sets such that $\card{\cB}=\card{\cC}$.  Let $\alpha\colon \dual{\integ[\cC]}\stackrel{\iso}{\to} \dual{\integ[\cB]}$ be an isomorphism of the duals of the associated lattices and define $\Omega_{q}^{\lambda,\alpha}=\Omega_{q}^{\lambda}\circ (\alpha \cross \alpha)$ the induced skew-symmetric bicharacter on $\dual{\integ[\cC]}$. 
	\begin{enumerate}[label=\textup{(\alph*)}]
		\item The map
		\[ \psi_{\alpha}\colon (\bK \dual{\integ[\cC]})^{\Omega_{q}^{\lambda,\alpha}}\to (\bK \dual{\integ[\cB]})^{\Omega_{q}^{\lambda}} \]
		defined by $\psi_{\alpha}(x^{w})=x^{\alpha(w)}$ is an isomorphism, with inverse given by $\psi_{\alpha}^{-1}(x^{w})=x^{\alpha^{-1}(w)}$.
		\item\label{l:can-basis-invar-b} We have 
		\[ \{ x^{\alpha(v)}=\psi_{\alpha}(x^{v}) \mid v \in \dual{\integ[\cC]} \} = \{ x^{v} \mid v\in \dual{\integ[\cB]} \}. \]
	\end{enumerate}
\end{lemma}

\begin{proof} For the first part, we have that
	\begin{align*}
		\psi_{\alpha}(x^{v}x^{w}) & = \psi_{\alpha}(\Omega_{q}^{\lambda,\alpha}(v,w)x^{v+w}) \\
		& = \Omega_{q}^{\lambda}(\alpha(v),\alpha(w))x^{\alpha(v+w)} \\
		& = \Omega_{q}^{\lambda}(\alpha(v),\alpha(w))x^{\alpha(v)+\alpha(w)} \\
		& = x^{\alpha(v)}x^{\alpha(w)} \\
		& =\psi_{\alpha}(x^{v})\psi_{\alpha}(x^{w}).
	\end{align*}
	The remainder is straightforward to check.
\end{proof}

Note the stronger claim in part~\ref{l:can-basis-invar-b}: that $\psi_{\alpha}$ maps the canonical basis of $(\bK \dual{\integ[\cC]})^{\Omega_{q}^{\lambda,\alpha}}$ precisely to that of $(\bK \dual{\integ[\cB]})^{\Omega_{q}^{\lambda}}$.  In particular, from the case $\cB=\cC$, we deduce invariance of the canonical basis $\{ x^{v} \}$ with respect to change of lattice (dual) basis.

\subsection{Toric frames and mutation}\label{ss:toric-frames}

Next, we will see that a choice of a skew-symmetric map $\lambda\colon \dual{\integ[\cB]}\to \integ[\cB]$ gives rise to a canonical \emph{toric frame}, in the sense of \cite{BZ-QCA}.  Let $\mathcal{F}(\qtorus{q}{\lambda}{\cB})$ denote the skew-field of fractions of $\qtorus{q}{\lambda}{\cB}$.

\begin{definition} Let $\lambda\colon \dual{\integ[\cB]}\to \integ[\cB]$ be skew-symmetric.  The function $M_{\lambda}\colon \dual{\integ[\cB]}\to \mathcal{F}(\qtorus{q}{\lambda}{\cB})$, $M_{\lambda}(\dual{b})=x^{\dual{b}}$ will be called the \emph{canonical toric frame} associated to $\lambda$.
\end{definition}

This is a toric frame: comparing with the definition as stated in \cite[Definition~2.2]{GoodearlYakimovQCA}, the $\bK$-span of the image of $M_{\lambda}$ is exactly $\qtorus{q}{\lambda}{\cB}$.  

Notice that this definition permits one to consider pre-composition by maps $\alpha\colon \dual{\integ[\cC]}\to \dual{\integ[\cB]}$.  If $\alpha$ is an isomorphism, then $M_{\lambda}\circ \alpha$ is again a toric frame. It follows from Lemma~\ref{l:can-basis-invar} that the $\bK$-span of the image of $M_{\lambda}\circ \alpha$ is $(\bK \dual{\integ[\cC]})^{\Omega_{q}^{\lambda,\alpha}}\iso \qtorus{q}{\lambda}{\cB}$ (corresponding to \cite[(2.11)]{GoodearlYakimovQCA}). 

Note too that the resulting toric frame is \emph{not} the correct candidate for mutation of $M_{\lambda}$.  From Lemma~\ref{l:can-basis-invar}, we see that $M_{\lambda}\circ \alpha$ is a ``monomial transformation'', reflecting a change of lattices, but this is at the tropical level.  Even for one-step mutation, cluster variables transform in a non-monomial way.  So this pre-composition cannot be the whole story: indeed, one needs a further ingredient, an automorphism of the field of fractions coming from the exchange matrix, of ``wall-crossing'' type.

Let $\ex\subseteq \cB$ and denote by $\iota\colon \integ[\ex]\inj \integ[\cB]$ the induced inclusion.  Fix $\beta\colon \integ[\ex]\to \dual{\integ[\cB]}$.  Denote by $\canform{\blank}{\blank}{\beta}\colon \integ[\ex]\cross\integ[\cB]\to \integ$ the $\integ$-bilinear form $\canform{v}{w}{\beta}=\evform{\beta(v)}{w}$.  The Gram matrix of this form is defined by $\beta(b)=\sum_{c\in \cB} \beta_{bc}\dual{c}$, a $\card{\ex}\cross \card{\cB}$ integer matrix.

Letting $\canform{\blank}{\blank}{\beta}^{p}\defeq \canform{\blank}{\blank}{\beta} \circ (\id \cross \iota)\colon \integ[\ex]\cross \integ[\ex] \to \integ$, we say that $\beta$ is \emph{skew-symmetrizable} if $\canform{\blank}{\blank}{\beta}$ is skew-symmetric (see Section~\ref{ss:forms-skew-sym} re this terminology).

From now on, we will assume $\cB$ is a countable set, $\ex \subseteq \cB$, $\beta\colon \integ[\ex]\to \dual{\integ[\cB]}$ is skew-symmetrizable and $\lambda\colon \dual{\integ[\cB]}\to \integ[\cB]$ is skew-symmetric.

\begin{definition}\label{d:compatible} We say that $\beta$ and $\lambda$ are \emph{compatible} if $\dual{\lambda}\circ \beta=\delta\circ \iota$.
\end{definition}

Note that if $\beta$ and $\lambda$ are compatible, then $\beta$ is necessarily injective.  This compatibility condition is dual to the usual one in terms of matrices, i.e.\ a condition on $B^{T}L$ for $B$ the matrix of $\beta$ and $L$ the matrix of $\lambda$, but is preferred in this approach due to the appearance of $\iota$ rather than its dual, where we would have $\dual{\beta}\circ \lambda$ equal to the canonical projection from $\dual{\integ[\cB]}$ onto $\dual{\integ[\ex]}$.  However the two are entirely equivalent so this is an \ae sthetic choice.

Also, for $k\in \ex$, we may re-write compatibility as
\[ \evform{\lambda(\dual{b})}{\beta(k)}=\evform{\dual{b}}{k}=\delta_{bk}. \]

For $\cB$ a countable set and $\integ[\cB]$ its associated lattice, we have a canonical submonoid $\nat[\cB]\subseteq \integ[\cB]$ obtained by taking the $\nat$-span of $\cB$; this submonoid is the discrete version of the canonical positive cone $\real_{+}[\cB]$ inside $\real[\cB]$.  

Given $b\in \integ[\cB]$ there is a unique decomposition $b=\bplus{b}{\cB}-\bminus{b}{\cB}$ with $\bplus{b}{\cB},\bminus{b}{\cB}\in \nat[\cB]$.  If $\ex \subseteq \cB$ and $b\in \integ[\ex]$, we similarly have $b=\bplus{b}{\ex}-\bminus{b}{\ex}$ with $\bplus{b}{\ex},\bminus{b}{\ex}\in \nat[\ex]$.  Furthermore, letting $\pi\colon \integ[\cB] \onto \integ[\ex]$ be the canonical splitting of $\iota$, we have $\pi([b]_{\pm}^{\cB})=[\pi(b)]_{\pm}^{\ex}$.

By identifying $\dual{\integ[\cB]}$ with $\integ[\dual{\cB}]$, we also obtain $\bplus{\dual{b}}{\dual{\cB}}$ and $\bminus{\dual{b}}{\dual{\cB}}$ in the same way.

Note that $\canform{b}{c}{\beta}=\beta(b)(c)$ implies that \[ \bplusminus{\beta(b)}{\dual{\cB}}=\sum_{\dual{c}\in \dual{\cB}} [\canform{b}{c}{\beta}]_{\pm}\dual{c},\] where for $n\in \integ$, $[n]_{+}\defeq \max\{n,0\}$ and $[n]_{-}\defeq [-n]_{+}=\max \{ 0,-n\}$ (noting that \hfill \hfill \linebreak $n=\bplus{n}{}-\bminus{n}{}$).  In the following and subsequently, we will use the assumption that $\beta$ is skew-symmetrizable to observe that for $k\in \ex$ we have $\beta(k)(k)=0$ and hence $\evform{\beta(k)}{k}=0$, or equivalently $\beta(k)\in \integ[\dual{\cB}\setminus \{\dual{k}\}]$.  It immmediately follows that $\bplusminus{\beta(k)}{\dual{\cB}}\in \nat[\dual{\cB}\setminus \{ \dual{k}\}]$. 

\begin{definition}\label{d:mut-of-index-sets} Let $\cB$ be a finite set and $\ex \subseteq \cB$. Fix $k\in \ex$ and let $\mu_{\cB}(k)$ denote a new element distinct from every element of $\cB$.  Then define
	\[ \mu_{k}(\cB) \defeq (\cB \setminus \{ k \}) \union \{ \mu_{\cB}(k) \}. \]
	Also, define
	\[ \mu_{k}(\ex) \defeq (\ex \setminus \{ k \}) \union \{ \mu_{\cB}(k) \} \]
	so that $\mu_{k}(\ex)=\mu_{k}(\cB)\setminus \ex^{c} \subseteq \mu_{k}(\cB)$.
\end{definition}

That is, mutation of our initial index set $\cB$ at an exchangeable (\ie mutable) index $k$ is achieved by removing $k$ and replacing it by a new (and, by fiat, different) indexing element, drawn from some unspecified universe of labels.  At this level, we are working by analogy with the exchange \emph{tree}, rather than the exchange graph: no equivalence relation is imposed, not even involutivity (see Lemma~\ref{l:F-mut-invol} below).

\begin{lemma}\label{l:F-isos} Let $k\in \ex$. There exist isomorphisms 
	\begin{align*}
		\mu_{k}^{+} & \colon \integ[\cB]\to \integ[\mu_{k}(\cB)], \mu_{k}^{+}(b)=\begin{cases} b+\evform{\bplus{\beta(k)}{\dual{\cB}}}{b}\mu_{\cB}(k) & \text{if}\ b\neq k \\ -\mu_{\cB}(k) & \text{if}\ b=k \end{cases} \quad\text{and} \\
		\mu_{k}^{-} & \colon \integ[\cB]\to \integ[\mu_{k}(\cB)], \mu_{k}^{-}(b)=\begin{cases} b+\evform{\bminus{\beta(k)}{\dual{\cB}}}{b}\mu_{\cB}(k) & \text{if}\ b\neq k \\ -\mu_{\cB}(k) & \text{if}\ b=k \end{cases} 
	\end{align*}
	whose inverses are, respectively,
	\begin{align*}
		\bar{\mu}_{k}^{+} & \colon \integ[\mu_{k}(\cB)]\to \integ[\cB], \bar{\mu}_{k}^{+}(b)=\begin{cases} b+\evform{\bplus{\beta(k)}{\dual{\cB}}}{b}k & \text{if}\ b\neq \mu_{\cB}(k) \\ -k & \text{if}\ b=\mu_{\cB}(k) \end{cases} \quad\text{and}\\
		\bar{\mu}_{k}^{-} & \colon \integ[\mu_{k}(\cB)]\to \integ[\cB], \bar{\mu}_{k}^{-}(b)=\begin{cases} b+\evform{\bminus{\beta(k)}{\dual{\cB}}}{b}k & \text{if}\ b\neq \mu_{\cB}(k) \\ -k & \text{if}\ b=\mu_{\cB}(k) \end{cases} .
	\end{align*}
	Furthermore, restricting to $\mu_{k}(\ex)$ (respectively, $\ex$), we have isomorphisms $\mu_{k}^{\pm}\colon \integ[\ex]\to \integ[\mu_{k}(\ex)]$ (respectively, $\bar{\mu}_{k}^{\pm}\colon \integ[\mu_{k}(\ex)]\to \integ[\ex]$).
\end{lemma}

\begin{proof} We immediately see that $(\bar{\mu}_{k}^{\pm}\circ \mu_{k}^{\pm})(k)=k$ and $(\mu_{k}^{\pm}\circ \bar{\mu}_{k}^{\pm})(\mu_{\cB}(k))=\mu_{\cB}(k)$.  So, consider $b\in \cB\setminus \{ k \}$.  Then
	\begin{align*} (\bar{\mu}_{k}^{+} \circ \mu_{k}^{+})(b) & = \bar{\mu}_{k}^{+}(b+\evform{\bplus{\beta(k)}{\dual{\cB}}}{b}\mu_{\cB}(k)) \\
		& = \bar{\mu}_{k}^{+}(b)-\evform{\bplus{\beta(k)}{\dual{\cB}}}{b}k \\
		& = (b+\evform{\bplus{\beta(k)}{\dual{\cB}}}{b}k)-\evform{\bplus{\beta(k)}{\dual{\cB}}}{b}k \\
		& =b.
	\end{align*}
	The remaining case is similar.
\end{proof}

By the comment prior to the lemma, we could also write
\begin{align*}
	\mu_{k}^{+} & \colon \integ[\cB]\to \integ[\mu_{k}(\cB)], \mu_{k}^{+}(b)=\begin{cases} b+[\canform{k}{b}{\beta}]_{+}\mu_{\cB}(k) & \text{if}\ b\neq k \\ -\mu_{\cB}(k) & \text{if}\ b=k \end{cases} \quad\text{and} \\
	\mu_{k}^{-} & \colon \integ[\cB]\to \integ[\mu_{k}(\cB)], \mu_{k}^{-}(b)=\begin{cases} b+[\canform{k}{b}{\beta}]_{-}\mu_{\cB}(k) & \text{if}\ b\neq k \\ -\mu_{\cB}(k) & \text{if}\ b=k \end{cases} 
\end{align*}
which allows a more direct comparison to the formul\ae\ for tropical mutation (\cf \cite{FockGoncharov}).  We choose to label $\mu$ by $+$ or $-$ in line with whether $\bplus{\ }{}$ or $\bminus{\ }{}$ occurs in the definition of the map.

With respect to the chosen bases, the matrices for these maps are those denoted $F_{\pm}$ in the cluster algebra literature (\cf \cite{FZ-CA1} and others).

Similarly, we have a dual\footnote{Strictly speaking, we restrict the duals of these isomorphisms to obtain those above; see Lemma~\ref{l:EF-dual}.} pair of isomorphisms, as follows.

\begin{lemma}\label{l:E-isos} Let $\cB$ be a finite set, $\ex \subseteq \cB$ and fix $k\in \ex$.
	
	Then there exist isomorphisms 
	\begin{align*}
		\mu_{k}^{+} & \colon \dual{\integ[\cB]}\to \dual{\integ[\mu_{k}(\cB)]}, \mu_{k}^{+}(\dual{b})=\begin{cases} \dual{b} & \text{if}\ \dual{b}\neq \dual{k} \\ \bplus{\beta(k)}{\dual{\cB}}-\dual{\mu_{\cB}(k)} & \text{if}\ \dual{b}=\dual{k} \end{cases} \quad \text{and} \\
		\mu_{k}^{-} & \colon \dual{\integ[\cB]}\to \dual{\integ[\mu_{k}(\cB)]}, \mu_{k}^{-}(\dual{b})=\begin{cases} \dual{b} & \text{if}\ \dual{b}\neq \dual{k} \\ \bminus{\beta(k)}{\dual{\cB}}-\dual{\mu_{\cB}(k)} & \text{if}\ \dual{b}=\dual{k} \end{cases} 
	\end{align*}
	whose inverses are, respectively,
	\begin{align*}
		\bar{\mu}_{k}^{+} & \colon \dual{\integ[\mu_{k}(\cB)]}\to \dual{\integ[\cB]}, \bar{\mu}_{k}^{+}(\dual{b})=\begin{cases} \dual{b} & \text{if}\ \dual{b}\neq \dual{\mu_{\cB}(k)} \\ \bplus{\beta(k)}{\dual{\cB}}-\dual{k} & \text{if}\ \dual{b}=\dual{\mu_{\cB}(k)} \end{cases} \quad\text{and}\\
		\bar{\mu}_{k}^{-} & \colon \dual{\integ[\mu_{k}(\cB)]}\to \dual{\integ[\cB]}, \bar{\mu}_{k}^{-}(\dual{b})=\begin{cases} \dual{b} & \text{if}\ \dual{b}\neq \dual{\mu_{\cB}(k)} \\ \bminus{\beta(k)}{\dual{\cB}}-\dual{k} & \text{if}\ \dual{b}=\dual{\mu_{\cB}(k)} \end{cases} .
	\end{align*}
\end{lemma}

Note that here we are using that $\bplusminus{\beta(k)}{\dual{\cB}}\in \nat[\dual{\cB}\setminus \{ \dual{k}\}]=\nat[\dual{\mu_{k}(\cB)}\setminus \{ \dual{\mu_{\cB}(k)} \}]$. 

\begin{proof}
	The only non-trivial check is 
	\begin{align*} (\bar{\mu}_{k}^{\pm}\circ \mu_{k}^{\pm})(\dual{k}) & = \bar{\mu}_{k}^{\pm}(\bplusminus{\beta(k)}{\dual{\cB}}-\dual{\mu_{\cB}(k)}) \\ & = \bplusminus{\beta(k)}{\dual{\cB}}-(\bplusminus{\beta(k)}{\dual{\cB}}-\dual{k})\\ & = \dual{k} \end{align*}
	where the second equality holds because $\bplusminus{\beta(k)}{\dual{\cB}}\in \nat[\dual{\mu_{k}(\cB)}\setminus\{ \dual{\mu_{\cB}(k)}\}]$ so $\mu_{k}^{\pm}$ acts as the identity on this.
\end{proof}

For future use, we formally record the duality claim made above.

\begin{lemma}\label{l:EF-dual}
	The isomorphisms $\mu_{k}^{\pm}\colon \integ[\cB]\to \integ[\mu_{k}(\cB)]$ and $\bar{\mu}_{k}^{\pm}\colon \dual{\integ[\mu_{k}(\cB)]}\to \dual{\integ[\cB]}$ satisfy
	\[ \evform{\dual{b}}{\mu_{k}^{\pm}(c)}=\evform{\bar{\mu}_{k}^{\pm}(\dual{b})}{c} \]
	where $\dual{b}\in \dual{\mu_{k}(\cB)}$ and $c\in \cB$.
\end{lemma}

\begin{proof}
	First, assume $c\neq k$. We compute
	\begin{align*}
		\evform{\dual{b}}{\mu_{k}^{\pm}(c)} & = \evform{\dual{b}}{c+\evform{\bplusminus{\beta(k)}{\dual{\cB}}}{c}\mu_{\cB}(k)} \\
		& = \begin{cases} \delta_{bc} & \text{if}\ \dual{b}\neq \dual{\mu_{\cB}(k)} \\ \evform{\bplusminus{\beta(k)}{\dual{\cB}}}{c} & \text{if}\ \dual{b}=\dual{\mu_{\cB}(k)} \end{cases}
	\end{align*}
	and
	\begin{align*}
		\evform{\bar{\mu}_{k}^{\pm}(\dual{b})}{c} & = \begin{cases} \delta_{bc} & \text{if}\ \dual{b}\neq \dual{\mu_{\cB}(k)} \\ \evform{\bplusminus{\beta(k)}{\dual{\cB}}-\dual{k}}{c} & \text{if}\ \dual{b}=\dual{\mu_{\cB}(k)} \end{cases} \\ 
		& = \begin{cases} \delta_{bc} & \text{if}\ \dual{b}\neq \dual{\mu_{\cB}(k)} \\ \evform{\bplusminus{\beta(k)}{\dual{\cB}}}{c} & \text{if}\ \dual{b}=\dual{\mu_{\cB}(k)} \end{cases}   .
	\end{align*}
	For $c=k$, we have
	\begin{align*} \evform{\dual{b}}{\mu_{k}^{\pm}(k)} & = \evform{\dual{b}}{-\mu_{\cB}(k)} \\ 
		& = -\delta_{b\mu_{\cB}(k)} \\ & = -\delta_{bk} \\ & = \begin{cases} 0=\evform{\dual{b}}{k} & \text{if}\ \dual{b}\neq \dual{\mu_{\cB}(k)} \\ \evform{\bplusminus{\beta(k)}{\dual{\cB}}-\dual{k}}{k} & \text{if}\ \dual{b}=\dual{\mu_{\cB}(k)} \end{cases} \\ & = \evform{\bar{\mu}_{k}^{\pm}(\dual{b})}{k}
	\end{align*}
	so we have the claim.
\end{proof}

Recall that associated to $\beta\colon \integ[\ex]\to \dual{\integ[\cB]}$, we have the form $\canform{\blank}{\blank}{\beta}\colon \integ[\cB]\cross \integ[\ex] \to \integ$, $\canform{v}{w}{\beta}=\evform{\beta(v)}{w}$ with $\canform{\blank}{\blank}{\beta}^{p}$ skew-symmetric.  We may use the above isomorphisms from Lemma~\ref{l:F-isos} to mutate these forms.

\begin{definition}
	Let $k\in \ex$.  We define
	\[ \mu_{k}\canform{\blank}{\blank}{\beta}\colon \integ[\mu_{k}(\cB)]\cross \integ[\mu_{k}(\ex)]\to \integ, \mu_{k}\canform{\blank}{\blank}{\beta}=\canform{\blank}{\blank}{\beta}\circ (\bar{\mu}_{k}^{+} \cross \bar{\mu}_{k}^{+}). \]
\end{definition}

It follows immediately from this definition that $\mu_{k}\canform{\blank}{\blank}{\beta}$ is also skew-symmetrizable.  

From this, we can reconstruct a map $\mu_{k}\beta$ as follows.

\begin{definition} Let $k\in \ex$.  Define $\mu_{k}\beta\colon \integ[\mu_{k}(\ex)]\to \dual{\integ[\mu_{k}(\cB)]}$ by $\mu_{k}\beta(b)=\mu_{k}\canform{b}{\blank}{\beta}$.
\end{definition}

By construction, we have $\canform{\blank}{\blank}{\mu_{k}\beta}=\mu_{k}\canform{\blank}{\blank}{\beta}$ and we will use both expressions as is most relevant to the context.  Furthermore, we see that
\[ \mu_{k}\canform{b}{c}{\beta}=\evform{\beta(\bar{\mu}_{k}^{+}(b))}{\bar{\mu}_{k}^{+}(c)}=\evform{(\mu_{k}^{+}\circ \beta \circ \bar{\mu}_{k}^{+})(b)}{c}\]
by a similar argument to that in Lemma~\ref{l:EF-dual}.  Hence
\begin{equation}\label{eq:beta-mut} \mu_{k}\beta=\mu_{k}^{+}\circ \beta \circ \bar{\mu}_{k}^{+}\end{equation}
which is the equivalent expression to $\mu_{k}(B)=E_{+}BF_{+}$ in the matrix approach.

For later use, we explicitly compute the values of $\mu_{k}\beta$ on the basis $\mu_{k}(\ex)$ in terms of $\beta$.

\begin{lemma}\label{l:mut-form-values}
	Let $k\in \ex$.  Then for $b\in \mu_{k}(\ex)$ and $c\in \mu_{k}(\cB)$ we have
	\[ \canform{b}{c}{\mu_{k}\beta}=\mu_{k}\beta(b)(c) = \begin{cases} \canform{b}{c}{\beta}+\canform{b}{k}{\beta}\bplus{\canform{k}{c}{\beta}}{}-\canform{c}
		{k}{\beta}\bplus{\canform{k}{b}{\beta}}{} & \text{if}\ b,c\neq \mu_{\cB}(k) \\ -\canform{b}{k}{\beta} & \text{if}\ b\neq \mu_{\cB}(k),c=\mu_{\cB}(k) \\ -\canform{k}{c}{\beta} & \text{if}\ b=\mu_{\cB}(k),c\neq \mu_{\cB}(k) \\ 0 & \text{if}\ b=c=\mu_{\cB}(k) \end{cases}. \]
\end{lemma}

\begin{proof} We will use the alternative form for the isomorphisms, given immediately after Lemma~\ref{l:F-isos}.  For $b\in \mu_{k}(\ex)$, $b\neq \mu_{\cB}(k)$ and $c\in \mu_{k}(\cB)$ we have
	\begin{align*}
		\mu_{k}\beta(b)(c) & = \mu_{k}\canform{b}{c}{\beta} \\ 
		& = \canform{\bar{\mu}_{k}^{+}(b)}{\bar{\mu}_{k}^{+}(c)}{\beta} \\
		& = \begin{cases} \canform{b+\bplus{\canform{k}{b}{\beta}}{}k}{c+\bplus{\canform{k}{c}{\beta}}{}k}{\beta} & \text{if}\ c\neq \mu_{\cB}(k) \\ \canform{b+\bplus{\canform{k}{b}{\beta}}{}k}{-k}{\beta} & \text{if}\ c=\mu_{\cB}(k) \end{cases} \\
		& = \begin{cases} \canform{b}{c}{\beta}+\bplus{\canform{k}{b}{\beta}}{}\canform{k}{c}{\beta}+\bplus{\canform{k}{c}{\beta}}{}\canform{b}{k}{\beta}+\bplus{\canform{k}{c}{\beta}}{}\bplus{\canform{k}{b}{\beta}}{}\canform{k}{k}{\beta} 
			& \text{if}\ c\neq \mu_{\cB}(k) \\ -\canform{b}{k}{\beta} & \text{if}\ c=\mu_{\cB}(k) \end{cases} \\
		& = \begin{cases} \canform{b}{c}{\beta}+\canform{b}{k}{\beta}\bplus{\canform{k}{c}{\beta}}{}-\canform{c}
			{k}{\beta}\bplus{\canform{k}{b}{\beta}}{}
			& \text{if}\ c\neq \mu_{\cB}(k) \\ -\canform{b}{k}{\beta} & \text{if}\ c=\mu_{\cB}(k) \end{cases} .
	\end{align*}
	Similarly, for $b=\mu_{\cB}(k)$ and $c\in \mu_{k}(\cB)$,
	\begin{align*}
		\mu_{k}\beta(\mu_{\cB}(k))(c) & = \mu_{k}\canform{\mu_{\cB}(k)}{c}{\beta} \\ 
		& = \canform{\mu_{k}^{+}(\mu_{\cB}(k))}{\mu_{k}^{+}(c)}{\beta} \\
		& = \begin{cases} \canform{-k}{c+\bplus{\canform{k}{c}{\beta}}{}k}{\beta} & \text{if}\ c\neq \mu_{\cB}(k) \\ \canform{-k}{-k}{\beta} & \text{if}\ c=\mu_{\cB}(k) \end{cases} \\
		& = \begin{cases} -\canform{k}{c}{\beta}
			& \text{if}\ c\neq \mu_{\cB}(k) \\ 0 & \text{if}\ c=\mu_{\cB}(k) \end{cases} \\
			& = -\beta(k)(c)
	\end{align*}
	noting that the equality $\mu_{k}\beta(\mu_{\cB}(k))=-\beta(k)$ will be useful later.
\end{proof}

The following two results give expressions for compositions of mutations.

\begin{lemma}\label{l:F-mut-invol}
	Let $k\in \ex$.  Let $\nu\colon \cB\to \mu_{\mu_{\cB}(k)}(\mu_{k}(\cB))$ be the bijection defined by $\nu(b)=b$ for $b\neq k$ and $\nu(k)=\mu_{\mu_{\cB}(k)}(\mu_{\cB}(k))$ and also write $\nu$ for the induced isomorphism of $\integ[\cB]$ with $\integ[\mu_{\mu_{\cB}(k)}(\mu_{k}(\cB))]$.
	
	Then the maps $\mu_{k}^{\pm}\colon \integ[\cB]\to\integ[\mu_{k}(\cB)]$ and $\mu_{\mu_{\cB}(k)}^{\mp}\colon \integ[\mu_{k}(\cB)]\to \integ[\mu_{\mu_{\cB}(k)}(\mu_{k}(\cB))]$ satisfy
	\[ \mu_{\mu_{\cB}(k)}^{\mp}\circ \mu_{k}^{\pm}=\nu. \]
	Equivalently, $\mu_{\mu_{\cB}(k)}^{\pm}=\nu \circ \bar{\mu}_{k}^{\mp}$ and $\bar{\mu}_{\mu_{\cB}(k)}^{\pm}=\mu_{k}^{\mp}\circ \nu^{-1}$.
\end{lemma}

\begin{proof} For $b\in \mu_{k}(\cB)$ we have
	\begin{align*}
		\mu_{\mu_{\cB}(k)}^{\pm}(b) & =\begin{cases} b+\bplusminus{\canform{\mu_{\cB}(k)}{b}{\mu_{k}\beta}}{}\mu_{\mu_{\cB}(k)}(\mu_{\cB}(k)) & \text{if}\ b\neq \mu_{\cB}(k) \\ -\mu_{\mu_{\cB}(k)}(\mu_{\cB}(k)) & \text{if}\ b=\mu_{\cB}(k) \end{cases} \\ 
		& = \begin{cases}
			b+\bplusminus{-\canform{k}{b}{\beta}}{}\mu_{\mu_{\cB}(k)}(\mu_{\cB}(k)) & \text{if}\ b\neq \mu_{\cB}(k) \\ -\mu_{\mu_{\cB}(k)}(\mu_{\cB}(k)) & \text{if}\ b=\mu_{\cB}(k) \end{cases} \\
		& = \begin{cases}
			\nu(b)+\bminusplus{\canform{k}{b}{\beta}}{}\nu(k) & \text{if}\ b\neq \mu_{\cB}(k) \\ -\nu(k) & \text{if}\ b=\mu_{\cB}(k) \end{cases} \\
		& = (\nu\circ \bar{\mu}_{k}^{\mp})(b). \qedhere
	\end{align*}
\end{proof}

\begin{proposition} The maps $\mu_{k}^{\pm}\colon \integ[\cB]\to \integ[\mu_{k}(\cB)]$ and $\bar{\mu}_{k}^{\pm}\colon \integ[\mu_{k}(\cB)]\to \integ[\cB]$ satisfy
	\[ (\bar{\mu}_{k}^{\mp}\circ \mu_{k}^{\pm})(b) = b\mp \canform{k}{b}{\beta}k.
	\]
\end{proposition}

\begin{proof}
	Let $b\in \cB$. Then
	\begin{align*}
		(\bar{\mu}_{k}^{-}\circ \mu_{k}^{+})(b) & = \begin{cases} \bar{\mu}_{k}^{-}(b+\bplus{\canform{k}{b}{\beta}}{}\mu_{\cB}(k)) & \text{if}\ b\neq k \\ \bar{\mu}_{k}^{-}(-\mu_{\cB}(k)) & \text{if}\ b=k \end{cases} \\
		& = \begin{cases} (b+\bminus{\canform{k}{b}{\beta}}{}k)+\bplus{\canform{k}{b}{\beta}}{}(-k) & \text{if}\ b\neq k \\ k & \text{if}\ b=k \end{cases} \\
		& = \begin{cases} b-(\bplus{\canform{k}{b}{\beta}}{}-\bminus{\canform{k}{b}{\beta}}{})k & \text{if}\ b\neq k \\ k & \text{if}\ b=k \end{cases} \\ 
		& = \begin{cases} b-\canform{k}{b}{\beta}k & \text{if}\ b\neq k \\ k & \text{if}\ b=k \end{cases} \\
		& = b-\canform{k}{b}{\beta}k.
	\end{align*}
	Similarly, $(\bar{\mu}_{k}^{+}\circ \mu_{k}^{-})(b)= b+\canform{k}{b}{\beta}k$. 
\end{proof}

Recall that $\mu_{k}\canform{\blank}{\blank}{\beta}=\canform{\blank}{\blank}{\beta}\circ (\bar{\mu}_{k}^{+} \cross \bar{\mu}_{k}^{+})$.  The next result shows ``sign invariance'', \ie that using $\bar{\mu}_{k}^{-}$ in the definition yields the same form.

\begin{lemma}\label{l:sign-invar}
	We have 
	\[ \mu_{k}\canform{\blank}{\blank}{\beta}=\canform{\blank}{\blank}{\beta}\circ (\bar{\mu}_{k}^{-} \cross \bar{\mu}_{k}^{-}). \]
\end{lemma}

\begin{proof}
	Let $b\in\mu_{k}(\ex)$, $c\in \mu_{k}(\cB)$.  Then
	\begin{align*}
		(\canform{\blank}{\blank}{\beta}\circ (\bar{\mu}_{k}^{-} \cross \bar{\mu}_{k}^{-}))(b,c) & = \begin{cases} \canform{b+\bminus{\canform{k}{b}{\beta}}{}k}{c+\bminus{\canform{k}{c}{\beta}}{}k}{\beta} & \text{if}\ b,c\neq \mu_{\cB}(k) \\ \canform{b+\bminus{\canform{k}{b}{\beta}}{}k}{-k}{\beta} & \text{if}\ b\neq \mu_{\cB}(k),c=\mu_{\cB}(k) \\ \canform{-k}{c+\bminus{\canform{k}{c}{\beta}}{}k}{\beta} & \text{if}\ b=\mu_{\cB}(k),c\neq \mu_{\cB}(k) \\ \canform{-k}{-k}{\beta}=0 & \text{if}\ b=c=\mu_{\cB}(k) \end{cases} \\
		& = \begin{cases} \canform{b}{c}{\beta}+\bminus{\canform{k}{b}{\beta}}{}\canform{k}{c}{\beta}-\bminus{\canform{k}{c}{\beta}}{}\canform{k}{b}{\beta} & \text{if}\ b,c\neq \mu_{\cB}(k) \\ -\canform{b}{k}{\beta} & \text{if}\ b\neq \mu_{\cB}(k),c=\mu_{\cB}(k) \\ -\canform{k}{c}{\beta} & \text{if}\ b=\mu_{\cB}(k),c\neq \mu_{\cB}(k) \\ 0 & \text{if}\ b=c=\mu_{\cB}(k) \end{cases} \\
		& = \begin{cases} \canform{b}{c}{\beta}+(\bplus{\canform{k}{b}{\beta}}{}-\canform{k}{b}{\beta})\canform{k}{c}{\beta} & \\ \qquad -(\bplus{\canform{k}{c}{\beta}}{}-\canform{k}{c}{\beta})\canform{k}{b}{\beta}
			 & \text{if}\ b,c\neq \mu_{\cB}(k) \\ -\canform{b}{k}{\beta} & \text{if}\ b\neq \mu_{\cB}(k),c=\mu_{\cB}(k) \\ -\canform{k}{c}{\beta} & \text{if}\ b=\mu_{\cB}(k),c\neq \mu_{\cB}(k) \\ 0 & \text{if}\ b=c=\mu_{\cB}(k) \end{cases} \\
		& = \mu_{k}\canform{b}{c}{\beta}
	\end{align*}
	by Lemma~\ref{l:mut-form-values}.
\end{proof}

As one expects, the mutation of forms (or equivalently, of the maps) is involutive, up to an identification of $\mu_{\mu_{\cB}(k)}(\mu_{k}(\cB))$ with $\cB$.

\begin{proposition}\label{p:beta-form-mut-invol} Let $k\in \ex$.  Let $\nu\colon \cB\to \mu_{\mu_{\cB}(k)}(\mu_{k}(\cB))$ be the bijection defined by $\nu(b)=b$ for $b\neq k$ and $\nu(k)=\mu_{\mu_{\cB}(k)}(\mu_{\cB}(k)))$ and also write $\nu$ for the induced isomorphism of $\integ[\cB]$ with $\integ[\mu_{\mu_{\cB}(k)}(\mu_{k}(\cB))]$.
	
	Then
	\[ \mu_{\mu_{\cB}(k)}\mu_{k}\canform{\blank}{\blank}{\beta}=\canform{\blank}{\blank}{\beta}\circ (\nu^{-1} \cross \nu^{-1}). \]
\end{proposition}

\begin{proof}
	By Lemmas~\ref{l:F-mut-invol} and \ref{l:sign-invar}, we have
	\begin{align*} \mu_{\mu_{\cB}(k)}\mu_{k}\canform{\blank}{\blank}{\beta} & = (\canform{\blank}{\blank}{\beta}\circ (\bar{\mu}_{k}^{-}\cross \bar{\mu}_{k}^{-}))\circ (\bar{\mu}_{\mu_{\cB}(k)}^{+} \cross \bar{\mu}_{\mu_{\cB}(k)}^{+}) \\
		& = (\canform{\blank}{\blank}{\beta}\circ (\bar{\mu}_{k}^{-}\cross \bar{\mu}_{k}^{-}))\circ (\mu_{k}^{-} \cross \mu_{k}^{-})\circ (\nu^{-1} \cross \nu^{-1}) \\
		& = (\canform{\blank}{\blank}{\beta}\circ (\nu^{-1} \cross \nu^{-1}). \qedhere
	\end{align*}
\end{proof}

Now, we may similarly mutate $\lambda\colon \dual{\integ[\cB]}\to \integ[\cB]$.  As above, associated to $\lambda\colon \dual{\integ[\cB]}\to \integ[\cB]$, we have the skew-symmetric form $\canform{\blank}{\blank}{\lambda}\colon \dual{\integ[\cB]}\cross \dual{\integ[\cB]} \to \integ$, $\canform{v}{w}{\lambda}=\evform{\lambda(v)}{w}$.  This time, we use the isomorphisms from Lemma~\ref{l:E-isos} to mutate these forms.

\begin{definition}
	Let $k\in \ex$.  We define
	\[ \mu_{k}\canform{\blank}{\blank}{\lambda}\colon \dual{\integ[\mu_{k}(\cB)]}\cross \dual{\integ[\mu_{k}(\cB)]}\to \integ, \mu_{k}\canform{\blank}{\blank}{\lambda}=\canform{\blank}{\blank}{\lambda}\circ (\bar{\mu}_{k}^{+} \cross \bar{\mu}_{k}^{+}). \]
\end{definition}

It follows immediately from this definition that $\mu_{k}\canform{\blank}{\blank}{\lambda}$ is also skew-symmetric.  Thus we have a quantum torus $\qtorus{q}{\mu_{k}\lambda}{\mu_{k}(\cB)}=(\bK \dual{\integ[\mu_{k}(\cB)]})^{\Omega_{q}^{\mu_{k}\lambda}}$.  Below, we will relate this torus to $\qtorus{q}{\lambda}{\cB}$ but first we need to see that the mutation of $\beta$ and $\lambda$ behave suitably with respect to one another.

Firstly, from the mutated form, we can reconstruct a map $\mu_{k}\lambda$ as follows.

\begin{definition} Let $k\in \ex$.  Define $\mu_{k}\lambda\colon \dual{\integ[\mu_{k}(\cB)]}\to \integ[\mu_{k}(\cB)]$ by $\mu_{k}\lambda(\dual{b})=\mu_{k}\canform{\dual{b}}{\blank}{\lambda}$.
\end{definition}

By construction, we have $\canform{\blank}{\blank}{\mu_{k}\lambda}=\mu_{k}\canform{\blank}{\blank}{\lambda}$ and we will use both expressions as is most relevant to the context.  Furthermore, we see that
\[ \mu_{k}\canform{\dual{b}}{\dual{c}}{\lambda}=\evform{\lambda(\bar{\mu}_{k}^{+}(\dual{b})}{\bar{\mu}_{k}^{+}(\dual{c})}=\evform{(\mu_{k}^{+}\circ \lambda \circ \bar{\mu}_{k}^{+})(\dual{b})}{\dual{c}}\]
by Lemma~\ref{l:EF-dual}.  Hence
\begin{equation}\label{eq:lambda-mut} \mu_{k}\lambda=\mu_{k}^{+}\circ \lambda \circ \bar{\mu}_{k}^{+}.\end{equation}

\begin{lemma}\label{l:lambda-form-mut}
	Let $k\in \ex$ and $\dual{b},\dual{c}\in \dual{\integ[\mu_{k}(\cB)]}$. We have
	\[ \canform{\dual{b}}{\dual{c}}{\mu_{k}\lambda}=\mu_{k}\lambda(\dual{b})(\dual{c})	
	= \begin{cases}
		\canform{\dual{b}}{\dual{c}}{\lambda} & \text{if}\ \dual{b},\dual{c}\neq \dual{\mu_{\cB}(k)} \\
			\canform{\dual{k}}{\dual{b}}{\lambda}-\canform{\bplus{\beta(k)}{\dual{\cB}}}{\dual{b}}{\lambda} & \text{if}\ \dual{b}\neq \dual{\mu_{\cB}(k)},\dual{c}=\dual{\mu_{\cB}(k)} \\
			-\canform{\dual{k}}{\dual{c}}{\lambda}+\canform{\bplus{\beta(k)}{\dual{\cB}}}{\dual{c}}{\lambda} & \text{if}\ \dual{b}=\dual{\mu_{\cB}(k)},\dual{c}\neq \dual{\mu_{\cB}(k)} \\
		\canform{\dual{k}}{\dual{k}}{\lambda}=0 & \text{if}\ \dual{b}=\dual{c}=\dual{\mu_{\cB}(k)} 
	\end{cases} .
	 \]
\end{lemma}

\begin{proof}
	We have
	\begin{align*}
		\mu_{k}\lambda(\dual{b})(\dual{c}) & = \canform{\bar{\mu}_{k}^{+}(\dual{b})}{\bar{\mu}_{k}^{+}(\dual{c})}{\lambda} \\
		& = \begin{cases}
			\canform{\dual{b}}{\dual{c}}{\lambda} & \text{if}\ \dual{b},\dual{c}\neq \dual{\mu_{\cB}(k)} \\
			\canform{\dual{b}}{\bplus{\beta(k)}{\dual{\cB}}-\dual{k}}{\lambda} & \text{if}\ \dual{b}\neq \dual{\mu_{\cB}(k)},\dual{c}=\dual{\mu_{\cB}(k)} \\
			\canform{\bplus{\beta(k)}{\dual{\cB}}-\dual{k}}{\dual{c}}{\lambda} & \text{if}\ \dual{b}=\dual{\mu_{\cB}(k)},\dual{c}\neq \dual{\mu_{\cB}(k)} \\
			\canform{\bplus{\beta(k)}{\dual{\cB}}-\dual{k}}{\bplus{\beta(k)}{\dual{\cB}}-\dual{k}}{\lambda} & \text{if}\ \dual{b}=\dual{c}=\dual{\mu_{\cB}(k)} 
		\end{cases} \\
		& = \begin{cases}
			\canform{\dual{b}}{\dual{c}}{\lambda} & \text{if}\ \dual{b},\dual{c}\neq \dual{\mu_{\cB}(k)} \\
			\canform{\dual{b}}{\bplus{\beta(k)}{\dual{\cB}}}{\lambda}-\canform{\dual{b}}{\dual{k}}{\lambda} & \text{if}\ \dual{b}\neq \dual{\mu_{\cB}(k)},\dual{c}=\dual{\mu_{\cB}(k)} \\
			\canform{\bplus{\beta(k)}{\dual{\cB}}}{\dual{c}}{\lambda}-\canform{\dual{k}}{\dual{c}}{\lambda} & \text{if}\ \dual{b}=\dual{\mu_{\cB}(k)},\dual{c}\neq \dual{\mu_{\cB}(k)} \\
			0=\canform{\dual{k}}{\dual{k}}{\lambda} & \text{if}\ \dual{b}=\dual{c}=\dual{\mu_{\cB}(k)} 
		\end{cases} .
	\end{align*}
\end{proof}

\begin{lemma}\label{l:mut-compatible}
	The maps $\mu_{k}\beta$ and $\mu_{k}\lambda$ are compatible.
\end{lemma}

\begin{proof} Recall that compatibility of $\beta$ and $\lambda$ means that $\dual{\lambda}\circ \beta=\delta_{\cB}\circ \iota_{\cB}$ for $\delta_{\cB}\colon \integ[\cB]\to \integ[\cB]^{**}$ and $\iota_{\cB}\colon \integ[\ex]\inj \integ[\cB]$.

	For $k\in \ex$, letting $\iota_{\mu_{k}(\cB)}\colon \integ[\mu_{k}(\ex)]\inj \integ[\mu_{k}(\cB)]$, one may check from the definitions that $\mu_{k}^{\pm}\circ \iota_{\cB}=\iota_{\mu_{k}(\cB)} \circ \mu_{k}^{\pm}$; this expresses that the mutation maps on $\integ[\ex]$ are the restrictions of those on $\integ[\cB]$. The analogous expression for $\delta_{\cB}$ and $\delta_{\mu_{k}(\cB)}$ holds by Lemma~\ref{l:EF-dual}.
	Then we have 
	\begin{align*} \dual{(\mu_{k}\lambda)}\circ \mu_{k}\beta & = \dual{(\mu_{k}^{+}\circ \lambda \circ \bar{\mu}_{k}^{+})} \circ (\mu_{k}^{+}\circ \beta \circ \bar{\mu}_{k}^{+}) \\ 
		& = \dual{(\bar{\mu}_{k}^{+})}\circ \dual{\lambda} \circ \dual{(\mu_{k}^{+})} \circ \mu_{k}^{+} \circ \beta \circ \bar{\mu}_{k}^{+} \\
		& = \dual{(\bar{\mu}_{k}^{+})}\circ \dual{\lambda} \circ \bar{\mu}_{k}^{+} \circ \mu_{k}^{+} \circ \beta \circ \bar{\mu}_{k}^{+} \\ 
		& = \dual{(\bar{\mu}_{k}^{+})}\circ \dual{\lambda} \circ \beta \circ \bar{\mu}_{k}^{+} \\ 
		& = \dual{(\bar{\mu}_{k}^{+})}\circ (\delta_{\cB}\circ \iota_{\cB}) \circ \bar{\mu}_{k}^{+} \\ 
		& = (\mu_{k}^{+}\circ (\bar{\mu}_{k}^{+} \circ \delta_{\mu_{k}(\cB)} \circ \mu_{k}^{+}) \circ (\bar{\mu}_{k}^{+}\circ \iota_{\mu_{k}(\cB)} \circ \mu_{k}^{+}) \circ \bar{\mu}_{k}^{+}) \\
		& = \delta_{\mu_{k}(\cB)}\circ \iota_{\mu_{k}(\cB)}
	\end{align*}
	as required.
\end{proof}

We also have the analogous properties of the mutation of $\lambda$ to that of $\beta$, starting with sign invariance (\cf Lemma~\ref{l:sign-invar}).

\begin{lemma}\label{l:lambda-form-sign-inv}
	Let $k\in \ex$.  Then 
	\[ \mu_{k}\canform{\blank}{\blank}{\lambda}= \canform{\blank}{\blank}{\lambda}\circ (\bar{\mu}_{k}^{-}\cross \bar{\mu}_{k}^{-}). \]
\end{lemma}

\begin{proof}
	For $\dual{b},\dual{c}\in \dual{\mu_{k}(\cB)}$ we have
	\begin{align*}
		(\canform{\blank}{\blank}{\lambda}\circ (\bar{\mu}_{k}^{-}\cross \bar{\mu}_{k}^{-}))(\dual{b},\dual{c}) 
		& = \begin{cases}
			\canform{\dual{b}}{\dual{c}}{\lambda} & \text{if}\ \dual{b},\dual{c}\neq \dual{\mu_{\cB}(k)} \\
			\canform{\dual{b}}{\bminus{\beta(k)}{\dual{\cB}}-\dual{k}}{\lambda} & \text{if}\ \dual{b}\neq \dual{\mu_{\cB}(k)},\dual{c}=\dual{\mu_{\cB}(k)} \\
			\canform{\bminus{\beta(k)}{\dual{\cB}}-\dual{k}}{\dual{c}}{\lambda} & \text{if}\ \dual{b}=\dual{\mu_{\cB}(k)},\dual{c}\neq \dual{\mu_{\cB}(k)} \\
			\canform{\bminus{\beta(k)}{\dual{\cB}}-\dual{k}}{\bminus{\beta(k)}{\dual{\cB}}-\dual{k}}{\lambda} & \text{if}\ \dual{b}=\dual{c}=\dual{\mu_{\cB}(k)} 
		\end{cases} \\
		& = \begin{cases}
			\canform{\dual{b}}{\dual{c}}{\lambda} & \text{if}\ \dual{b},\dual{c}\neq \dual{\mu_{\cB}(k)} \\
			-\canform{\dual{b}}{\dual{k}}{\lambda}+\canform{\dual{b}}{\bminus{\beta(k)}{\dual{\cB}}}{\lambda} & \text{if}\ \dual{b}\neq \dual{\mu_{\cB}(k)},\dual{c}=\dual{\mu_{\cB}(k)} \\
			-\canform{\dual{k}}{\dual{c}}{\lambda}+\canform{\bminus{\beta(k)}{\dual{\cB}}}{\dual{c}}{\lambda} & \text{if}\ \dual{b}=\dual{\mu_{\cB}(k)},\dual{c}\neq \dual{\mu_{\cB}(k)} \\
			0=\canform{\dual{k}}{\dual{k}}{\lambda} & \text{if}\ \dual{b}=\dual{c}=\dual{\mu_{\cB}(k)} 
		\end{cases} \\
		& = \begin{cases}
			\canform{\dual{b}}{\dual{c}}{\lambda} & \text{if}\ \dual{b},\dual{c}\neq \dual{\mu_{\cB}(k)} \\
			-\canform{\dual{b}}{\dual{k}}{\lambda}+\canform{\dual{b}}{\bplus{\beta(k)}{\dual{\cB}}}{\lambda} & \text{if}\ \dual{b}\neq \dual{\mu_{\cB}(k)},\dual{c}=\dual{\mu_{\cB}(k)} \\
			-\canform{\dual{k}}{\dual{c}}{\lambda}+\canform{\bplus{\beta(k)}{\dual{\cB}}}{\dual{c}}{\lambda} & \text{if}\ \dual{b}=\dual{\mu_{\cB}(k)},\dual{c}\neq \dual{\mu_{\cB}(k)} \\
			0=\canform{\dual{k}}{\dual{k}}{\lambda} & \text{if}\ \dual{b}=\dual{c}=\dual{\mu_{\cB}(k)} 
		\end{cases} \\
		& = \mu_{k}\canform{\dual{b}}{\dual{c}}{\lambda}
	\end{align*}
	by Lemma~\ref{l:lambda-form-mut} and since by compatibility $\evform{\beta(k)}{\lambda(\dual{b})}=\evform{k}{\dual{b}}=\delta_{bk}$ we have
	\begin{align*}
		\delta_{bk}=\canform{\dual{b}}{\beta(k)}{\lambda} & = \canform{\dual{b}}{\bplus{\beta(k)}{\dual{\cB}}-\bminus{\beta(k)}{\dual{\cB}}}{\lambda} \\
		& = \canform{\dual{b}}{\bplus{\beta(k)}{\dual{\cB}}}{\lambda}-\canform{\dual{b}}{\bminus{\beta(k)}{\dual{\cB}}}{\lambda}
	\end{align*}
	and hence for $\dual{b}\neq \dual{k}$, 
	\[ \canform{\dual{b}}{\bplus{\beta(k)}{\dual{\cB}}}{\lambda}=\canform{\dual{b}}{\bminus{\beta(k)}{\dual{\cB}}}{\lambda} \]
	with the other case obtained by skew-symmetry.
\end{proof}

\begin{proposition}\label{p:lambda-mut-invol}
	Let $k\in \ex$.  Let $\nu\colon \cB\to \mu_{\mu_{\cB}(k)}(\mu_{k}(\cB))$ be the bijection defined by $\nu(b)=b$ for $b\neq k$ and $\nu(k)=\mu_{\mu_{\cB}(k)}(\mu_{\cB}(k)))$ and write $\dual{\nu}\colon \dual{\integ[\mu_{\mu_{k}(\cB)}(\mu_{k}(\cB))]}\to \dual{\integ[\cB]}$ for the dual of the induced isomorphism of $\integ[\cB]$ with $\integ[\mu_{\mu_{\cB}(k)}(\mu_{k}(\cB))]$.
	
	Then
	\[ \mu_{\mu_{\cB}(k)}\mu_{k}\canform{\blank}{\blank}{\lambda}=\canform{\blank}{\blank}{\lambda}\circ (\dual{\nu} \cross \dual{\nu}). \]
\end{proposition}

\begin{proof}
	The argument is essentially the same as the proof of Proposition~\ref{p:beta-form-mut-invol}, using Lemma~\ref{l:EF-dual} and Lemma~\ref{l:F-mut-invol} along with Lemma~\ref{l:lambda-form-sign-inv}.
\end{proof}

\subsection{Quantum cluster algebras}\label{ss:QCAs}

We now begin a sequence of results and definitions that will give us quantum cluster mutation.  Recall that associated to the data of $\cB$ and $\lambda$ we have a quantum torus $\qtorus{q}{\lambda}{\cB}=(\bK \dual{\integ[\cB]})^{\Omega_{q}^{\lambda}}$.

\begin{proposition}\label{p:mut-autom}
	Let $\cB$ be a countable set, $\ex \subseteq \cB$, $\beta\colon \integ[\ex]\to \dual{\integ[\cB]}$ skew-symmetrizable and $\lambda\colon \dual{\integ[\cB]}\to \integ[\cB]$ skew-symmetric and compatible with $\beta$. Let $k\in \ex$ and $\alpha\in \units{\bK}$.
	\begin{enumerate}[label=\textup{(\alph*)}]
		\item There exists a homomorphism $\rho_{\alpha}^{\cB,k}\colon \qtorus{q}{\mu_{k}\lambda}{\mu_{k}(\cB)}\to \curly{F}(\qtorus{q}{\lambda}{\cB})$
		defined on a generating set of its domain by
		\[ \rho_{\alpha}^{\cB,k}(x^{\dual{b}})=\begin{cases} x^{\bar{\mu}_{k}^{+}(\dual{\mu_{\cB}(k)})}(1+\alpha x^{-\beta(k)}) & \text{if}\ \dual{b}=\dual{\mu_{\cB}(k)} \\ x^{\dual{b}} & \text{if}\ \dual{b}\neq \dual{\mu_{\cB}(k)} \end{cases} \]
		for $\dual{b}\in \dual{\mu_{k}(\cB)}$.
		\item There exists a homomorphism $(\rho')_{\alpha}^{\cB,k}\colon \qtorus{q}{\lambda}{\cB}\to \curly{F}(\qtorus{q}{\mu_{k}\lambda}{\mu_{k}(\cB)})$
		defined on a generating set of its domain by
		\[ (\rho')_{\alpha}^{\cB,k}(x^{\dual{b}})=\begin{cases} (1+\alpha x^{\mu_{k}\beta(\mu_{\cB}(k)})x^{\mu_{k}^{+}(\dual{k})} & \text{if}\ \dual{b}=\dual{k} \\ x^{\dual{b}} & \text{if}\ \dual{b}\neq \dual{k} \end{cases} \]
		for $\dual{b}\in \dual{\cB}$.
		\item The homomorphisms $\rho_{\alpha}^{\cB,k}$ and $(\rho')_{\alpha}^{\cB,k}$ extend to inverse isomorphisms
		\[ \rho_{\alpha}^{\cB,k}\colon \curly{F}(\qtorus{q}{\mu_{k}\lambda}{\mu_{k}(\cB)}) \stackrel{\iso}{\longleftrightarrow} \curly{F}(\qtorus{q}{\lambda}{\cB}) \colon (\rho')_{\alpha}^{\cB,k}. \]
	\end{enumerate}
\end{proposition}

\begin{proof} {\ }
	\begin{enumerate}[label=\textup{(\alph*)}]
		\item It suffices to check that the images $\rho_{\alpha}^{\cB,k}(x^{\dual{b}})$ satisfy the quasi-commutation relations of $\qtorus{q}{\mu_{k}\lambda}{\mu_{k}(\cB)}$.  If $\dual{b},\dual{c}\neq \dual{\mu_{\cB}(k)}$ or $\dual{b}=\dual{c}=\dual{\mu_{\cB}(k)}$, this is immediate.
		
		For the case $\dual{b}=\dual{\mu_{\cB}(k)}$, $\dual{c}\neq \dual{\mu_{\cB}(k)}$, first note that by compatibility, $\canform{\dual{c}}{\beta(k)}{\lambda}=\evform{\lambda(\dual{c})}{\beta(k)}=\delta_{ck}$ so that in $\curly{F}(\qtorus{q}{\lambda}{\cB})$, $x^{\beta(k)}$ commutes with $x^{\dual{c}}$ for $c\neq k$.  Also, $\canform{\bar{\mu}_{k}^{+}(\dual{\mu_{\cB}(k)})}{\dual{c}}{\lambda}=\canform{\dual{\mu_{\cB}(k)}}{\dual{c}}{\mu_{k}\lambda}$ for $c\neq k$.
		
		Then for $\dual{b}=\dual{\mu_{\cB}(k)}$, $\dual{c}\neq \dual{\mu_{\cB}(k)}$,
		\begin{align*}
			\rho_{\alpha}^{\cB,k}(x^{\dual{\mu_{\cB}(k)}})\rho_{\alpha}^{\cB,k}(x^{\dual{c}}) & = x^{\bar{\mu}_{k}^{+}(\dual{\mu_{\cB}(k)})}(1+\alpha x^{-\beta(k)})x^{\dual{c}} \\
			& = x^{\bar{\mu}_{k}^{+}(\dual{\mu_{\cB}(k)})}x^{\dual{c}}(1+\alpha x^{-\beta(k)}) \\
			& = q^{\canform{\bar{\mu}_{k}^{+}(\dual{\mu_{\cB}(k))}}{\dual{c}}{\lambda}}x^{\dual{c}}x^{\bar{\mu}_{k}^{+}(\dual{\mu_{\cB}(k)})}(1+\alpha x^{-\beta(k)}) \\
			& = q^{\canform{\dual{\mu_{\cB}(k)}}{\dual{c}}{\mu_{k}\lambda}}x^{\dual{c}}x^{\bar{\mu}_{k}^{+}(\dual{\mu_{\cB}(k)})}(1+\alpha x^{-\beta(k)}) \\
			& = q^{\canform{\dual{\mu_{\cB}(k)}}{\dual{c}}{\mu_{k}\lambda}}\rho_{\alpha}^{\cB,k}(x^{\dual{c}})\rho_{\alpha}^{\cB,k}(x^{\dual{\mu_{\cB}(k)}})
		\end{align*}
		as required.  
		\item Similarly to the above, we have that $x^{\mu_{k}\beta(\mu_{\cB}(k))}$ commutes with $x^{\dual{c}}$ for $\dual{c}\neq \dual{\mu_{\cB}(k)}$.  We then have for $\dual{b}=\dual{k}$, $\dual{c}\neq \dual{k}$,
		\begin{align*}
			(\rho')_{\alpha}^{\cB,k}(x^{\dual{k}})(\rho')_{\alpha}^{\cB,k}(x^{\dual{c}}) & = (1+\alpha x^{\mu_{k}\beta(\mu_{\cB}(k))})x^{\mu_{k}^{+}(\dual{k})}x^{\dual{c}} \\
			& = q^{\canform{\mu_{k}^{+}(\dual{k})}{\dual{c}}{\mu_{k}\lambda}}(1+\alpha x^{\mu_{k}\beta(\mu_{\cB}(k))})x^{\dual{c}}x^{\mu_{k}^{+}(\dual{k})} \\
			& = q^{\canform{\mu_{k}^{+}(\dual{k})}{\dual{c}}{\mu_{k}\lambda}}x^{\dual{c}}(1+\alpha x^{\mu_{k}\beta(\mu_{\cB}(k))})x^{\mu_{k}^{+}(\dual{k})} \\
			& = q^{\canform{\bar{\mu}_{k}^{+}\mu_{k}^{+}(\dual{k})}{\bar{\mu}_{k}^{+}\dual{c}}{\lambda}}x^{\dual{c}}(1+\alpha x^{\mu_{k}\beta(\mu_{\cB}(k))})x^{\mu_{k}^{+}(\dual{k})} \\
			& = q^{\canform{\dual{k}}{\dual{c}}{\lambda}}(\rho')_{\alpha}^{\cB,k}(x^{\dual{c}})(\rho')_{\alpha}^{\cB,k}(x^{\dual{k}}).
		\end{align*}
		\item We have
		\begin{align*}
			(\rho')_{\alpha}^{\cB,k}\left(\rho_{\alpha}^{\cB,k}(x^{\dual{\mu_{\cB}(k)}})\right) & = (\rho')_{\alpha}^{\cB,k}\left(x^{\bar{\mu}_{k}^{+}(\dual{\mu_{\cB}(k)})}(1+\alpha x^{-\beta(k)})\right) \\ 
			& = (\rho')_{\alpha}^{\cB,k}\left(x^{\bar{\mu}_{k}^{+}(\dual{\mu_{\cB}(k)})}\right)(1+\alpha x^{-\beta(k)}) \\
			& = (\rho')_{\alpha}^{\cB,k}\left(x^{\bplus{\beta(k)}{\dual{\cB}}-\dual{k}}\right)(1+\alpha x^{-\beta(k)}) \\
			& = x^{\bplus{\beta(k)}{\dual{\cB}}}(\rho')_{\alpha}^{\cB,k}\left(x^{\dual{k}}\right)^{-1}(1+\alpha x^{-\beta(k)}) \\
			& = x^{\bplus{\beta(k)}{\dual{\cB}}}\left((1+\alpha x^{\mu_{k}\beta(\mu_{\cB}(k))})x^{\mu_{k}^{+}(\dual{k})}\right)^{-1}(1+\alpha x^{-\beta(k)}) \\
			& = x^{\bplus{\beta(k)}{\dual{\cB}}}(x^{\bplus{\beta(k)}{\dual{B}}-\dual{\mu_{\cB}(k)}})^{-1}(1+\alpha x^{-\beta(k)})^{-1}(1+\alpha x^{-\beta(k)}) \\
			& = x^{\dual{\mu_{\cB}(k)}} 
		\end{align*}
		since $x^{\mu_{k}\beta(\mu_{\cB}(k))}=x^{-\beta(k)}$ and similarly for the other case. \qedhere
	\end{enumerate}
\end{proof}

That is, the two quantum tori $\qtorus{q}{\lambda}{\cB}$ and $\qtorus{q}{\mu_{k}\lambda}{\mu_{k}(\cB)}$ are birationally equivalent.  Note that this is a much stronger claim in the non-commutative setting, where the two quantum tori are not isomorphic except in degenerate cases.

We now use the canonical toric frames associated to $\lambda$ and $\mu_{k}\lambda$,
\begin{align*} M_{\lambda} & \colon \dual{\integ[\cB]}\to \mathcal{F}(\qtorus{q}{\lambda}{\cB}), \quad M_{\lambda}(\dual{b})=x^{\dual{b}}, \\ M_{\mu_{k}\lambda} & \colon \dual{\integ[\mu_{k}(\cB)]}\to \mathcal{F}(\qtorus{q}{\mu_{k}\lambda}{\mu_{k}(\cB)}), \quad M_{\mu_{k}\lambda}(\dual{b})=x^{\dual{b}}  \end{align*}
and the above isomorphism to produce a map from $\integ[\mu_{k}(\cB)]$ to the skew-field of fractions of the ``root'' quantum torus, $\curly{F}(\qtorus{q}{\lambda}{\cB})$.

\begin{definition}\label{d:one-step-mut}
	Let $\cB$ be a countable set, $\ex \subseteq \cB$, $\beta\colon \integ[\ex]\to \dual{\integ[\cB]}$ skew-symmetrizable and $\lambda\colon \dual{\integ[\cB]}\to \integ[\cB]$ skew-symmetric and compatible with $\beta$.
	
	For $k\in \ex$, define
	\[ \mu_{k}M_{\lambda}\colon \dual{\integ[\mu_{k}(\cB)]}\to \curly{F}(\qtorus{q}{\lambda}{\cB}), \mu_{k}M_{\lambda}=\rho_{\Omega_{q}^{\mu_{k}\lambda}(\mu_{k}\beta(\mu_{\cB}(k)),\dual{\mu_{\cB}(k)})}^{\cB,k}\circ M_{\mu_{k}\lambda} . \]
\end{definition}

\begin{lemma} {\ }
	\begin{enumerate}[label=\textup{(\alph*)}]
		\item The function $\mu_{k}M_{\lambda}$ satisfies
		\[ \mu_{k}M_{\lambda}(\dual{b}) = \begin{cases} M_{\lambda}(\bar{\mu}_{k}^{+}(\dual{\mu_{\cB}(k)}))+M_{\lambda}(\bar{\mu}_{k}^{-}(\dual{\mu_{\cB}(k)})) & \text{if}\ \dual{b}=\dual{\mu_{\cB}(k)} \\ M_{\lambda}(\dual{b}) & \text{if}\ \dual{b}\neq \dual{\mu_{\cB}(k)} \end{cases}. \]
		\item The $\bK$-span of the image of $\mu_{k}M_{\lambda}$ is isomorphic to $\qtorus{q}{\mu_{k}\lambda}{\mu_{k}(\cB)}$.
	\end{enumerate}
\end{lemma}

\begin{proof} {\ }
	\begin{enumerate}[label=\textup{(\alph*)}]
		\item The case $\dual{b}\neq \dual{\mu_{\cB}(k)}$ is immediate.  For the other case we have
		\begin{align*}
			\mu_{k}M_{\lambda}(\dual{\mu_{\cB}(k)}) & = (\rho_{\Omega_{q}^{\mu_{k}\lambda}(\mu_{k}\beta(\mu_{\cB}(k)),\dual{\mu_{\cB}(k)})}^{\cB,k}\circ M_{\mu_{k}\lambda})(\dual{\mu_{\cB}(k)}) \\
			& = \rho_{\Omega_{q}^{\mu_{k}\lambda}(\mu_{k}\beta(\mu_{\cB}(k)),\dual{\mu_{\cB}(k)})}^{\cB,k}(x^{\dual{\mu_{\cB}(k)}}) \\
			& = x^{\bar{\mu}_{k}^{+}(\dual{\mu_{\cB}(k)})}(1+\Omega_{q}^{\mu_{k}\lambda}(\mu_{k}\beta(\mu_{\cB}(k)),\dual{\mu_{\cB}(k)}) x^{-\beta(k)}) \\
			& = x^{\bar{\mu}_{k}^{+}(\dual{\mu_{\cB}(k)})}(1+\Omega_{q}^{\mu_{k}\lambda}(\dual{\mu_{\cB}(k)},\beta(k)) x^{-\beta(k)}) \\
			& = M_{\lambda}(\bar{\mu}_{k}^{+}(\dual{\mu_{\cB}(k)}))+\Omega_{q}^{\mu_{k}\lambda}(\dual{\mu_{\cB}(k)},\beta(k)) \Omega_{q}^{\lambda}(\bar{\mu}_{k}^{+}(\dual{\mu_{\cB}(k)}),-\beta(k)) M_{\lambda}(\bar{\mu}_{k}^{+}(\dual{\mu_{\cB}(k)})-\beta(k)) \\
			& = M_{\lambda}(\bar{\mu}_{k}^{+}(\dual{\mu_{\cB}(k)}))+\Omega_{q}^{\mu_{k}\lambda}(\dual{\mu_{\cB}(k)},\beta(k)) \Omega_{q}^{\lambda}(\bar{\mu}_{k}^{+}(\dual{\mu_{\cB}(k)}),-\beta(k)) M_{\lambda}(\bar{\mu}_{k}^{+}(\dual{\mu_{\cB}(k)})-\beta(k)) \\
			& = M_{\lambda}(\bar{\mu}_{k}^{+}(\dual{\mu_{\cB}(k)}))+\Omega_{q}^{\mu_{k}\lambda}(\dual{\mu_{\cB}(k)},\beta(k)) \Omega_{q}^{\mu_{k}\lambda}(\beta(k),\dual{\mu_{\cB}(k)}) M_{\lambda}(\bar{\mu}_{k}^{+}(\dual{\mu_{\cB}(k)})-\beta(k)) \\
			& = M_{\lambda}(\bar{\mu}_{k}^{+}(\dual{\mu_{\cB}(k)}))+M_{\lambda}(\bar{\mu}_{k}^{-}(\dual{\mu_{\cB}(k)}))
		\end{align*}
		since
		\[ \beta(k)=\bplus{\beta(k)}{\dual{\cB}}-\bminus{\beta(k)}{\dual{\cB}}=\bar{\mu}_{k}^{+}(\dual{\mu_{\cB}(k)})-\bar{\mu}_{k}^{-}(\dual{\mu_{\cB}(k)}) \]
		and
		\begin{align*}
			\canform{\bar{\mu}_{k}^{+}(\dual{\mu_{\cB}(k)})}{-\beta(k)}{\lambda} & = \canform{\bar{\mu}_{k}^{+}(\dual{\mu_{\cB}(k)})}{\bar{\mu}_{k}^{-}(\dual{\mu_{\cB}(k)})-\bar{\mu}_{k}^{+}(\dual{\mu_{\cB}(k)})}{\lambda} \\
			& = \canform{\bar{\mu}_{k}^{+}(\dual{\mu_{\cB}(k)})}{\bar{\mu}_{k}^{-}(\dual{\mu_{\cB}(k)})}{\lambda} \\
			& = \canform{\dual{\mu_{\cB}(k)}}{\mu_{k}^{+}\bar{\mu}_{k}^{-}(\dual{\mu_{\cB}(k)})}{\mu_{k}\lambda}  \\
			& = \canform{\dual{\mu_{\cB}(k)}}{\dual{\mu_{\cB}(k)}-\beta(k)}{\mu_{k}\lambda}  \\
			& = \canform{\beta(k)}{\dual{\mu_{\cB}(k)}}{\mu_{k}\lambda}.
		\end{align*}
		\item This follows from Proposition~\ref{p:mut-autom}: we saw that the elements 
		\[ \left\{ \mu_{k}M_{\lambda}(\dual{b})=\rho_{\Omega_{q}^{\mu_{k}\lambda}(\mu_{k}\beta(\mu_{\cB}(k)),\dual{\mu_{\cB}(k)})}^{\cB,k}\circ M_{\mu_{k}\lambda}(\dual{b}) \mid \dual{b}\in \dual{\mu_{k}(\cB)} \right\}\]
		quasi-commute and that this quasi-commutation is encoded by $\mu_{k}\lambda$. \qedhere
	\end{enumerate}
\end{proof}

Let $\cB$ be a countable set and $\ex \subseteq \cB$.  Recall from Definition~\ref{d:mut-of-index-sets} that for $k\in \ex$, we have $\mu_{k}(\ex) \subseteq \mu_{k}(\cB)$ with $\mu_{k}(\cB)=(\cB\setminus \{k\})\union \{ \mu_{\cB}(k)\}$ and $\mu_{k}(\ex)=\mu_{k}(\cB)\setminus \ex^{c}$.  Furthermore, there are bijections \label{pg:mut-bij}
\begin{align*}
	\mu_{k} \colon \cB \to \mu_{k}(\cB),\ & \mu_{k}(b)= \begin{cases} b & \text{if}\ b\neq k \\ \mu_{\cB}(k) & \text{if}\ b=k \end{cases} \\ \intertext{and}
	\mu_{k} \colon \ex \to \mu_{k}(\ex),\ & \mu_{k}(b)= \begin{cases} b & \text{if}\ b\neq k \\ \mu_{\cB}(k) & \text{if}\ b=k \end{cases}.
\end{align*}
We may iterate this construction: given $l\in \mu_{k}(\ex)$, we have $\mu_{l}(\mu_{k}(\ex))\subseteq \mu_{l}(\mu_{k}(\cB))$ defined similarly to above, with $\mu_{l}(\mu_{k}(\cB))=(\mu_{k}(\cB)\setminus \{ l \})\union \{ \mu_{\mu_{k}(\ex)}(l) \}$.  Again, these admit bijections $\mu_{l}\colon \mu_{k}(\cB)\to \mu_{l}(\mu_{k}(\cB))$ and $\mu_{l}\colon \mu_{k}(\ex)\to \mu_{l}(\mu_{k}(\ex))$.  Repeating this process, we obtain for any tuple $\uk=(k_{r},k_{r-1},\dotsc ,k_{2},k_{1})$ with 
\begin{align*}
	k_{1} & \in \ex, \\
	k_{2} & \in \mu_{k_{1}}(\ex), \\
	& \phantom{\in}\!\vdots \\
	k_{r-1} & \in \mu_{k_{r-2}}(\dotsm (\mu_{k_{2}}(\mu_{k_{1}}(\ex)))), \\
	k_{r} & \in \mu_{k_{r-1}}(\mu_{k_{r-2}}(\dotsm (\mu_{k_{2}}(\mu_{k_{1}}(\ex))))),
\end{align*}
a set 
\begin{equation}\label{eq:mu-k-B} \mu_{\uk}(\cB) \defeq \mu_{k_{r}}(\mu_{k_{r-1}}(\mu_{k_{r-2}}(\dotsm (\mu_{k_{2}}(\mu_{k_{1}}(\cB)))))) \end{equation}
and a subset
\begin{equation}\label{eq:mu-k-ex} \mu_{\uk}(\ex) \defeq \mu_{\uk}(\cB) \setminus \ex^{c} \end{equation}
with associated bijections $\mu_{k_{r}}\colon \mu_{(k_{r-1},\dotsc ,k_{1})}(\cB)\to \mu_{\uk}(\cB)$ and $\mu_{k_{r}}\colon \mu_{(k_{r-1},\dotsc ,k_{1})}(\ex)\to \mu_{\uk}(\ex)$.

\begin{definition}\label{d:admissible-mut-seq}
	Let $\cB$ be a finite set and $\ex \subseteq \cB$.
	
	Denote by $\curly{K}(\ex,\cB)$ the set of all tuples $\uk=(k_{r},\dotsc ,k_{1})$, such that for all $1\leq i \leq r$, $k_{i}\in \mu_{k_{i-1}}(\dotsm (\mu_{k_{2}}(\mu_{k_{1}}(\ex))))$, including the empty tuple $()$.
	
	We say that an element $\uk\in \curly{K}(\ex,\cB)$ is an $(\ex,\cB)$-admissible mutation sequence.
\end{definition}

We adopt the natural convention that $(\mu_{()}(\ex),\mu_{()}(\cB))=(\ex,\cB)$.  We also need to iterate the construction in Definition~\ref{d:one-step-mut}.  

We first extend the above notation to encompass the iterated mutation of $\beta$ and $\lambda$ to obtain
\begin{align*}
	\mu_{\uk}\beta & =\mu_{k_{r}}(\mu_{k_{r-1}}(\dotsm (\mu_{k_{2}}(\mu_{k_{1}}\beta)))) \\ \intertext{and}
	\mu_{\uk}\lambda & =\mu_{k_{r}}(\mu_{k_{r-1}}(\dotsm (\mu_{k_{2}}(\mu_{k_{1}}\lambda))))  
\end{align*}
for $\uk=(k_{r},\dotsc ,k_{1})\in \curly{K}$.

By analogy with Definition~\ref{d:one-step-mut}, for $\uk=(k_{r},\dotsc ,k_{1})\in \curly{K}$, define
\[ \alpha_{i}=\Omega_{q}^{\mu_{(k_{i},\dotsc, k_{1})}\lambda}(\mu_{(k_{i},\dotsc,k_{1})}\beta(\mu_{\mu_{(k_{i-1},\dotsc ,k_{1})}(\cB)}(k_{i})),\dual{\mu_{\mu_{(k_{i-1},\dotsc ,k_{1})}(\cB)}(k_{i})}) \]
for $1\leq i\leq r$.  Perhaps surprisingly, the apparent complexity of this expression is misleading: we claim that $\alpha_{i}=q^{-\frac{1}{2}}$ for all valid choices of input data.

This follows from compatibility; induction on the length of $\uk$ yields that $\mu_{\uk}\lambda$ and $\mu_{\uk}\beta$ are compatible, via Lemma~\ref{l:mut-compatible}.  That is, 
\[ \evform{\mu_{\uk}\lambda(\dual{b})}{\mu_{\uk}\beta(c)}=\canform{\dual{b}}{\mu_{\uk}\beta(c)}{\mu_{\uk}\lambda}=\delta_{bc}. \]
Then, for $\uk=(k_{i},\dotsc ,k_{1})$,
\begin{align*} \alpha_{i} & =q^{\frac{1}{2}\canform{\mu_{(k_{i},\dotsc,k_{1})}\beta(\mu_{\mu_{(k_{i-1},\dotsc ,k_{1})}(\cB)}(k_{i}))}{\dual{\mu_{\mu_{(k_{i-1},\dotsc ,k_{1})}(\cB)}(k_{i})}}{\mu_{(k_{i},\dotsc, k_{1})}\lambda}} \\ &
	=q^{-\frac{1}{2}\canform{\dual{\mu_{\mu_{(k_{i-1},\dotsc ,k_{1})}(\cB)}(k_{i})}}{\mu_{(k_{i},\dotsc,k_{1})}\beta(\mu_{\mu_{(k_{i-1},\dotsc ,k_{1})}(\cB)}(k_{i}))}{\mu_{(k_{i},\dotsc, k_{1})}\lambda}} \\ 
	& =q^{-\frac{1}{2}(\delta_{\mu_{\mu_{(k_{i-1},\dotsc ,k_{1})}(\cB)}(k_{i}),\mu_{\mu_{(k_{i-1},\dotsc ,k_{1})}(\cB)}(k_{i})})}\\ & =q^{-\frac{1}{2}}. \end{align*}

\begin{definition}
	Let $\cB$ be a countable set, $\ex \subseteq \cB$, $\beta\colon \integ[\ex]\to \dual{\integ[\cB]}$ skew-symmetrizable and $\lambda\colon \dual{\integ[\cB]}\to \integ[\cB]$ skew-symmetric and compatible with $\beta$.
	
 Define $\mu_{\uk}M_{\lambda}\colon \dual{\integ[\mu_{\uk}(\cB)]}\to \curly{F}(\qtorus{q}{\lambda}{\cB})$ by 
	\[ \mu_{\uk}M_{\lambda}\defeq \rho_{q^{-1/2}}^{\cB,k_{1}}\circ \rho_{q^{-1/2}}^{\mu_{k_{1}}(\cB),k_{2}} \circ \dotsm \circ \rho_{q^{-1/2}}^{\mu_{(k_{r-2},\dotsc ,k_{1})}(\cB),k_{r-1}} \circ \rho_{q^{-1/2}}^{\mu_{(k_{r-1},\dotsc ,k_{1})}(\cB),k_{r}}\circ M_{\mu_{\uk}\lambda}. \]
\end{definition}

\begin{definition}
	Let $\cB$ be a countable set, $\ex \subseteq \cB$, $\beta\colon \integ[\ex]\to \dual{\integ[\cB]}$ skew-symmetrizable and $\lambda\colon \dual{\integ[\cB]}\to \integ[\cB]$ skew-symmetric and compatible with $\beta$. Let $\bK$ be a field and $q^{\frac{1}{2}}\in \units{\bK}$.
	
	Define the \emph{quantum cluster algebra} $\curly{C}_{q}(\ex,\cB,\beta,\lambda)$ to be the $\bK$-subalgebra of $\curly{F}(\qtorus{q}{\lambda}{\cB})$ generated by the elements
	\[ \{ \mu_{\uk}M_{\lambda}(\dual{b}) \mid \dual{b}\in \dual{\mu_{\uk}(\cB)}, \uk\in \curly{K} \}. \]
	Elements of this generating set are called \emph{quantum cluster variables}.
\end{definition}

\begin{remark}\label{r:frozens}
	Note that for $\dual{f}\in \dual{(\cB\setminus \ex)}$, $\mu_{\uk}M_{\lambda}(\dual{f})=x^{\dual{f}}$ for any $\uk\in \curly{K}$ since all the maps $\rho_{\bullet}^{\bullet}$ act as the identity on these elements.  Consequently, for any $\dual{b}\in \dual{\mu_{\uk}(\cB)}$, both $\mu_{\uk}M_{\lambda}(\dual{b})$ and $\mu_{\uk}M_{\lambda}(\dual{f})=x^{\dual{f}}$ are generators of the quantum torus $\qtorus{q}{\mu_{\uk}\lambda}{\mu_{\uk}(\cB)}$ and therefore quasi-commute.
	
	That is, frozen initial quantum cluster variables are frozen variables in every quantum cluster and quasi-commute with every quantum cluster variable.
\end{remark}

\begin{remark} Note that choosing $q^{\frac{1}{2}}=1$ yields a commutative algebra $\cC_{1}(\ex,\cB,\beta,\lambda)$, since the associated bicharacter $\Omega_{1}^{\lambda}$ is equal to the constant bicharacter $\mathbbm{1}\colon \dual{\integ[\cB]}\cross \dual{\integ[\cB]}\to \units{\bK}$, $\mathbbm{1}(v,w)=1_{\bK}$.  Moreover, if $\lambda'\colon \dual{\integ[\cB]}\to \integ[\cB]$ is also skew-symmetric and compatible with $\beta$ then $\cC_{1}(\ex,\cB,\beta,\lambda)=\cC_{1}(\ex,\cB,\beta,\lambda')$.
	
	One might then observe that $\lambda=0$ also gives rise to $\Omega_{q}^{0}=\mathbbm{1}$ for any choice of $q^{\frac{1}{2}}$.  However, there is an obstruction here: the requirement that $\beta$ and $\lambda$ be compatible means that $\lambda=0$ is not allowed.
	
	Also, compatibility implies that $\beta$ is injective, so while in the case that $\beta$ is injective, any choice of compatible $\lambda$ gives the same associated commutative cluster algebra $\cC_{1}(\ex,\cB,\beta,\lambda)$, strictly speaking we have not yet defined cluster algebras for $\beta$ not injective. Fortunately, this is just a technicality: all the above theory involving just $\beta$ goes through without modification and taking $\alpha=1$ in Proposition~\ref{p:mut-autom} we still obtain birational maps for the cluster tori, where instead of $\qtorus{q}{\lambda}{\cB}$ we simply take $\qtorus{}{}{\cB}\defeq \bK \dual{\integ[\cB]}$, with no need to introduce $\lambda$.  
	
	Following through the definitions above, we see that we can define $\cC(\ex,\cB,\beta)$ a (commutative) cluster algebra.  If $\beta$ is injective and $\lambda$ is compatible with $\beta$, then $\cC_{1}(\ex,\cB,\beta,\lambda)=\cC(\ex,\cB,\beta)$, as one would wish.	
\end{remark}

Above, we saw that tropical mutation is involutive up to the canonical bijection $\nu$ of the second mutation $\mu_{\mu_{\cB}(k)}(\mu_{k}(\cB))$ with the initial basis $\cB$.  In the next lemma, we see the analogous claim for mutation of quantum cluster variables. 

\begin{proposition} Let $\cB$ be a countable set, $\ex \subseteq \cB$, $\beta\colon \integ[\ex]\to \dual{\integ[\cB]}$ skew-symmetrizable and $\lambda\colon \dual{\integ[\cB]}\to \integ[\cB]$ skew-symmetric and compatible with $\beta$. Let $\bK$ be a field and $q^{\frac{1}{2}}\in \units{\bK}$.
	
	Let $k\in \ex$ and $\uk=(\mu_{\cB}(k),k)$.  Let $\nu\colon \cB\to \mu_{\mu_{\cB}(k)}(\mu_{k}(\cB))$ be the bijection defined by $\nu(b)=b$ for $b\neq k$ and $\nu(k)=\mu_{\mu_{\cB}(k)}(\mu_{\cB}(k)))$ and $\dual{\nu}$ its dual.
	
	Then there is an induced isomorphism $\psi_{\dual{\nu}}\colon  \qtorus{q}{\mu_{\uk}\lambda}{\mu_{\uk}(\cB)}\to \qtorus{q}{\lambda}{\cB}$ defined by $\psi_{\dual{\nu}}(x^{\dual{b}})=x^{\dual{\nu}(\dual{b})}$ and for all $\dual{b}\in \dual{\mu_{\mu_{\cB}(k)}(\mu_{k}(\cB))}$ we have
	\[ \mu_{\uk}M_{\lambda}(\dual{b})=x^{\dual{\nu}(\dual{b})} =\psi_{\dual{\nu}}(x^{\dual{b}}).\]
	That is, mutation in the direction $k$ followed by mutation back, in the direction $\mu_{\cB}(k)$, is the identity modulo $\dual{\nu}$.
\end{proposition}

\begin{proof}
	To see that $\nu$ is indeed an isomorphism of quantum tori, recall from Lemma~\ref{l:can-basis-invar} that $\dual{\nu}$ induces an isomorphism 		
	\[ \psi_{\dual{\nu}}\colon (\bK \dual{\integ[\mu_{\mu_{\cB}(k)}(\mu_{k}(\cB))]})^{\Omega_{q}^{\lambda,\dual{\nu}}}\to (\bK \dual{\integ[\cB]})^{\Omega_{q}^{\lambda}}=\qtorus{q}{\lambda}{\cB} \]
	with $\psi_{\dual{\nu}}(x^{\dual{b}})=x^{\dual{\nu}(\dual{b})}$ is an isomorphism, with inverse given by $\psi_{\dual{\nu}}^{-1}(x^{\dual{b}})=x^{(\dual{\nu})^{-1}(\dual{b})}$.
	 
	Then, by Proposition~\ref{p:lambda-mut-invol}, 
	\begin{align*} \Omega_{q}^{\lambda,\dual{\nu}} & =\Omega_{q}^{\lambda}\circ (\dual{\nu} \cross \dual{\nu}) \\ & = q^{\frac{1}{2}\canform{\dual{\nu}(\blank)}{\dual{\nu}(\blank)}{\lambda}} \\ & = q^{\frac{1}{2}\mu_{\mu_{\cB}(k)}\mu_{k}\canform{\blank}{\blank}{\lambda}} \\ & = \Omega_{q}^{\mu_{\uk\lambda}}
	\end{align*}
	so that $\psi_{\dual{\nu}}\colon \qtorus{q}{\mu_{\uk}\lambda}{\mu_{\uk}(\cB)}\to \qtorus{q}{\lambda}{\cB}$ is an isomorphism as claimed.
		
	Now,
		\[ \mu_{\uk}M_{\lambda}=\rho_{q^{-1/2}}^{\cB,k}\circ \rho_{q^{-1/2}}^{\mu_{k}(\cB),\mu_{\cB}(k)} \circ M_{\mu_{\uk}\lambda} \]
		and for $\dual{b}\in \dual{\mu_{\mu_{\cB}(k)}(\mu_{k}(\cB))}\setminus \{ \dual{\mu_{\mu_{\cB}(k)}(\mu_{\cB}(k))} \}$ we have $M_{\mu_{\uk\lambda}}(\dual{b})=x^{\dual{b}}$, $\rho_{q^{-1/2}}^{\mu_{k}(\cB),\mu_{\cB}(k)}(x^{\dual{b}})=x^{\dual{b}}$ and $\rho_{q^{-1/2}}^{\cB,k}(x^{\dual{b}})=x^{\dual{b}}$, since this is the situation where no mutation takes place in the direction $b$ at either step. So $\mu_{\uk}M_{\lambda}(\dual{b})=x^{\dual{b}}=x^{\dual{\nu}(\dual{b})}=\psi_{\dual{\nu}}(x^{\dual{b}})$.
		
		So it remains to compute $\mu_{\uk}M_{\lambda}(\dual{\mu_{\mu_{k}(\cB)}(\mu_{\cB}(k))})$. Firstly, $M_{\mu_{\uk\lambda}}(\dual{\mu_{\mu_{k}(\cB)}(\mu_{\cB}(k))})=x^{\dual{\mu_{\mu_{k}(\cB)}(\mu_{\cB}(k))}}$. Then, using Lemma~\ref{l:E-isos}, Lemma~\ref{l:mut-form-values} and the dual of Lemma~\ref{l:F-mut-invol}, we have that
		\begin{align*} \rho_{q^{-1/2}}^{\mu_{k}(\cB),\mu_{\cB}(k)}(x^{\dual{\mu_{\mu_{k}(\cB)}(\mu_{\cB}(k))}}) 
		& =x^{\bar{\mu}^{+}_{\mu_{\cB}(k)}(\dual{\mu_{\mu_{k}(\cB)}(\mu_{\cB}(k))})}(1+q^{-1/2}x^{-\mu_{k}\beta(\mu_{\cB}(k))}) \\ 
		& = x^{\bplus{\mu_{k}\beta(\mu_{\cB}(k))}{\dual{\mu_{k}(\cB)}}-\dual{\mu_{\cB}(k)}}(1+q^{-1/2}x^{\bminus{\mu_{k}\beta(\mu_{\cB}(k))}{\dual{\mu_{k}(\cB)}}-\bplus{\mu_{k}\beta(\mu_{\cB}(k))}{\dual{\mu_{k}(\cB)}}}) \\ 
		& = x^{\mu_{k}\beta(\mu_{\cB}(k))+\bminus{\mu_{k}\beta(\mu_{\cB}(k))}{\dual{\mu_{k}(\cB)}}-\dual{\mu_{\cB}(k)}}+q^{-1/2}\alpha_{1}x^{\bminus{\mu_{k}\beta(\mu_{\cB}(k))}{\dual{\mu_{k}(\cB)}}-\dual{\mu_{\cB}(k)}} \\ 
		& = x^{\mu_{k}\beta(\mu_{\cB}(k))+\bar{\mu}^{-}_{\mu_{\cB}(k)}(\dual{\mu_{\mu_{k}(\cB)}(\mu_{\cB}(k))})}+q^{-1/2}\alpha_{1}x^{\bar{\mu}^{-}_{\mu_{\cB}(k)}(\dual{\mu_{\mu_{k}(\cB)}(\mu_{\cB}(k))})} \\ 
		& = \alpha_{2}x^{\mu_{k}\beta(\mu_{\cB}(k))}x^{\bar{\mu}^{-}_{\mu_{\cB}(k)}(\dual{\mu_{\mu_{k}(\cB)}(\mu_{\cB}(k))})}+q^{-1/2}\alpha_{1}x^{\bar{\mu}^{-}_{\mu_{\cB}(k)}(\dual{\mu_{\mu_{k}(\cB)}(\mu_{\cB}(k))})} \\ 
		& = (q^{-1/2}\alpha_{1}+\alpha_{2}x^{\mu_{k}\beta(\mu_{\cB}(k))})x^{\bar{\mu}^{-}_{\mu_{\cB}(k)}(\dual{\mu_{\mu_{k}(\cB)}(\mu_{\cB}(k))})} \\ 
		& = (q^{-1/2}\alpha_{1}+\alpha_{2}x^{\mu_{k}\beta(\mu_{\cB}(k))})x^{\mu^{+}_{k}(\dual{k})} \\
		& =  (1+q^{-1/2}x^{\mu_{k}\beta(\mu_{\cB}(k)})x^{\mu_{k}^{+}(\dual{k})} \\
		& = (\rho')_{q^{-1/2}}^{\cB,k}(x^{\dual{k}}) \\
		& = (\rho_{q^{-1/2}}^{\cB,k})^{-1}(x^{\dual{k}})
		\end{align*}
		since
		\begin{align*} \alpha_{1} & =\Omega_{q}^{\mu_{k}\lambda}(\bplus{\mu_{k}\beta(\mu_{\cB}(k))}{\dual{\mu_{k}(\cB)}}-\dual{\mu_{\cB}(k)},-\mu_{k}\beta(\mu_{\cB}(k))) \\
			& = q^{\frac{1}{2}\canform{\bplus{\mu_{k}\beta(\mu_{\cB}(k))}{\dual{\mu_{k}(\cB)}}-\dual{\mu_{\cB}(k)}}{-\mu_{k}\beta(\mu_{\cB}(k))}{\mu_{k}\lambda}} \\
			& = q^{\frac{1}{2}\evform{\bplus{\mu_{k}\beta(\mu_{\cB}(k))}{\dual{\mu_{k}(\cB)}}-\dual{\mu_{\cB}(k)}}{-\mu_{\cB}(k)}} \\
			& = q^{\frac{1}{2}}
		\end{align*}
		and
		\begin{align*}
			\alpha_{2} & = \Omega_{q}^{\mu_{k}\lambda}(\mu_{k}\beta(\mu_{\cB}(k)),\bar{\mu}^{-}_{\mu_{\cB}(k)}(\dual{\mu_{\mu_{k}(\cB)}(\mu_{\cB}(k))}))^{-1} \\
				& = \Omega_{q}^{\mu_{k}\lambda}(\bminus{\mu_{k}\beta(\mu_{\cB}(k))}{\dual{\mu_{k}(\cB)}}-\dual{\mu_{\cB}(k)},\mu_{k}\beta(\mu_{\cB}(k))) \\
				& = q^{\frac{1}{2}\evform{\bminus{\mu_{k}\beta(\mu_{\cB}(k))}{\dual{\mu_{k}(\cB)}}-\dual{\mu_{\cB}(k)}}{\mu_{\cB}(k)}} \\
				& = q^{-\frac{1}{2}}.
		\end{align*}
		Hence 
		\[ \mu_{\uk}M_{\lambda}(\dual{\mu_{\mu_{k}(\cB)}(\mu_{\cB}(k))})=x^{\dual{k}}=x^{\dual{\nu}(\dual{\mu_{\mu_{k}(\cB)}(\mu_{\cB}(k))})} \]
		as required.
\end{proof}

In some instances, it is desirable to allow some of the coefficients (that is, the frozen variables $M_{\lambda}(\dual{f})$ for $f\in \dual{\cB}\setminus \dual{\ex}$) to be invertible in $\cC_{q}(\ex,\cB,\beta,\lambda)$.  The following shows that we may form the localisation of a quantum cluster algebra at a set of coefficients.

\begin{proposition}
	Let $\cB$ be a countable set, $\ex \subseteq \cB$, $\beta\colon \integ[\ex]\to \dual{\integ[\cB]}$ skew-symmetrizable and $\lambda\colon \dual{\integ[\cB]}\to \integ[\cB]$ skew-symmetric and compatible with $\beta$. Let $\bK$ be a field and $q^{\frac{1}{2}}\in \units{\bK}$.
	
	Let $\invfrozen \subseteq \cB\setminus \ex$.  The set $F(\invfrozen)=\{ M_{\lambda}(\dual{f}) \mid f\in \integ[\invfrozen] \}$ is an Ore set in $\cC_{q}(\ex,\cB,\beta,\lambda)$.
\end{proposition}

\begin{proof}
	The set $F$ is multiplicative by definition.  It is left and right Ore since, as noted above in Remark~\ref{r:frozens}, the $M_{\lambda}(\dual{f})=x^{\dual{f}}$ quasi-commute with every quantum cluster variable, \ie with every generator of $\cC_{q}(\ex,\cB,\beta,\lambda)$, from which the claim follows.
\end{proof}

\begin{definition}\label{d:QCA-with-invfroz}
	Let $\cB$ be a countable set, $\ex \subseteq \cB$, $\invfrozen \subseteq \cB\setminus \ex$, $\beta\colon \integ[\ex]\to \dual{\integ[\cB]}$ skew-symmetrizable and $\lambda\colon \dual{\integ[\cB]}\to \integ[\cB]$ skew-symmetric and compatible with $\beta$. Let $\bK$ be a field and $q^{\frac{1}{2}}\in \units{\bK}$.
	
	Define $\cC_{q}(\ex,\cB,\invfrozen,\beta,\lambda)\defeq \cC_{q}(\ex,\cB,\beta,\lambda)[F(\invfrozen)^{-1}]$ for $F(\invfrozen)=\{ M_{\lambda}(\dual{f}) \mid f\in \integ[\invfrozen] \}$.
\end{definition}

We will extend the terminology above and also say that $\cC_{q}(\ex,\cB,\invfrozen,\beta,\lambda)$ is a quantum cluster algebra.  Noting that $\cC_{q}(\ex,\cB,\emptyset,\beta,\lambda)=\cC_{q}(\ex,\cB,\beta,\lambda)$, if we include $\invfrozen$ in the initial data we mean the localised quantum cluster algebra and if it is omitted, we mean the quantum cluster algebra without localisation.

Furthermore, if we write ``let $\cC_{q}=\cC_{q}(\ex,\cB,\invfrozen,\beta,\lambda)$ be a quantum cluster algebra'', we mean that $\cC_{q}$ is the quantum cluster algebra obtained from the given data, which satisfy the conditions of Definition~\ref{d:QCA-with-invfroz}.

\begin{definition}
	For $\uk\in \curly{K}$, define $\qtorusaffine{q}{\mu_{\uk}\lambda}{\mu_{\uk}(\cB)}$ to be the $\bK$-subalgebra of $\curly{F}(\qtorus{q}{\lambda}{\cB})$ generated by the set
	\[ \{ \mu_{\uk}M_{\lambda}(\dual{b}) \mid \dual{b}\in \dual{\mu_{\uk}(\cB)} \} \union \{ \mu_{\uk}M_{\lambda}(\dual{b})^{-1} \mid \dual{b}\in \dual{\mu_{\uk}(\ex)}\union \dual{\invfrozen} \} .\]
\end{definition}

Then $\qtorusaffine{q}{\mu_{\uk}\lambda}{\mu_{\uk}(\cB)}$ is a mixed quantum torus-quantum affine algebra generated by the cluster variables from the cluster obtained from the initial one by the mutation sequence $\uk$, in which the mutable quantum cluster variables (indexed by $\mu_{\uk}(\ex)$) and those frozen variables indexed by $\invfrozen$ are invertible but the remaining frozen variables are not.

In (quantum) cluster algebra theory, we have the (quantum) \emph{Laurent phenomenon} (\cite{FZ-CA1},\cite{BZ-QCA}), which---loosely put---states that every cluster variable can be written as a Laurent polynomial in the variables of an initial (and hence of any) cluster.  In the quantum case, this is expressed as follows, using the notion of the upper quantum cluster algebra.

\begin{definition}
	Let $\cB$ be a finite set, $\ex \subseteq \cB$, $\invfrozen \subseteq \cB\setminus \ex$, $\beta\colon \integ[\ex]\to \dual{\integ[\cB]}$ skew-symmetrizable and $\lambda\colon \dual{\integ[\cB]}\to \integ[\cB]$ skew-symmetric and compatible with $\beta$. Let $\bK$ be a field and $q^{\frac{1}{2}}\in \units{\bK}$.
	
	Define the \emph{upper quantum cluster algebra} $\curly{U}_{q}(\ex,\cB,\invfrozen,\beta,\lambda)$ to be the $\bK$-algebra defined as
	\[ \curly{U}_{q}(\ex,\cB,\invfrozen,\beta,\lambda) = \bigintersection_{\uk\in \curly{K}} \qtorusaffine{q}{\mu_{\uk}\lambda}{\mu_{\uk}(\cB)}.\]
\end{definition}

Then the quantum Laurent phenomenon states that $\cC_{q}(\ex,\cB,\invfrozen,\beta,\lambda)$ embeds in the natural way into its upper quantum cluster algebra, $\curly{U}_{q}(\ex,\cB,\invfrozen,\beta,\lambda)$.  Note too the weaker but also useful statement that $\cC_{q}(\ex,\cB,\invfrozen,\beta,\lambda)$ is a subalgebra of $\qtorus{q}{\lambda}{\cB}$ (and not just the skew-field of fractions of the latter).

\subsection{The abstract quantum cluster structure arising from a quantum cluster algebra}\label{ss:AQCS-from-QCA}

Let $\cC_{q}=\cC_{q}(\ex,\cB,\invfrozen,\beta,\lambda)$ be a quantum cluster algebra.  We can associate to this a canonical abstract quantum cluster structure, as follows.

Let $E(\cC_{q})$ be the directed exchange tree associated to $\cC_{q}$.  That is, $E(\cC_{q})$ has vertex set $\curly{K}$ and an arrow $\mu_{k}\colon \uk\to \underline{l}$ from $\uk=(k_{r},\dotsc ,k_{1})$ to $\underline{l}$ if there exists $k$ such that $\underline{l}=(k,k_{r},\dotsc ,k_{1})$.

\begin{theorem}\label{t:AQCS-from-QCA}
	The quantum cluster algebra $\cC_{q}=\cC_{q}(\ex,\cB,\invfrozen,\beta,\lambda)$ gives rise to an abstract quantum cluster structure $\cC(\cC_{q})=(\cE,\cX,\beta,\lambda,\cA,\ip{\blank}{\blank})$ with
	\begin{enumerate}
		\item $\cE=\cE(E(\cC_{q}))$, the path category of the exchange tree $E(\cC_{q})$;
		\item $\cX\colon \op{\cE}\to \Abcat$ defined by $\cX \uk=\integ[\mu_{\uk}(\ex)]$ and $\cX(\mu_{k}^{\pm}\colon \uk\to \underline{l})=\bar{\mu}_{k}^{\pm}$ (for $\bar{\mu}_{k}^{\pm}$ the isomorphisms of Lemma~\ref{l:F-isos});
		\item $\beta_{\uk}=\mu_{\uk}\beta$;
		\item $\lambda_{\uk}=\mu_{\uk}\lambda$;
		\item $\cA\colon \op{\cE}\to \Abcat$ defined by $\cA \uk=\dual{\integ[\mu_{\uk}(\cB)]}$ and $\cA(\mu_{k}^{\pm}\colon \uk\to \underline{l})=\mu_{k}^{\pm}$ (for $\mu_{k}^{\pm}$ the isomorphisms of Lemma~\ref{l:E-isos}); and
		\item $\ip{\blank}{\blank}_{\uk}\colon \cA \uk \tensor \cX \uk \to \sZ$ given by $\ip{\dual{b}}{c}_{\uk}=\evform{\dual{b}}{c}$ for $\dual{b}\in \dual{\mu_{\uk}(\cB)}$ and $c\in \mu_{\uk}(\ex)$.
	\end{enumerate}
\end{theorem}

\begin{proof}
The non-trivial part is the factorization property for $\beta$, which follows from \eqref{eq:beta-mut} and sign-invariance (Lemma~\ref{l:sign-invar}) to obtain the other version of this equation with the opposite sign, and similarly the factorization property for $\lambda$ via \eqref{eq:lambda-mut} and Lemma~\ref{l:lambda-form-sign-inv}.  Skew-symmetry for $\lambda$ and the required commuting diagram relating $\lambda\tensor_{\integ} \beta$ and $\ip{\blank}{\blank}$ follow from skew-symmetry and compatibility of the respective data for $\cC_{q}$.
\end{proof}

\begin{remark} We may regard the passage from cluster algebra to cluster structure as a form of ``tropicalization'' or ``taking logarithms'', since we have that $\cA$ is given by the reverse operation to the ``exponentiation'' that forms the twisted group algebra $\qtorus{q}{\mu_{\uk}\lambda}{\mu_{\uk}(\cB)}=(\bK \dual{\integ[\mu_{\uk}(\cB)]})^{\Omega_{q}^{\mu_{\uk}(\lambda)}}$ from $\cA \uk=\dual{\integ[\mu_{\uk}(\cB)]}$.
\end{remark}

\subsection{The quantum cluster algebra arising from an abstract quantum cluster structure}\label{ss:QCA-from-AQCS}

Conversely, we can use the data in an abstract quantum cluster structure of finite rank to obtain a quantum cluster algebra.

\begin{theorem}\label{t:QCA-from-AQCS}
	Let $\cC=(\cE,\cX,\beta,\lambda,\cA,\ip{\blank}{\blank})$ be a skew-symmetrizable abstract quantum cluster structure of finite rank. Let $\bK$ be a field and $q^{\frac{1}{2}}\in \units{\bK}$.  Fix $c\in \cE$. Choose $\dual{\ex}$ a basis for $\cX c$ and $\cB$ a basis for $\dual{(\cA c)}$ such that $\delta_{\cX c}(\dual{\ex})\subseteq \cB$.
	
	There exists an associated quantum cluster algebra $\cC_{c,q}=\cC_{q}(\ex,\cB,\invfrozen,\beta,\lambda)$ with
	\begin{enumerate}
		\item $\ex=\delta_{\cX c}(\dual{\ex})$;
		\item $\cB$ the chosen basis for $\dual{(\cA c)}$;
		\item $\invfrozen \subseteq \cB\setminus \ex$ arbitrary;
		\item $\beta\colon \integ[\ex]\to \dual{\integ[\cB]}$ given by $\beta(b)=\beta_{c}(\dual{b})$ for $b\in \ex$ and $\dual{b}=\delta_{\cX c}^{-1}(b)\in \dual{\ex}$; and
		\item $\lambda\colon \dual{\integ[\cB]}\to \integ[\cB]$ given by $\lambda(\dual{b})=\lambda_{c}(\dual{b})$ for $\dual{b}\in \dual{\cB}$.		
	\end{enumerate}
\end{theorem}

\begin{proof}
	By assumption, $\ip{\blank}{\blank}$ is right non-degenerate, meaning that $\delta_{\cX c}$ is injective.  Since $\cC$ is of finite rank, we may then make the required choices.
	
	Now $\beta$ is skew-symmetrizable since $\beta_{c}$ is, $\lambda$ is skew-symmetric since $\lambda_{c}$ is and $\lambda_{c}$ and $\beta_{c}$ are compatible.
	
	Hence the chosen data is what is required to define a quantum cluster algebra as per Definition~\ref{d:QCA-with-invfroz}.
\end{proof}

\begin{remark}
	An important immediate question is the dependence of $\cC_{c,q}$ on the various choices made, not least that of $c\in\cE$.  We will return to this in Section~\ref{s:morphisms-of-reps}.
\end{remark}

\begin{remark}
	Careful examination of the constructions involved enable us to identify the results of applying the above theorem to the product of abstract cluster structures and the coproduct of abstract quantum cluster structures, as obtained in Section~\ref{ss:props-of-ACS-cat}.  
	
	Namely, the cluster algebra associated to the product of abstract cluster structures is the tensor product of the input cluster algebras.  This is in line with expectations from cluster algebra theory.

	The (quantum) cluster algebra associated to the coproduct of abstract (quantum) cluster structures has as underlying vector space the direct sum of the underlying vector spaces of the input (quantum) cluster algebras.
	
	It would be natural to think that we should be considering the direct product algebra structure, but this does not have the ``obvious'' cluster algebra structure: a cluster algebra should have one initial seed and its cluster variables are obtained from this.  But the direct product does not have a single initial cluster, but in some sense would have two!  
	
	Nevertheless, we would argue that this phenomenon suggests that one way to understand why ring-theoretic and categorical approaches to cluster algebras encounter difficulties is that the class of objects under consideration is not large enough.  For example, the failure of many cluster algebras to admit non-trivial surjective maps is in part due to insufficiently many projections in the corresponding category.  As we observed in Section~\ref{ss:initial-terminal}, the categories $\ACScat$ and $\AQCScat$ are simply not as nice as one might have hoped but we suggest that these are better categories than one might end up with if, for example, the ability to take products and coproducts were also absent.
\end{remark}

\sectionbreak
\section{Geometric representations}\label{s:geom-reps}

\subsection{Cluster varieties}\label{ss:cluster-varieties}

We follow \cite{CGSS} in the use of terminology in this section, briefly recalling first the relationship between cluster algebras and cluster varieties.  To begin with, we will discuss the non-quantum situation and then enhance this with the additional data required to consider the quantum case too.

The (affine) \emph{cluster variety} associated to a (commutative) cluster algebra $\cC=\cC(\ex,\cB,\beta)$ is simply the affine scheme $V=\Spec{\cC}$.  We also have $U=\Spec{\cU}$ for $\cU=\cU(\ex,\cB,\cB\setminus \ex,\beta)$ the associated upper cluster algebra (where we choose $\invfrozen=\cB\setminus \ex$, so that all frozen variables are invertible).  

For each $\uk\in \curly{K}$, we have an associated torus $\qtorus{}{}{\mu_{\uk}(\cB)}$ and, via the maps $\mu_{k}^{\pm}$ of Lemma~\ref{l:E-isos}, birational maps of these along which we may glue.  Then $U$ contains the scheme obtained by gluing these tori, $\bigunion_{\uk\in \curly{K}} \qtorus{}{}{\mu_{\uk}(\cB)}$.  This scheme of glued tori is the cluster $\cA$-variety of \cite{FockGoncharov} and is sometimes called the cluster manifold (\cite{GSV-Book}).

Note that $V$ contains but is not equal to the union $\bigunion_{\uk\in \curly{K}} \qtorus{}{}{\mu_{\uk}(\cB)}$ of the tori; the complement is known as the \emph{deep locus} and is where one may find singularities of $V$, should they exist.

As we see below, it is the version with glued tori that most directly relates abstract cluster structures to geometry.  However, it is clear that should one wish to associate a cluster variety in the above sense to an abstract cluster structure, the most direct way to do this is to bootstrap from our previous results: form the associated cluster algebra and then take $\Spec{}$.

\subsubsection{The Poisson cluster varieties arising from an abstract quantum cluster structure}\label{sss:Poisson-CV-from-AQCS}

The philosophy we wish to promote is that an abstract cluster structure consists of data resembling the notion of a cluster modular groupoid from \cite{FockGoncharov}, in the form of $\cE$ (although we do not require $\cE$ to be a groupoid), and a tropical version of the $\cA$- and $\cX$-positive spaces of loc.\ cit., in the form of the sheaves $\cA$ and $\cX$.  Then it is natural to think that we may obtain cluster varieties of one kind or another by exponentiating appropriately.

Let $\bK$ be a field and $\mathbb{G}_{m}$ its multiplicative group; to avoid excessive complications, let us assume $\bK$ is of characteristic zero. Consider the functor $\Hom{\Abcat}{\blank}{\mathbb{G}_{m}}$.  On free Abelian groups of finite rank, this functor gives us a split algebraic torus.  There is also a functor in the opposite direction, by taking the character lattice of a split algebraic torus.  It is the former functor that naturally takes us from an abstract cluster structure to geometry.  Let $\Sch{\bK}$ denote the category of schemes over $\bK$.

Now, let $\cC=(\cE,\cX,\beta,\cA,\ip{\blank}{\blank})$ be an abstract cluster structure of finite rank.  From this, we obtain functors 
\[ A=\Hom{\Abcat}{\blank}{\mathbb{G}_{m}}\circ \cA\colon \cE \to \Sch{\bK} \] and \[  X=\Hom{\Abcat}{\blank}{\mathbb{G}_{m}}\circ \cX\colon \op{\cE}\to \Sch{\bK} .\] These functors give us preschemes $\mathbb{A}$ and $\mathbb{X}$, obtained by gluing the tori $\qtorus{\cA}{}{c}=\Hom{\Abcat}{\cA c}{\mathbb{G}_{m}}$ and $\qtorus{\cX}{}{c}=\Hom{\Abcat}{\cX c}{\mathbb{G}_{m}}$ using the morphisms of tori $A\alpha^{+}\colon \qtorus{\cA}{}{c}\to \qtorus{\cA}{}{d}$ (for $\alpha^{+}\colon c\to d$) and $X\alpha^{+}$, respectively\footnote{In this level of generality, it is almost surely the case that choosing $A\alpha^{-}$ and $X\alpha^{-}$ will give different schemes in many examples.  Indeed, one has to work relatively hard to prove sign invariance claims for cluster algebras.}.

Notwithstanding some minor differences due to our setup, the schemes $\mathbb{A}$ and $\mathbb{X}$ are the direct analogues of what are termed cluster $\cA$- and $\cX$-schemes in \cite{FockGoncharov}.  

To make stronger statements, we need to connect to the previous set of constructions.  Via Theorem~\ref{t:QCA-from-AQCS} (forgetting the additional quantum structure there), we have that from an abstract cluster structure of finite rank with skew-symmetric part, we obtain a cluster algebra $\cC_{c}$ associated to a chosen $c\in \cE$ (effectively, a chosen initial seed).  Then, following through the constructions, in some cases---depending on precisely which localisations are made---the algebra of regular functions on $\mathbb{A}$ will recover the upper cluster algebra $\cU_{c}$.

That is, an abstract cluster structure gives us cluster schemes and when that abstract cluster structure also gives rise to a cluster algebra, the relationship between the $\cA$-scheme and the algebra is as expected.  We have not considered in detail the $\cX$-side analogue of the ($\cA$-side) cluster algebra.  However, one may use the above as motivation for defining the $\cX$-side algebras via the geometry and check that this gives back the expected algebraic constructions when suitable additional assumptions are made.

Note too that, tautologically (if we mimic the Fock--Goncharov definition of positive space as a functor of the type above) the factorization $\beta$ is a map from $\mathbb{A}$ to $\mathbb{X}$, which coincides with the map $p$ of \cite{FockGoncharov} in the aforementioned situation.

In order to handle the quantum case, we briefly recall the Poisson cluster theory introduced in \cite{GSV-Book} and \cite{FockGoncharov}.  A Poisson manifold $M$ is a smooth real manifold whose algebra $C^{\infty}(M)$ of smooth functions is a Poisson algebra, i.e.\ it admits a skew-symmetric bilinear map $\Pip{\blank}{\blank}\colon C^{\infty}(M)\cross C^{\infty}(M)\to C^{\infty}(M)$ satisfying the Leibniz and Jacobi identities (see \cite[Section~1.3]{GSV-Book}).  More generally, an affine Poisson variety is one whose algebra of rational functions is a Poisson algebra.

Gekhtman--Shapiro--Vainshtein's notion of a Poisson structure that is compatible with a cluster algebra structure is via the notion of log-canonical sets of functions (\cite[Section~4.1]{GSV-Book}).  Namely, a collection of rational functions $\{ f_{i} \}$ is said to be log-canonical with respect to a Poisson bracket $\Pip{\blank}{\blank}$ if there exist integers $\omega_{ij}$ such that $\Pip{f_{i}}{f_{j}}=\omega_{ij}f_{i}f_{j}$, for all $i,j$.  Given a cluster algebra $A$, a Poisson bracket $\Pip{\blank}{\blank}$ on $A$ is said to be compatible with the cluster structure if every cluster of $A$ is log-canonical.

Then the main result we will need (\cite[Theorem~4.5]{GSV-Book}) is that if the exchange matrix $B$ of a cluster of $A$ has full rank\footnote{Hence all do, since rank is preserved under mutation, cf. Lemma~\ref{l:constant-rank}.}, then there exists a compatible Poisson structure on $A$ and furthermore, such a Poisson structure is determined by any choice of skew-symmetric integer matrix $L$ such that $B^{T}L$ is zero except on diagonal entries corresponding to mutable indices, where the value is strictly positive.  The Poisson algebra structure is obtained by using $L$ to define a Poisson structure on the torus associated to the cluster (with respect to which the natural transcendence base is log-canonical with coefficients encoded in $L$) and restricting, since the Laurent phenomenon tells us that the cluster algebra is contained in this Laurent polynomial ring.

Note that, up to an unimportant scale factor, this condition is exactly that of Definition~\ref{d:compatible}, i.e.\ compatibility for quantum cluster structures.  That is, a Poisson $\cA$-cluster variety is obtained from an abstract quantum cluster structure by the above construction, using the component $\lambda_{c}$ to endow the torus $Ac$ with a Poisson structure.  Compatibility then ensures that these Poisson tori glue to give a Poisson $\cA$-cluster variety.

\begin{remark}
	It is important to note that there are a number of different additional geometric structures on cluster varieties.  Indeed, some authors refer to the $\cA$-cluster varieties as cluster $K_{2}$-varieties and to $\cX$-cluster varieties as cluster Poisson varieties.  There is a potential for confusion here, so let us explain.
	
	Any $\cX$-cluster variety, in our sense above, admits a Poisson structure, determined by $\beta$.  Specifically, let $\Pip{\blank}{\blank}_{\cX c}$ be a skew-symmetrization of the form $\ip{\blank}{\blank}_{\cX c}$ of Definition~\ref{d:sX-form}; for this, we should assume $\beta_{c}$ is skew-symmetrizable.  This determines a log-canonical Poisson structure for $\qtorus{\cX}{}{c}$ and the factorization property of $\beta$ ensures that these glue to endow $\mathbb{X}$ with the structure of a Poisson (cluster) variety.
	
	For the data associated to a cluster algebra, or (unquantized) abstract cluster structure, there is no such Poisson structure on $\mathbb{A}$.  However, as above, we can restore symmetry by choosing a quantum structure $\lambda$; by the same mechanism (i.e.\ via the skew-symmetric form $\ip{a}{b}_{\cA c}=\evform{\lambda_{c}(a)}{b}$ as per Definition~\ref{d:A-form}), we obtain the Poisson cluster variety $\mathbb{A}$.
	
	Without the additional choice of quantum data, $\mathbb{A}$ has only the structure of a $K_{2}$-variety.  As explaining this in detail would take us too far afield, we refer the reader to \cite{FockGoncharov} and associated literature.
\end{remark}

\subsubsection{The abstract quantum cluster structure arising from Poisson cluster varieties}\label{sss:AQCS-from-Poisson-CV}

As the above discussion should make clear, the data of a pair of Poisson cluster varieties  $(\mathbb{A},\mathbb{X})$ gives us a pair of families of Poisson tori together with some gluing maps, where both families share a common indexing set.  Then the character lattices of the (split) tori give us the specification on objects of our desired free Abelian presheaves, e.g. set $\cA c=\Hom{}{\qtorus{\cA}{}{c}}{\mathbb{G}_{m}}$.  The Poisson structures on $\mathbb{X}$ and $\mathbb{A}$ give us $\beta$ and $\lambda$ respectively and as noted in Section~\ref{ss:quantum-str}, $\ip{\blank}{\blank}$ is determined by these (due to compatibility).

Hence, pairs of Poisson cluster varieties give us abstract quantum cluster structures.  Note that we need more than simply the underlying variety: we must know its \emph{cluster variety} structure.  We have also taken advantage of the Poisson structure on $\mathbb{A}$ to recover the pairing $\ip{\blank}{\blank}$; we expect that it is sufficient to know the $K_{2}$-structure but leave this for future investigation\footnote{This ought only to be relevant when $\beta$ is not skew-symmetric, since the delicacy around possible non-uniqueness of $\ip{\blank}{\blank}$ is related to the choice of skew-symmetrizers.}.  

\subsection{Triangulations and geometric models}\label{ss:triangulations}

We claim in this section that there is a natural abstract cluster structure associated to a triangulation of a surface, where we will follow the originators of the cluster approach to these (\cite{FST}) and consider connected oriented Riemann surfaces $\mathbf{S}$ with boundary and a finite set of marked points $\mathbf{M}$, with at least one marked point on each boundary component.  For simplicity, we will not address punctured surfaces, tagging of arcs, etc.

We need to begin by identifying the simple directed graph $E$ underlying our abstract cluster structure.  Let the vertices of $E$ be triangulations of $\mathbf{S}$ (in the sense of \cite{FST}, i.e.\ a maximal collection of non-crossing arcs in the surface between marked points, up to isotopy).  We include the boundary arcs in our triangulations.

Then if $c$ and $d$ are triangulations that differ by the flip of exactly one arc, we have arrows $c\to d$ and $d\to c$ in $E$; that is, $E$ is bi-directed and indeed is the bi-directed graph associated to the usual exchange graph for triangulations.  That is, this is precisely the bi-directed version of the graph $\mathbf{E}^{\circ}(\mathbf{S},\mathbf{M})$ of \cite{FST}.

Let $\curly{A}(c)$ denote the set of arcs of the triangulation $c$, $\curly{B}(c)$ the set of boundary arcs of $c$ and $\curly{T}(c)$ the set of (ideal) triangles of $c$.  In a triangulation, every non-boundary arc is an edge of exactly two distinct triangles.  Let us call such a pair of triangles (sharing a unique non-boundary arc) a \emph{quadrilateral} and let $\curly{Q}(c)$ denote the set of quadrilaterals of the triangulation $c$.  

For a non-boundary arc $a\in \curly{A}(c)$, write $q(a)$ for the quadrilateral containing $a$; more formally, let $q\colon (\curly{A}(c)\setminus \curly{B}(c))\to \curly{Q}(c)$ be the function sending a non-boundary arc to the quadrilateral containing it. This is a bijection, since by the non-crossing property, each quadrilateral contains exactly one such arc (which is the common edge of the two triangles comprising the quadrilateral).

The remaining data for an abstract cluster structure is chosen as follows:
\begin{enumerate}
	\item $\cX\colon \op{\cE}\to \Abcat$ is defined on objects (i.e.\ triangulations) by $\cX c=\integ[\curly{Q}(c)]$, the free Abelian group on the set $\curly{Q}(c)$;
	\item $\cA\colon \cE\to \Abcat$ is defined on objects (i.e.\ triangulations) by $\cA c=\integ[\curly{A}(c)]$, the free Abelian group on the set $\curly{A}(c)$;
	\item $\ip{\blank}{\blank}$ is defined by $\ip{a}{q}_{c}=1$ if $a\in \cA(c)\setminus \cB(c)$ and $q(a)=q$ and equals $0$ otherwise.
\end{enumerate}

This is not yet sufficient, of course: we need to determine the effect of $\cA$ and $\cX$ on morphisms and to define the factorization $\beta$.  However, in order to see that what we obtain satisfies the required properties, we make these definitions together.

We first re-interpret the construction of the exchange matrix associated to a triangulation, as follows.  Fix a triangulation $c$; for ease of the discussion that follows, assume that $c$ has no self-folded triangles.  As above, for each quadrilateral $q\in \curly{Q}(c)$, there is a unique non-boundary arc $a$ such that $q(a)=q$ and hence exactly two triangles in the triangulation having $a$ as an edge.  In each of these, there are two other arcs forming the edges of the triangle.

Thus, for each quadrilateral $q\in \curly{Q}(c)$, the triangulation contains four arcs which we call $q_{\circ,+}$, $q_{\circ,-}$, $q_{\bullet,+}$ and $q_{\bullet,-}$ such that $(a,q_{\circ,+},q_{\circ,-})$ and $(a,q_{\bullet,+},q_{\bullet,-})$ are the two triangles of $q$ with the arcs written in order in accordance with the orientation of the surface (so that $q_{\circ,+}$ follows $a$ with respect to the orientation, etc.).

Then we may define a $\integ$-linear map $\beta_{c}\colon \integ[\curly{Q}(c)]\to \integ[\curly{A}(c)]$ by 
\begin{equation}\label{eq:beta-for-surfaces} \beta_{c}(q)=(q_{\circ,+}+q_{\bullet,+})-(q_{\circ,-}+q_{\bullet,-}). \end{equation}
As we see in the example below, $\beta_{c}(q)$ is computed by taking the pairs of opposite edges of the quadrilateral, summing those of the same sign and taking the difference of these sums.  Equivalently, we take a signed sum of the edges of the quadrilateral, with signs derived from the orientation with respect to the arc $q(a)$.

We excluded the case of self-folded triangles above for simplicity, but note that the prescription in \cite[Definition~4.1]{FST} covers this case in addition; a more general definition of $\beta_{c}$ is then obtained by adjusting the exchange matrix of loc.\ cit. to include rows corresponding to boundary vertices and to change the column indexing set to $\curly{Q}(c)$ (via the bijection with $\curly{A}(c)\setminus \curly{B}(c)$).

Aligning with the approach and definitions of Section~\ref{ss:toric-frames}, let $c$ and $d$ be triangulations such that $d$ is obtained from $c$ by flipping the non-boundary arc $a$ in the quadrilateral $q=q(a)$; that is, we remove $a$ and replace it by the unique distinct arc $a'$ such that the result is again a triangulation. Then let us write $d=\mu_{q(a)}(c)$ to indicate that $d$ is obtained from $c$ by mutating at $a$ (i.e.\ by flip at $a$) and write $\mu_{c}(a)$ for the arc in $\mu_{q(a)}(c)$ obtained by flipping $a$.  Then $\cA(\mu_{q(a)}(c))$, the set of arcs of the mutated triangulation is equal to $(\cA(c)\setminus \{ a \})\union \{ \mu_{c}(a) \}$, corresponding to $\mu_{k}(\cB)$ as in Definition~\ref{d:mut-of-index-sets}.  

Already we see a slight deviation from the setting of Definition~\ref{d:mut-of-index-sets}: our ``mutable indices'' are a subset  $\ex \subseteq \cB=\cA(c)$ but will be identified with $\curly{Q}(c)$.  However, it follows from the above discussion that there is an injective function $q^{*}\colon \curly{Q}(c)\to \curly{A}(c)$ taking each quadrilateral to the unique non-boundary arc it contains, allowing us to mostly gloss over this tacit identification. This introduces a modest, but not unnatural, asymmetry: mutation of arcs is denoted $\mu_{c}(a)$ but mutation of triangulations is $\mu_{q}(c)$.

Note that the above inner product $\canform{\blank}{\blank}{c}$ can then also be expressed as $\canform{a}{q}{c}\delta_{a,q^{*}(q)}$.

Mutation of quadrilaterals is denoted $\mu_{c}(q)$: by this we mean that $\mu_{c}(q)$ is the quadrilateral in $\mu_{q}(c)$ such that $\mu_{c}(q^{*}(q))=q^{*}(\mu_{q}(c))$ (that is, the arc obtained by flipping the non-boundary arc $q^{*}(q)$ in $q$ is the unique non-boundary arc of the mutation $\mu_{c}(q)$).

Now, from $\beta_{c}$ we obtain isomorphisms $\mu_{q}^{\pm}\colon \integ[\cX(c)]\to \integ[\cX(\mu_{q}(c))]$,
\[ 	\mu_{q}^{\pm}(p)= \begin{cases} p+\evform{\bplusminus{\beta_{c}(q)}{}}{p}\mu_{c}(q) & \text{if}\ p\neq q \\ -\mu_{c}(q) & \text{if}\ p=q \end{cases} \]
and $\mu_{q}^{\pm}\colon \integ[\cA(c)]\to \integ[\cA(\mu_{q}(c))]$,
\[ \mu_{q}^{\pm}(a)=\begin{cases} a & \text{if}\ a\neq q^{*}(q) \\ \bplusminus{\beta_{c}(q)}{}-\mu_{c}(a) & \text{if}\ a=q^{*}(q) \end{cases}. \]

These formul\ae\ are exactly as in Lemmas~\ref{l:F-isos} and \ref{l:E-isos} and as noted earlier, these are precisely the maps required to mutate an exchange matrix in accordance with Fomin--Zelevinsky (matrix) mutation.  This has two consequences: firstly, we immediately deduce that the family of maps $\beta_{c}$ determines a factorization with respect to $\cX\alpha^{\pm}=(\mu_{q}^{\pm})^{-1}$ and $\cA \alpha^{\pm}=\mu_{q}^{\pm}$, since we have checked this in Section~\ref{ss:toric-frames}.  

Secondly, the fundamental result underlying the claim that triangulations of surfaces model cluster combinatorics is that the exchange matrices of triangulations obtained from each other by flips of arcs are related by Fomin--Zelevinsky matrix mutation.  Therefore, using the above prescription of $\cX \alpha^{\pm}$ and $\cA \alpha^{\pm}$ does indeed give us that $\mu_{q}\beta_{c}=\mu_{q}^{\pm}\circ \beta_{c}\circ (\mu_{q}^{\pm})^{-1}$ is equal to $\beta_{\mu_{q}(c)}$, which \emph{a priori} is not clear.

We illustrate and interpret the above in the canonical example, triangulations of a hexagon ($\mathbf{S}$ is a closed disc and there are six marked points on the boundary; we straighten the boundary edges between the marked points to obtain the hexagon). Let us take the anti-clockwise orientation on this surface.

First, we recall arc mutation: on the left in the following diagram is a triangulation $c$ of the hexagon, with an arc $a$ and the quadrilateral $q=q(a)$ with edges $q_{\circ,+}$, $q_{\circ,-}$, $q_{\bullet,+}$ and $q_{\bullet,-}$.  Note that $a$ is the common boundary edge of the two triangles $(a,q_{\circ,+},q_{\circ,-})$ and $(a,q_{\bullet,+},q_{\bullet,-})$.  In the middle picture, the arc $a$ is removed and in the third it is replaced by $\mu_{c}(a)$, the unique distinct (non-boundary) arc giving us another triangulation $\mu_{q}(c)$.
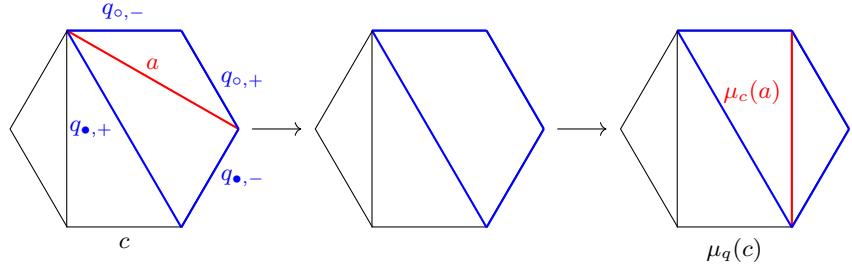
	\begin{center}
		\scalebox{1}{\begin{tikzpicture}[font=\footnotesize]

\node (Hexagon1) at (90:0cm) [regular polygon, regular polygon sides=6, draw,minimum size=3cm,label=below:{$c$}] {};
\node (Hexagon2) [right=of Hexagon1, regular polygon, regular polygon sides=6, draw,minimum size=3cm] {};
\node (Hexagon3) [right=of Hexagon2, regular polygon, regular polygon sides=6, draw,minimum size=3cm,label=below:{$\mu_{q}(c)$}] {};

\draw[->,shorten <=5pt,shorten >=5pt] (Hexagon1) to (Hexagon2);
\draw[->,shorten <=5pt,shorten >=5pt] (Hexagon2) to (Hexagon3);

\draw[red,thick] (Hexagon1.corner 2) -- (Hexagon1.corner 6) node [midway,above] {$a$};
\draw[blue,thick] (Hexagon1.corner 2) -- (Hexagon1.corner 5) node [midway,left] {$q_{\bullet,+}$};
\draw[blue,thick] (Hexagon1.corner 5) -- (Hexagon1.corner 6) node [midway,right] {$q_{\bullet,-}$};
\draw[blue,thick] (Hexagon1.corner 6) -- (Hexagon1.corner 1) node [midway,right] {$q_{\circ,+}$};
\draw[blue,thick] (Hexagon1.corner 1) -- (Hexagon1.corner 2) node [midway,above] {$q_{\circ,-}$};
\draw[] (Hexagon1.corner 2) -- (Hexagon1.corner 4);

\draw[blue,thick] (Hexagon2.corner 2) -- (Hexagon2.corner 5);
\draw[blue,thick] (Hexagon2.corner 5) -- (Hexagon2.corner 6);
\draw[blue,thick] (Hexagon2.corner 6) -- (Hexagon2.corner 1);
\draw[blue,thick] (Hexagon2.corner 1) -- (Hexagon2.corner 2);
\draw[] (Hexagon2.corner 2) -- (Hexagon2.corner 4);

\draw[red,thick] (Hexagon3.corner 1) -- (Hexagon3.corner 5) node [midway,above left,yshift=0.5em] {$\mu_{c}(a)$};
\draw[blue,thick] (Hexagon3.corner 2) -- (Hexagon3.corner 5) node [midway,left] {};
\draw[blue,thick] (Hexagon3.corner 5) -- (Hexagon3.corner 6) node [midway,right] {};
\draw[blue,thick] (Hexagon3.corner 6) -- (Hexagon3.corner 1) node [midway,right] {};
\draw[blue,thick] (Hexagon3.corner 1) -- (Hexagon3.corner 2) node [midway,above] {};
\draw[] (Hexagon3.corner 2) -- (Hexagon3.corner 4);

\end{tikzpicture}}
	\end{center}

Borrowing an idea we will see much more of in Section~\ref{ss:cluster-cats}, the values of the maps $\cA \alpha^{\pm}=\mu_{q}^{\pm}$ (for $\alpha^{\pm}\colon c\to \mu_{q}(c)$) on $a$ can be written explicitly as
	\begin{align*}
		\mu_{q}^{+}(a)=(q_{\circ,+}+q_{\bullet,+})-\mu_{c}(a)=q_{\circ,+}-\mu_{c}(a)+q_{\bullet,+}; \\
		\mu_{q}^{-}(a)=(q_{\circ,-}+q_{\bullet,-})-\mu_{c}(a)=q_{\circ,-}-\mu_{c}(a)+q_{\bullet,-}. 
	\end{align*}
	Note that $\mu_{q}^{\pm}$ acts as the identity on all other arcs, as one would expect, since these are unchanged.
	
	Then these maps can be thought of taking $a$ to its two ``$\mu_{q}(c)$-approximations''.  For $a$ is not an arc in $\mu_{q}(c)$, but we can consider what combination of arcs might have the same endpoints.  Specifically, we see that traversing the sequences of arcs $(q_{\circ,+},\mu_{c}(a),q_{\bullet,+})$ or $(q_{\circ,-},\mu_{c}(a),q_{\bullet,-})$ traces a path that starts and ends where $a$ did:
	\begin{center}
		\scalebox{1}
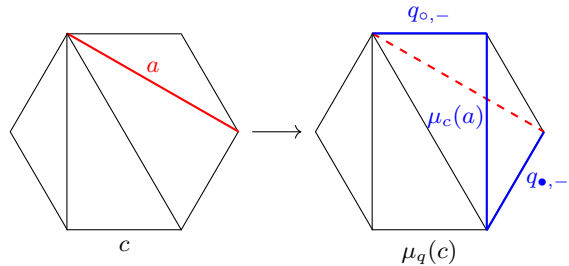
{\begin{tikzpicture}[font=\footnotesize]

\node (Hexagon1) at (90:0cm) [regular polygon, regular polygon sides=6, draw,minimum size=3cm,label=below:{$c$}] {};
\node (Hexagon2) [right=of Hexagon1, regular polygon, regular polygon sides=6, draw,minimum size=3cm,label=below:{$\mu_{q}(c)$}] {};

\draw[->,shorten <=5pt,shorten >=5pt] (Hexagon1) to (Hexagon2);

\draw[red,thick] (Hexagon1.corner 2) -- (Hexagon1.corner 6) node [midway,above] {$a$};
\draw[] (Hexagon1.corner 2) -- (Hexagon1.corner 5);
\draw[] (Hexagon1.corner 2) -- (Hexagon1.corner 4);

\draw[red,dashed,thick] (Hexagon2.corner 2) -- (Hexagon2.corner 6) node [midway,above] {};
\draw[] (Hexagon2.corner 2) -- (Hexagon2.corner 5);
\draw[blue,thick] (Hexagon2.corner 5) -- (Hexagon2.corner 6) node [midway,right] {$q_{\bullet,-}$};
\draw[blue,thick] (Hexagon2.corner 1) -- (Hexagon2.corner 2) node [midway,above] {$q_{\circ,-}$};
\draw[] (Hexagon2.corner 2) -- (Hexagon2.corner 4);
\draw[blue,thick] (Hexagon2.corner 1) -- (Hexagon2.corner 5) node [midway,left,yshift=0.5em,xshift=0.3em] {$\mu_{c}(a)$};


\end{tikzpicture}}
	\end{center}
	
Quadrilateral mutation can also be visualised.  The following illustrates the three quadrilaterals $\curly{Q}(c)=\{q=q_{1},q_{2},q_{3}\}$ in $c$ and the three quadrilaterals $\curly{Q}(\mu_{q}(c))=\{q_{1}',q_{2}',q_{3}'\}$ in $\mu_{q}(c)$.  Notice that $q_{1}$ contains our chosen arc $a$, $q_{2}$ has $a$ as an edge and $q_{3}$ does not intersect with $a$.
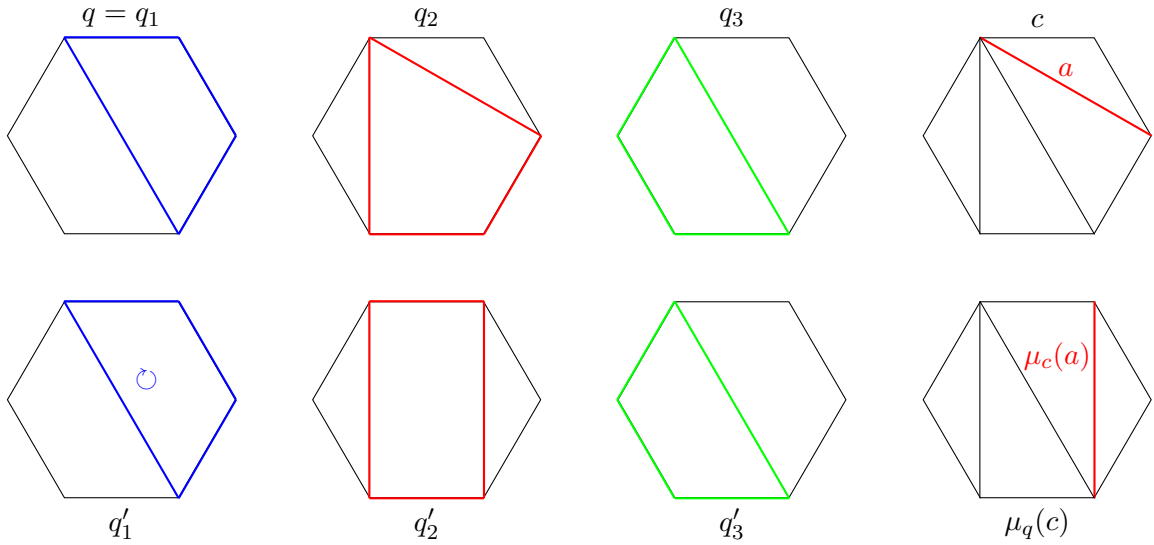
	\begin{center}
		\scalebox{1}{\begin{tikzpicture}

\begin{scope}
\node (Hexagon1) at (90:0cm) [regular polygon, regular polygon sides=6, draw,minimum size=3cm,label=above:{$q=q_{1}$}] {};
\node (Hexagon2) [right=of Hexagon1, regular polygon, regular polygon sides=6, draw,minimum size=3cm,label=above:{$q_{2}$}] {};
\node (Hexagon3) [right=of Hexagon2, regular polygon, regular polygon sides=6, draw,minimum size=3cm,label=above:{$q_{3}$}] {};
\node (Hexagon4) [right=of Hexagon3, regular polygon, regular polygon sides=6, draw,minimum size=3cm,label=above:{$c$}] {};

\draw[blue,thick] (Hexagon1.corner 2) -- (Hexagon1.corner 5);
\draw[blue,thick] (Hexagon1.corner 5) -- (Hexagon1.corner 6);
\draw[blue,thick] (Hexagon1.corner 6) -- (Hexagon1.corner 1);
\draw[blue,thick] (Hexagon1.corner 1) -- (Hexagon1.corner 2);

\draw[red,thick] (Hexagon2.corner 2) -- (Hexagon2.corner 6);
\draw[red,thick] (Hexagon2.corner 6) -- (Hexagon2.corner 5);
\draw[red,thick] (Hexagon2.corner 5) -- (Hexagon2.corner 4);
\draw[red,thick] (Hexagon2.corner 4) -- (Hexagon2.corner 2);

\draw[green,thick] (Hexagon3.corner 2) -- (Hexagon3.corner 5);
\draw[green,thick] (Hexagon3.corner 5) -- (Hexagon3.corner 4);
\draw[green,thick] (Hexagon3.corner 4) -- (Hexagon3.corner 3);
\draw[green,thick] (Hexagon3.corner 3) -- (Hexagon3.corner 2);

\draw[red,thick] (Hexagon4.corner 2) -- (Hexagon4.corner 6) node [midway,above] {$a$};
\draw[] (Hexagon4.corner 2) -- (Hexagon4.corner 5);
\draw[] (Hexagon4.corner 2) -- (Hexagon4.corner 4);
\end{scope}

\begin{scope}[yshift=-3.5cm]
\node (Hexagon1) at (90:0cm) [regular polygon, regular polygon sides=6, draw,minimum size=3cm,label=below:{$q_{1}'$}] {};
\node (Hexagon2) [right=of Hexagon1, regular polygon, regular polygon sides=6, draw,minimum size=3cm,label=below:{$q_{2}'$}] {};
\node (Hexagon3) [right=of Hexagon2, regular polygon, regular polygon sides=6, draw,minimum size=3cm,label=below:{$q_{3}'$}] {};
\node (Hexagon4) [right=of Hexagon3, regular polygon, regular polygon sides=6, draw,minimum size=3cm,label=below:{$\mu_{q}(c)$}] {};

\draw[blue,thick] (Hexagon1.corner 2) -- (Hexagon1.corner 5) node [midway,above,xshift=0.33cm] {$\circlearrowright$};
\draw[blue,thick] (Hexagon1.corner 5) -- (Hexagon1.corner 6);
\draw[blue,thick] (Hexagon1.corner 6) -- (Hexagon1.corner 1);
\draw[blue,thick] (Hexagon1.corner 1) -- (Hexagon1.corner 2);

\draw[red,thick] (Hexagon2.corner 2) -- (Hexagon2.corner 1);
\draw[red,thick] (Hexagon2.corner 1) -- (Hexagon2.corner 5);
\draw[red,thick] (Hexagon2.corner 5) -- (Hexagon2.corner 4);
\draw[red,thick] (Hexagon2.corner 4) -- (Hexagon2.corner 2);

\draw[green,thick] (Hexagon3.corner 2) -- (Hexagon3.corner 5);
\draw[green,thick] (Hexagon3.corner 5) -- (Hexagon3.corner 4);
\draw[green,thick] (Hexagon3.corner 4) -- (Hexagon3.corner 3);
\draw[green,thick] (Hexagon3.corner 3) -- (Hexagon3.corner 2);

\draw[red,thick] (Hexagon4.corner 1) -- (Hexagon4.corner 5) node [midway,left,yshift=1.5em,xshift=0.25em] {$\mu_{c}(a)$};
\draw[] (Hexagon4.corner 2) -- (Hexagon4.corner 5);
\draw[] (Hexagon4.corner 2) -- (Hexagon4.corner 4);
\end{scope}

\end{tikzpicture}}
	\end{center}

Now as before, working through the definitions of $\mu_{q}^{\pm}$, we find that
	\begin{align*}
		\mu_{q}^{+}(q_{1}) & =-q_{1}' & \mu_{q}^{+}(q_{2}) & =q_{2}'-q_{1}' &  \mu_{q}^{+}(q_{3}) & =q_{3}' \\
		\mu_{q}^{-}(q_{1}) & =-q_{1}' & \mu_{q}^{-}(q_{2}) & =q_{2}'+q_{1}' &  \mu_{q}^{-}(q_{3}) & =q_{3}' 
	\end{align*}
That is, $\mu_{q}^{\pm}$ is algebraically a sign change on $q=q_{1}$ (which we can interpret as a change in orientation), is more complicated on $q_{2}$ (but can be seen as removing the triangle in the intersection of $q$ and $q_{2}$ and replacing it by a triangle of $q_{1}'$, noting that the triangles $q_{1}'$ are determined by $\mu_{c}(a)$) and is ``the identity'' on $q_{3}$ as one would expect.  The general rule is unsurprisingly more complex, involving (signed) adjacency to the quadrilateral $q$.

Finally, as noted above, the map $\beta_{c}$ can be thought of as a ``boundary map'', i.e.\ a signed sum of edges of quadrilaterals:
	\begin{center}
		\scalebox{1}
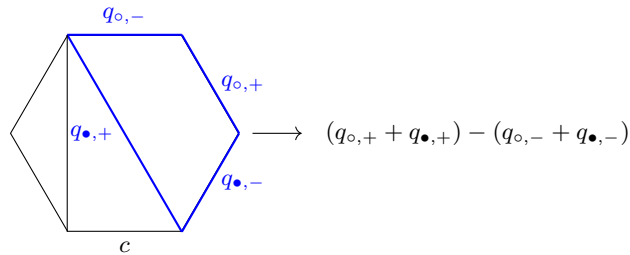
{\begin{tikzpicture}[font=\footnotesize]

\node (Hexagon1) at (90:0cm) [regular polygon, regular polygon sides=6, draw,minimum size=3cm,label=below:{$c$}] {};
\node (Hexagon2) [right=of Hexagon1] {$(q_{\circ,+}+q_{\bullet,+})-(q_{\circ,-}+q_{\bullet,-})$};

\draw[->,shorten <=5pt,shorten >=5pt] (Hexagon1) to (Hexagon2);


\draw[blue,thick] (Hexagon1.corner 2) -- (Hexagon1.corner 5) node [midway,left] {$q_{\bullet,+}$};
\draw[blue,thick] (Hexagon1.corner 5) -- (Hexagon1.corner 6) node [midway,right] {$q_{\bullet,-}$};
\draw[blue,thick] (Hexagon1.corner 6) -- (Hexagon1.corner 1) node [midway,right] {$q_{\circ,+}$};
\draw[blue,thick] (Hexagon1.corner 1) -- (Hexagon1.corner 2) node [midway,above] {$q_{\circ,-}$};
\draw[] (Hexagon1.corner 2) -- (Hexagon1.corner 4);

\end{tikzpicture}}
	\end{center}
Computing $\beta_{c}$ on $\curly{Q}(c)$, we obtain the exchange matrix of the triangulation $c$; converting this to a quiver in the usual way gives us the following, where we have just written $*$ for arcs we have not labelled in the above diagrams:
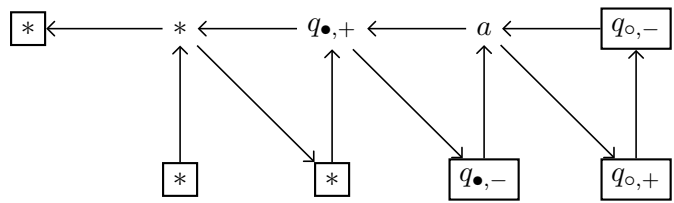
\begin{center}
	\scalebox{1}{\begin{tikzpicture}[node distance=2cm,on grid,>=angle 90]

\node (10) at (0,0) [rectangle,draw=black,thick] {$*$};
\node (11) [right=of 10] {$*$}; 
\node (12) [right=of 11] {$q_{\bullet,+}$};
\node (13) [right=of 12] {$a$};
\node (14) [right=of 13,rectangle,draw=black,thick] {$q_{\circ,-}$};

\node (21) [below=of 11,rectangle,draw=black,thick] {$*$}; 
\node (22) [right=of 21,rectangle,draw=black,thick] {$*$};
\node (23) [right=of 22,rectangle,draw=black,thick] {$q_{\bullet,-}$};
\node (24) [right=of 23,rectangle,draw=black,thick] {$q_{\circ,+}$};

\draw[semithick,->] (11) to (10);
\draw[semithick,->] (12) to (11);
\draw[semithick,->] (13) to (12);
\draw[semithick,->] (14) to (13);

\draw[semithick,->] (21) to (11);
\draw[semithick,->] (22) to (12);
\draw[semithick,->] (23) to (13);
\draw[semithick,->] (24) to (14);

\draw[semithick,->] (11) to (22);
\draw[semithick,->] (12) to (23);
\draw[semithick,->] (13) to (24);

\end{tikzpicture}}
\end{center}

To summarise, therefore, to a pair $(\mathbf{S},\mathbf{M})$ we may associate an abstract cluster structure $\cC(\mathbf{S},\mathbf{M})=\cC(\cE,\cX,\beta,\cA,\ip{\ }{\ })$ where
\begin{enumerate}
	\item $\cE$ is the signed path category of $E$, the bi-directed version of the graph $\mathbf{E}^{\circ}(\mathbf{S},\mathbf{M})$ of \cite{FST};
	\item $\cX\colon \op{\cE}\to \Abcat$ is defined on triangulations by $\cX c=\integ[\curly{Q}(c)]$, the free Abelian group on the set of quadrilaterals of $c$, $\curly{Q}(c)$ and on flips by the ($\cX$-)mutation maps $(\mu_{q}^{\pm})^{-1}$;
	\item $\cA\colon \cE\to \Abcat$ is defined on triangulations by $\cA c=\integ[\curly{A}(c)]$, the free Abelian group on the set of arcs of $c$, $\curly{A}(c)$ and on flips by the ($\cA$-)mutation maps $\mu_{q}^{\pm}$; and
	\item $\ip{\blank}{\blank}$ is defined by $\ip{a}{q}_{c}=\delta_{q(a),q}=\delta_{a,q^{*}(q)}$, that is, $\ip{a}{q}_{c}=1$ if $a$ is contained in $q$ in $c$ and equals $0$ otherwise.
\end{enumerate}

Although we shall return to this later, we note now that it should be clear that the (usual, $\cA$-side) cluster algebra associated to $\cC(\mathbf{S},\mathbf{M})$ as obtained via the non-quantum version of Theorem~\ref{t:QCA-from-AQCS} is precisely the usual cluster algebra associated to triangulations of $(\mathbf{S},\mathbf{M})$, as per \cite{FST}, modulo some trivial work to correctly identify the relevant indexing sets.  For this simply amounts to the claim that we have encoded the correct exchange matrices in our abstract cluster structure.

We note, however, that $\cC(\mathbf{S},\mathbf{M})$ contains information not recorded (or at least, not recorded explicitly) in the associated cluster algebra, namely the $\cX$-side quadrilateral mutation and its relationship with arc mutation.  Note too that unlike in previous examples, we make no claim to be able to produce a geometric model from an abstract cluster structure: indeed, given that exchange matrices in geometric models have very restrictive properties, it is clear no such general construction could exist.

\sectionbreak
\section{Categorical representations}\label{s:cat-reps}

\subsection{Cluster categories}\label{ss:cluster-cats}
	
In this section, we show that a cluster category (in the sense of \cite{CSFRT}) has a natural associated abstract cluster structure. Indeed, as with the other cases above, the definition of an abstract cluster structure has been developed with this in mind.  The proofs required to establish the claim here are substantial: we will not give the details, as they are contained in our principal source, namely the recent work of the first author and Pressland \cite{CSFRT}.  Rather, here we will sketch the relevant constructions of the input data for an abstract cluster structure and make appropriate citations.

Note that, as with the previous example, we make no claim to be able obtain cluster categories from abstract cluster structures.  However, as is the case with geometric models, this is in some sense not the point.  Rather, we claim that having abstract cluster structures underpin all the different realisations of cluster combinatorics means that we can make claims on relationships between the realisations via the abstract cluster structures, even when there is not a direct connection.  For example, in type $A$, we know cluster categorifications and geometric models, but expressing the claim of these having the same cluster combinatorics (as each other, and as a cluster algebra of the same type) has hitherto been relatively informal.

Inevitably, we will require significant background knowledge of representation theory to state what follows.  We refer the reader to \cite{CSFRT} and references therein for undefined terminology.  For this reason, we sketch the main ideas and give the reader full licence to jump directly to Section~\ref{s:morphisms-of-reps} should they wish.

Let $\cC$ be a cluster category of finite rank\footnote{\cite{CSFRT} works in greater generality than this but the infinite rank case introduces a number of additional technicalities, so we will restrict this discussion to the finite rank case and leave to the interested reader the task of adjusting to the general setting.} and $\cT$ a cluster-tilting subcategory of $\cC$ (\cite[Section~2]{CSFRT}).  The graph $E$ is taken to have as vertex set the collection of all cluster-tilting subcategories of $\cC$ and is complete\footnote{It is one of the main results of \cite{CSFRT} that the maps to be defined shortly exist for any pair of cluster-tilting subcategories.}.

The free Abelian sheaves $\cA$ and $\cX$ are given on objects by taking a cluster-tilting subcategory $\cT$ to Grothendieck groups: we set $\cA \cT=\Kgp{\cT}$ and $\cX \cT=\Kgp{\fd \stab{\cT}}$.  Here, $\Kgp{\fd \stab{\cT}}$ is the Grothendieck group of finite-dimensional $\stab{\cT}$-modules, that is, functors $\stab{\cT} \to \fd \bK$ from the stable category of $\cT$, $\stab{\cT}$, to the category of finite-dimensional $\bK$-vector spaces.

Given a cluster-tilting subcategory, we have the notion of left and right $\cT$-approximations of objects $X\in \cC$: one can view these as being maps in $\cC$ with domain or codomain $X$ respectively satisfying a particular property with respect to a Hom-functor (\cite[Appendix~A]{CSFRT}). Alternatively, approximations can be encoded in conflations\footnote{The paper \cite{CSFRT} works at the level of generality of extriangulated categories.  The reader who is unfamiliar with these will not lose anything significant by reading ``exact sequence'' for ``conflation''.}
\begin{equation}
	\label{eq:index-seq}
	\begin{tikzcd}
		\rightker{\cT}{X}\arrow[infl]{r}&\rightapp{\cT}{X}\arrow[defl]{r}{}&{X}\arrow[confl]{r}&\phantom{}
	\end{tikzcd}
\end{equation}
and
\begin{equation}
	\label{eq:coindex-seq}
	\begin{tikzcd}
		X\arrow[infl]{r}&\leftapp{\cT}{X}\arrow[defl]{r}&\leftcok{\cT}{X}\arrow[confl]{r}&\phantom{}
	\end{tikzcd}
\end{equation}
such that \(\rightker{\cT}{X},\rightapp{\cT}{X},\leftapp{\cT}{X},\leftcok{\cT}{X}\in\cT\),
Moreover, the values of \([{\rightapp{\cT}{X}}]-[{\rightker{\cT}{X}}]\in \Kgp{\cT}\) and $[{\leftapp{\cT}{X}}]-[{\leftcok{\cT}{X}}]\in \Kgp{\cT}$ are independent of the choice of approximation conflation, so we obtain homomorphisms of Grothendieck groups (\cite[Section~3.1]{CSFRT})
\[
\ind{\cC}{\cT}  \colon \Kgp{\addcat{\cC}} \to \Kgp{\cT}, \ind{\cC}{\cT}(X)=[{\rightapp{\cT}{X}}]-[{\rightker{\cT}{X}}] \]
and
\[ 
\coind{\cC}{\cT}  \colon \Kgp{\addcat{\cC}} \to \Kgp{\cT}, \coind{\cC}{\cT}(X)=[{\leftapp{\cT}{X}}]-[{\leftcok{\cT}{X}}].
\]
Moreover, if the input is restricted to $\Kgp{\cU}\leq \Kgp{\addcat{\cC}}$ for a second cluster-tilting subcategory $\cU$, we in fact obtain isomorphisms (\cite[Section~3.4]{CSFRT})
\[
\ind{\cU}{\cT}  \colon \Kgp{\cU} \to \Kgp{\cT}, \quad \coind{\cU}{\cT}  \colon \Kgp{\cU} \to \Kgp{\cT}.
\]
For $\alpha^{\pm}\colon \cU \to \cT$, we then define $\cA \alpha^{+}=\ind{\cU}{\cT}$ and $\cA \alpha^{-}=\coind{\cU}{\cT}$.

The functor $\cX$ is a little more tricky, and in fact we need to discuss $\ip{\blank}{\blank}$ first.  The inner product here turns out to be extremely natural: for any additive category \(\cT\), viewed as a split exact category, there is a bilinear form \(\ip{\blank}{\blank}\colon\Kgp{\lfd{\cT}}\times\Kgp{\cT}\to\integ\) given by
\begin{equation}
	\label{eq:fundamental-pairing}
	\ip{[M]}{[T]}=\dim_{\bK} M(T)
\end{equation}
for objects \(M\in \lfd \cT\) and \(T\in \cT\) (and extended linearly to differences of classes of objects). For us, $\lfd \cT=\fd \cT$ and for our form $\ip{\blank}{\blank}_{\cT}$ we need to take the opposite of this form and restrict it to $\Kgp{\fd \stab{\cT}}\leq \Kgp{\fd \cT}$:
\begin{equation}
	\ip{[T]}{[M]}_{\cT}=\dim_{\bK} M(T)
\end{equation}
 since we need $\cA \cT=\Kgp{\cT}$ in the first argument and $\cX \cT=\Kgp{\fd \stab{\cT}}$ in the second.  Notwithstanding these alterations, the pairing $\ip{\blank}{\blank}_{\cT}$ is right non-degenerate (\cite[Section~3.3]{CSFRT}), because it amounts to pairing indecomposables with their corresponding simple modules.
 
 Then one can show that it is possible to define
 	\begin{align*}
 		\stabcoindbar{\cT}{\cU}  =\adj{(\stabind{\cU}{\cT})}&\colon \Kgp{\fd \stab{\cT}} \to \Kgp{\fd \stab{\cU}}, \\
 		\stabindbar{\cT}{\cU} =\adj{(\stabcoind{\cU}{\cT})}&\colon \Kgp{\fd \stab{\cT}} \to \Kgp{\fd \stab{\cU}},
 	\end{align*}
 	by restricting $\ind{}{}$ and $\coind{}{}$ appropriately and taking adjoints with respect to \(\canform{\blank}{\blank}{\cT}\) and \(\canform{\blank}{\blank}{\cU}\).
We then set $\cX \alpha^{+}=\stabcoindbar{\cT}{\cU}$ and $\cX \alpha^{-}=\stabindbar{\cT}{\cU}$.
 
 Here, adjunction means
 \begin{equation}
 	\canform{\stabcoindbar{\cT}{\cU}[M]}{[U]}{\cU}=\canform{[M]}{\ind{\cU}{\cT}[U]}{\cT},
 \end{equation}
 for all \(T\in\cT\) and \(M\in\fd{\stab{\cT}}\) and similarly for \(\stabindbar{}{}\) and \(\coind{}{}\). That these equations hold tell us that $\ip{\blank}{\blank}\colon \cA \tensor_{\integ} \cX \to \sZ$ is a dinatural transformation.

Obtaining $\beta$ can be done in one of two ways (\cite[Section~4]{CSFRT}).  The philosophically correct way is to show that there is a map $p_{\cT}\colon \Kgp{\fpmod \stab{\cT}}\to \Kgp{\cT}$, where $\fpmod \stab{\cT}$ is the category of finitely presented $\stab{\cT}$-modules.  This map $p_{\cT}$ is essentially given by taking projective resolutions: we compute a projective resolution of our module, take the class of this and then note that we may pass from $\Kgp{\proj \cT}$ to $\Kgp{\cT}$ via the Yoneda isomorphism.  However to do this we have to pass via an exact lift of $\cC$, which entails some technicalities.

More concretely, which is helpful both for subsequent proofs and for calculations, any \(M\in\fpmod{\stab{\cT}}\) is isomorphic to the module \(\Ext{1}{\cC}{\blank}{X}|_{\cT}\) for some \(X\in\cC\), and for any such \(X\) we have
\begin{equation}
	\label{eq:p-vs-ind-coind}
	p_{\cT}[M]=\ind{\cC}{\cT}[X]-\coind{\cC}{\cT}[X].
\end{equation}
In particular, \(p_{\cT}\) depends only on \(\cT\ctsubcat\cC\), and not on the choice of exact lift.  

We define $\beta_{\cT}=-p_{\cT}$, with the sign choice being due to existing cluster conventions.  Key for us is that the following diagrams commute (\cite[Section~4.3]{CSFRT}):
  \[\begin{tikzcd}
 \Kgp{\fd \stab{\cU}} \arrow{r}{\beta_{\cU}} & \Kgp{\cU} \arrow{d}{\ind{\cU}{\cT}=\cA\alpha^{+}} \\
 \Kgp{ \fd \stab{\cT}} \arrow{r}{\beta_{\cT}} \arrow{u}{\cX \alpha^{+}=\stabcoindbar{\cT}{\cU}} & \Kgp{\cT}
 \end{tikzcd}\qquad
 \begin{tikzcd}
 \Kgp{\fd \stab{\cU}} \arrow{r}{\beta_{\cU}} & \Kgp{\cU} \arrow{d}{\coind{\cU}{\cT}=\cA \alpha^{-}} \\
 \Kgp{ \fd \stab{\cT}} \arrow{r}{\beta_{\cT}} \arrow{u}{\cX \alpha^{-}=\stabindbar{\cT}{\cU}} & \Kgp{\cT}
 \end{tikzcd}\]
 showing that $\beta$ is a factorization with components $\beta_{\cT}$, as required.
 
Hence, from a cluster category of finite rank $\cC$, we obtain an abstract cluster structure $\cC(\cC)=\cC(\cE,\cX,\beta,\cA,\ip{\blank}{\blank})$ where
\begin{enumerate}
	\item $E$ is the complete bi-directed graph on the set of cluster-tilting subcategories of $\cC$;
	\item $\cX \cT=\Kgp{\fd \stab{\cT}}$ and $\cX\alpha^{+}=\stabcoindbar{\cU}{\cT}$, $\cX\alpha^{-}=\stabindbar{\cU}{\cT}$ for $\alpha^{\pm}\colon \cT \to \cU$;
	\item $\beta_{\cT}=-p_{\cT}$ is (essentially) given by taking projective resolutions (or via \eqref{eq:p-vs-ind-coind});
	\item $\cA \cT=\Kgp{\cT}$ and $\cA\alpha^{+}=\ind{\cT}{\cU}$, $\cA\alpha^{-}=\coind{\cT}{\cU}$ for $\alpha^{\pm}\colon \cT \to \cU$; and
	\item $\ip{\blank}{\blank}$ has $\cT$-component given by $\ip{[T]}{[M]}_{\cT}=\dim_{\bK} M(T)$.
\end{enumerate}

\sectionbreak
\section{Morphisms of representations}\label{s:morphisms-of-reps}
	
In this final section, we look to move beyond objects of $\ACScat$ and to morphisms; as in any category, this is where the interesting things really are.

There will be plenty left for future exploration, as we concentrate on two particular types of morphism.  The first is to show that our framework encompasses what is known already, in the form of (rooted) cluster morphisms between cluster algebras.  More precisely, we show in Section~\ref{ss:mor-from-RCM} that given two cluster algebras and a rooted cluster morphism between them, there is an induced morphism of the associated abstract cluster structures.

The second aspect we cover briefly, in Section~\ref{ss:classifying-mor}, is morphisms between cluster representations (to use the language of this Part) of different types.  This is an attempt to formalise what it means for a cluster category to decategorify to a cluster algebra, or for a surface to be a geometric model for a cluster algebra.  Specifically, we assert that a reasonable definition is that there is a morphism between the respective abstract cluster structures, perhaps with particular properties.  That is, while we cannot say that a cluster algebra and a geometric model for it are ``the same'', we can claim that they have the same underlying cluster combinatorics, as captured by their abstract cluster structures.  We will use the example of the Grassmannian cluster algebra $\curly{O}(\mathrm{Gr}(2,6))$ and its surface model by triangulations of a hexagon, since the latter was already introduced in Section~\ref{ss:triangulations}.
	
We will attempt to simplify matters somewhat by working only in the non-quantum case, i.e.\ with abstract cluster structures and commutative cluster algebras.

\subsection{Morphisms of exchange trees}\label{ss:mor-of-exch-tree}

In order to provide technical underpinning for the main result of the next subsection, we first give some constructions associated to exchange trees.  First, recall\footnote{We have made some minor adjustments to notation here in order to help disambiguation, given what will shortly follow, e.g.\ adding $\vec{\ }$ to the name of the edges.} that the directed exchange tree $E(\ex,\cB)$ associated to $\ex \subseteq \cB$ has vertices the admissible sequences $\uk\in \curly{K}(\ex,\cB)$ and arrows $\vec{\mu}_{k}\colon \uk\to (k,\uk)$ (where $(\underline{a},\underline{b})$ means the concatenation of $\underline{a}$ and $\underline{b}$). The exchange tree is rooted, with root the empty sequence $()$.

In more detail, we recall the following definition.  Given $\uk\in \curly{K}(\ex,\cB)$, set
\[ \uk_{\leq i}= (k_{i},k_{i-1},\dotsc ,k_{1})\]
(with the convention that $\uk_{\leq 0}=()$) and
\[ \mu_{\uk_{\leq i}}(\ex)=\mu_{k_{i}}(\dotsm (\mu_{k_{2}}(\mu_{k_{1}}(\ex)))) \]
	for $\mu_{j}\colon \mu_{\uk_{\leq j-1}}(\ex)\to \mu_{\uk_{\leq j}}(\ex)$ the bijections\footnote{Originally we defined the set $\mu_{k}(\ex)$, with $\mu_{k}(\blank)$ simply being notation, but later we upgraded this via the bijections referred to here, so that the set $\mu_{k}(\ex)$ is indeed the image of $\ex$ under $\mu_{k}$, so no ambiguity in fact arises.} obtained iteratively as described on page~\pageref{pg:mut-bij}.

\begin{definition*}[{Definition~\ref{d:admissible-mut-seq}; cf.\ \cite[Definition~1.3]{ADS}}] 
	Let $\cB$ be a finite set and $\ex \subseteq \cB$.
	
	Denote by $\curly{K}(\ex,\cB)$ the set of all tuples $\uk=(k_{r},\dotsc ,k_{1})$, such that for all $1\leq i \leq r$, $k_{i}\in \mu_{\uk_{\leq i-1}}(\ex)$, including the empty tuple $()$.
	
	We say that an element $\uk\in \curly{K}(\ex,\cB)$ is an $(\ex,\cB)$-admissible (mutation) sequence, of length $\card{\uk}=r$.
\end{definition*}

By construction, $\uk_{\leq i}\in \curly{K}(\ex,\cB)$ for all $1\leq i\leq \card{\uk}$.
	
The usual construction of the (undirected) exchange tree has a root initial seed, the remaining vertices being seeds\footnote{In practice, the exchange matrix is often suppressed, so that the vertices of the exchange graph are considered to be the clusters; in the presence of assumptions such that one knows that clusters determine their seeds (as originally conjectured in \cite{FZ-CA2}), this is natural.} obtained by one-step mutation and an edge for each such mutation.  The next construction is closer to this than the definition above, but rather than using cluster variables themselves, we just record their labels; the role this plays in relation to morphisms of cluster algebra and abstract cluster structures will become clearer later.

From the directed exchange tree $E(\ex,\cB)$, we obtain, as previously, a path category $\cE(\ex,\cB)$; by contrast to Section~\ref{ss:def-of-ACS}, we take the ordinary path category, as $E(\ex,\cB)$ is not signed.  Note that we could also consider the \emph{bi}directed exchange tree, with a pair of edges forming a 2-cycle for each one-step mutation and the construction below would go through with only the obvious changes required, but for simplicity we will use the directed exchange tree for now.

Let us define a functor $\mu_{\cB}\colon \cE(\ex,\cB)\to \Setcat$ using the definitions of Sections~\ref{ss:toric-frames} and \ref{ss:QCAs} similarly to the discussion in Section~\ref{ss:def-of-ACS}, we could think of this as a (pre)sheaf of sets on the graph $E(\ex,\cB)$. Specifically, for each vertex of the exchange tree, we put an indexing set for the associated cluster variables.

For a set $S$, let us write $S^{0}$ for $S\disjointunion \{0\}$; for all sets to which we apply this notation, we will have $0\notin S$.  Indeed, $0$ could easily be $*$ or $\bullet$ or some other symbol, provided none of our sets contain this.  We use $0$ because $\log 1=0$ and, as we will see, this additional element provides a means to record specialisation of variables to integers, but at the tropical level.

First, we define $\mu_{\cB}$ on objects $\uk\in \cE(\ex,\cB)$ (the objects of the path category being the vertices of the underlying graph) by 
\[ \mu_{\cB}(\uk)=\mu_{\uk}(\cB)^{0} \]
for $\mu_{\uk}(\cB)$ as defined in \eqref{eq:mu-k-B}.  Then it suffices to specify $\mu$ on paths of length one, i.e.\ arrows in $E(\ex,\cB)$.  For $\vec{\mu}_{k}\colon \uk\to (k,\uk)$, we define
\[ \mu_{\cB}(\vec{\mu}_{k}) = \mu_{k}^{0} \]
where $\mu_{k}^{0}$ is the bijection $\mu_{k}^{0}\colon \mu_{\uk}(\cB)^{0}\to \mu_{(k,\uk)}(\cB)^{0}$ such that $\mu_{k}^{0}|_{\mu_{\uk}(\cB)}=\mu_{k}$ (the bijection $\mu_{k}\colon \mu_{\uk}(\cB)\to \mu_{(k,\uk)}(\cB)$ defined immediately prior to Definition~\ref{d:admissible-mut-seq}) and $\mu_{k}^{0}(0)=0$.

\begin{remark}
	In principle, given $\mu_{k}^{0}$ (or just $\mu_{k}$) its domain, $\mu_{\uk}(\cB)^{0}$ (respectively, $\mu_{\uk}(\cB)$), tells us that we are mutating in direction $k$ from the $\uk$-iterated mutation of $\cB$.  That is, the notation hides the previous mutations, which may be unhelpful.  When we work more with admissible sequences below, we will want to be more careful about this, so to avoid confusion, we will sometimes enhance the notation to also record $\uk$: that is, we will write $\mu_{k,\uk}^{0}$ for the bijection $\mu_{k}^{0}\colon \mu_{\uk}(\cB)^{0}\to \mu_{(k,\uk)}(\cB)^{0}$ such that $\mu_{k}^{0}|_{\mu_{\uk}(\cB)}=\mu_{k,\uk}$, also enhancing the notation so that we have $\mu_{k,\uk}\colon \mu_{\uk}(\cB)\to \mu_{(k,\uk)}(\cB)$ for the bijection defined immediately prior to Definition~\ref{d:admissible-mut-seq}), again with $\mu_{k,\uk}^{0}(0)=0$.
	
	Then the functor $\mu_{\cB}$ assigns to the path of length one $\uk\stackrel{\vec{\mu}_{k}}{\to}(k,\uk)$ the bijection $\mu_{k,\uk}^{0}$.
	
	Here we see that if we had used the bidirected exchange graph construction, we would naturally assign $(\mu_{k,\uk}^{0})^{-1}$ to the reverse arrow $\uk\from (k,\uk)$.
\end{remark} 

\begin{proposition}
	The above assignment $\mu_{\cB}(\uk)=\mu_{\uk}(\cB)^{0}$ and $\mu_{\cB}(\vec{\mu}_{k})=\mu_{k,\uk}^{0}$ may be extended via composition of functions to define a functor $\mu_{\cB}\colon \cE(\ex,\cB)\to \Setcat$.
\end{proposition}

\begin{proof}
	Given $\uk\stackrel{\vec{\mu}_{k}}{\to}(k,\uk)\stackrel{\vec{\mu}_{l}}{\to}(l,k,\uk)$, we intend to define $\mu_{\cB}$ by assigning to this path of length two the composition $\mu_{l,(k,\uk)}^{0}\circ \mu_{k,\uk}^{0}$.  This is well-defined, and the same is true by induction for longer paths.  The directed exchange tree is acyclic, so there is no obstruction to extending $\mu_{\cB}$ to a functor.
\end{proof}

With our sincere apologies for subjecting the reader to yet more complex notation, we will also have need to work with subsequences and define certain maps iteratively.  To enable this, we extend the notation $\uk_{\leq i}$ and $\mu_{\uk_{\leq i}}(\ex)$ as follows.

For $\uk=(k_{r},\dotsc ,k_{1})\in \curly{K}(\ex,\cB)$ and $j\leq r$, define
$\mu_{\uk_{\leq j}}\colon \cB \to \mu_{\uk_{\leq j}}(\cB)$ by
\[ \mu_{\uk_{\leq j}} = \mu_{k_{j},\uk_{\leq j-1}}\circ \mu_{k_{j-1},\uk_{\leq j-2}}\circ \dotsm \circ \mu_{k_{2},\uk_{\leq 1}}\circ \mu_{k_{1},\uk_{\leq 0}}. \]
For $j=r=\card{\uk}$, we write $\mu_{\uk}\colon \cB \to \mu_{\uk}(\cB)$ and let $\mu_{\uk}^{0}\colon \cB^{0}\to \mu_{\uk}(\cB)^{0}$ be the extension of this with $\mu_{\uk}^{0}(0)=0$.  

We note the following:
\begin{itemize}
	\item Since $\uk$ is assumed to be admissible, all of these maps also restrict to maps \hfill \hfill \linebreak $\mu_{\uk_{\leq j}}\colon \ex \to \mu_{\uk_{\leq j}}(\ex)$ and $\mu_{\uk}\colon \ex \to \mu_{\uk}(\ex)$.  We will not overload the notation further by indicating the restriction.
	\item Since all the maps involved in the composition above are bijections, so are $\mu_{\uk_{\leq j}}$ and $\mu_{\uk}$.
	\item The image of $\mu_{\uk}\colon \cB \to \mu_{\uk}(\cB)$ is $\mu_{\uk}(\cB)$, since $\mu_{\uk}$ is in particular surjective. That is, the two uses of the notation $\mu_{\uk}(\cB)$ are compatible.
	\item By extending with sending $0\mapsto 0$ as above, we also have $\mu_{\uk_{\leq j}}^{0}$ and $\mu_{\uk}^{0}$ variants, with domains and codomains extended by disjoint union with $\{0\}$.
\end{itemize}

To reiterate, since there are (deliberately) many similar but different uses of $\mu$ above,
\begin{enumerate}
	\item the admissible sequences $\uk$ have admissible subsequences $\uk_{\leq i}$;
	\item the directed exchange tree has vertices the admissible sequences $\uk\in \curly{K}(\ex,\cB)$ and arrows $\uk \stackrel{\vec{\mu}_{k}}{\to} (k,\uk)$ corresponding to one-step mutation;
	\item\label{mu-k-B} there are sets $\mu_{\uk}(\cB)$ and $\mu_{\uk}(\ex)$ obtained by an iterated construction and bijections $\mu_{\uk}$ whose domains are $\cB$ and $\ex$ and whose images are these sets;
	\item the claims in \ref{mu-k-B} hold, \emph{mutatis mutandis}\footnote{We make no apology for our satisfaction in being able to legitimately use this phrase in the current context.}, 
	for admissible subsequences $\uk_{\leq i}$ of an admissible sequence $\uk$ as well as for the sets and maps with the additional ${}^{0}$ decoration indicating ``extension by $0$'';
	\item the functor $\mu_{\cB}\colon \cE(\ex,\cB)\to \Setcat$ has $\mu_{\cB}(\uk)=\mu_{\uk}(\cB)^{0}$ and $\mu_{\cB}(\vec{\mu}_{k})=\mu_{k,\uk}^{0}$ (with $\vec{\mu}_{k}$ a path of length one corresponding to an arrow in $E(\ex,\cB)$), so writing $\vec{\mu}_{\uk}$ for the canonical path of length $\card{\uk}$ from $()$ to $\uk$ determined by $\uk$, we have $\mu_{\cB}(\vec{\mu}_{\uk})=\mu_{\uk}^{0}$.
\end{enumerate}
We hope this summary serves to persuade the reader that the overloading of notation here is in fact fully consistent.  Nevertheless, we will still need to take care to type-check statements as we proceed.

Next, in preparation for examining morphisms of abstract cluster structures arising from morphisms of cluster algebras, we want to investigate the consequences of having a map between two indexing sets in relation to the functors introduced above.

\begin{definition}\label{d:phi-ex-adm} Let $\cB_{1}$ and $\cB_{2}$ be countable sets with distinguished subsets $\ex_{1}\subseteq \cB_{1}$, $\ex_{2}\subseteq \cB_{2}$.  As above, for a set $S$, write $S^{0}$ for $S\disjointunion \{0\}$ (implicitly assuming $0\notin S$).
	
	We say that a function $\phi\colon \cB_{1}^{0}\to \cB_{2}^{0}$ is $\ex$-admissible with respect to $\cB_{1}^{0}$ and $\cB_{2}^{0}$ (or simply, with respect to its domain and codomain) if $\phi(0)=0$ and $\phi(\ex_{1})\subseteq \ex_{2}^{0}$.
\end{definition}

Then if $\phi$ is $\ex$-admissible, there is a well-defined induced function $\ex_{1}^{0}\to \ex_{2}^{0}$ obtained by restricting the domain and codomain of $\phi$ appropriately.  We will also denote this restriction by $\phi$.

\begin{lemma}\label{l:mu-admissible} For $\uk\in \curly{K}(\ex,\cB)$, the bijections $\mu_{k,\uk}^{0}\colon \mu_{\uk}(\cB)^{0}\to \mu_{(k,\uk)}(\cB)^{0}$ (and respectively $\mu_{\uk}^{0}\colon \cB^{0} \to \mu_{\uk}(\cB)^{0}$) are $\ex$-admissible with respect to their domains and codomains where the distinguished subsets are $\mu_{\uk}(\ex)$ and $\mu_{(k,\uk)}(\ex)$ (respectively $\ex$ and $\mu_{\uk}(\ex)$).
\end{lemma}

\begin{proof}
	The claims follow immediately on examination of the definitions of $\mu_{k,\uk}$ and $\mu_{\uk}$: mutable elements are replaced by mutable elements.
\end{proof}

Given such an $\ex$-admissible function $\phi$, our immediate goal is to construct induced functions $\phi_{\uk}\colon \mu_{\cB_{1}}(\uk)\to \mu_{\cB_{2}}(F\uk)$, where (as above) $\mu_{\cB_{1}}(\uk)=\mu_{\uk}(\cB_{1})^{0}$ is the indexing set\footnote{Strictly, this set after disjoint union with $\{0\}$.} obtained by mutation along an $(\ex_{1},\cB_{1})$-admissible sequence $\uk$  and $\mu_{\cB_{2}}(F\uk)=\mu_{F\uk}(\cB_{2})^{0}$ similarly, with respect to an $(\ex_{2},\cB_{2})$-admissible mutation sequence $F\uk$ obtained from $\uk$ via $\phi$.

Consider the empty sequence $()$.  For the purposes of the iterative definitions to come, let us set $\phi_{()}=\phi$ and $F()=()$.  If $\uk$ is an admissible sequence, since $\uk_{\leq 0}=()$ by convention, we also have $\phi_{\uk_{\leq 0}}=\phi_{()}=\phi$ and $F\uk_{\leq 0}=F()=()$ for any admissible sequence $\uk$.

Let us next consider the $(\ex_{1},\cB_{1})$-admissible sequence $(k)$, where $k\in \ex_{1}$; that is, we make a single one-step mutation from our initial seed.

Since $\phi$ is $\ex$-admissible, either $\phi(k)\in \ex_{2}$ or $\phi(k)=0$.  We define $F(k)$ by
\[ F(k) = \begin{cases} (\phi(k)) & \text{if}\ \phi(k)\in \ex_{2} \\ () & \text{if}\ \phi(k)=0 \end{cases}. \]
Note in particular that, since $\phi$ is $\ex$-admissible, $F(k)$ is $(\ex_{2},\cB_{2})$-admissible.

Since this definition is key to understanding what follows, we give a little extra commentary.  We have a valid mutation direction, $k\in \ex_{1}$, and under $\phi$, we may have that its image $\phi(k)$ is also a valid mutation direction, i.e.\ $\phi(k)\in \ex_{2}$.  In this case, it is natural to ``pair up'' $(k)$ with $F(k)=(\phi(k))$, and then both $(k)$ and $F(k)$ are admissible sequences with respect to their natural domains.

The case $\phi(k)=0$ is intended to model specialisation of the corresponding variable to an integer, we shall see when we discuss rooted cluster morphisms below.  That is, at this point, the associated variable stops being mutable and no further mutations can be made in this direction.  We then set as the image of $(k)$ the empty sequence $()$.  It is natural to think of this in terms of the image of the map $F$ from $E(\ex_{1},\cB_{1})$ to $E(\ex_{2},\cB_{2})$ induced by $\phi$ as being the contraction of edges in $E(\ex_{1},\cB_{1})$ in directions $k$ where $\phi(k)\notin \ex_{2}$.
	
Next, we define an induced map $\phi_{(k)}\colon \mu_{\cB_{1}}((k))\to \mu_{\cB_{2}}(F(k))$, recalling that $\mu_{\cB_{1}}((k))=\mu_{k}(\cB_{1})^{0}$ and $\mu_{\cB_{2}}(F(k))=\mu_{F(k)}(\cB_{2})^{0}$.  The definition of $\phi_{(k)}$ will depend on whether $\phi(k)$ is exchangeable or zero (these being the two options, since $k\in \ex_{1}$ and since $\phi$ is $\ex$-admissible).

We set
\[ \phi_{(k)} = \begin{cases} \mu_{\phi_{()}(k),F()}^{0} \circ \phi_{()} \circ (\mu_{k,()}^{0})^{-1} & \text{if}\ \phi_{()}(k)\neq 0 \\ \phi_{()} \circ (\mu_{k,()}^{0})^{-1} & \text{if}\ \phi_{()}(k)= 0 \end{cases} .\]

\begin{remark} We have written empty sequences where they are not necessary (e.g.\ $\phi_{()}=\phi$, $F()=()$) to help with comparison with the definition to come, making it clear that this is the base case for the construction.
	
	We can also avoid the case split if we extend the notation $\mu_{k,\uk}^{0}$ to include the possibility that $k=0$.  That is, if we put $\mu_{0,\uk}^{0}=\id_{\mu_{\uk}(\cB)}$, we can simply write
\begin{equation}\label{e:phi-k} \phi_{(k)} = \mu_{\phi_{()}(k),F()}^{0} \circ \phi_{()} \circ (\mu_{k,()}^{0})^{-1} \end{equation}
and observe that $\mu_{\phi_{()}(k),F()}^{0}=\id_{\mu_{F()}(\cB_{2})^{0}}$ when $\phi_{()}(k)=0$.
\end{remark}

Let us compute this more explicitly, in this initial case where this is feasible. First, note that
$\mu_{k,()}^{0}\colon \cB_{1}^{0}\to \mu_{k}(\cB_{1})^{0}$ is given by
\[ \mu_{k,()}^{0}(b)=\begin{cases} b & \text{if}\ b\neq k,0 \\ \mu_{\cB_{1}}(k) & \text{if}\ b=k \\ 0 & \text{if}\ b=0 \end{cases} \]
for $b\in \cB_{1}^{0}$, so that
\[ (\mu_{k,()}^{0})^{-1}(\mu_{k}(b))=\begin{cases} b & \text{if}\ \mu_{k}(b)\neq \mu_{k}(k)\ (=\mu_{\cB_{1}}(k)),0 \\ k & \text{if}\ \mu_{k}(b)=\mu_{k}(k)\ (=\mu_{\cB_{1}}(k)) \\ 0 & \text{if}\ \mu_{k}(b)=0 \end{cases} \]
for $\mu_{k}(b)\in \mu_{k}(\cB_{1})^{0}$.

Similarly,
\[ \mu_{\phi_{()}(k),F()}^{0}(\phi_{()}(b))=\begin{cases} \phi_{()}(b) & \text{if}\ \phi_{()}(b)\neq \phi_{()}(k), 0 \\ \mu_{\cB_{2}}(\phi_{()}(k)) & \text{if}\ \phi_{()}(b)=\phi_{()}(k) \\ 0 & \text{if}\ \phi_{()}(b)=0 \end{cases} \]
for $\phi_{()}(b)\in \Image{\phi_{()}}\subseteq \cB_{2}^{0}$.

Then if $\phi_{()}(k)\neq 0$, 
\begin{align*}
 \phi_{(k)}(\mu_{k}(b)) & = (\mu_{\phi_{()}(k),F()}^{0} \circ \phi_{()} \circ (\mu_{k,()}^{0})^{-1})(\mu_{k}(b)) \\	
 & = \begin{cases}
 	(\mu_{\phi_{()}(k),F()}^{0} \circ \phi_{()})(b) & \text{if}\ \mu_{k}(b)\neq \mu_{k}(k),0 \\
 	(\mu_{\phi_{()}(k),F()}^{0} \circ \phi_{()})(k) & \text{if}\ \mu_{k}(b)=\mu_{k}(k) \\
 	0 & \text{if}\ \mu_{k}(b)=0
 \end{cases} \\
 & = \begin{cases}
 	\phi_{()}(b) & \text{if}\ \mu_{k}(b)\neq \mu_{k}(k),0\ \text{and}\ \phi_{()}(b)\neq \phi_{()}(k) \\
 	\mu_{\cB_{2}}(\phi_{()}(k)) & \text{if}\ \mu_{k}(b)\neq \mu_{k}(k),0\ \text{and}\ \phi_{()}(b)=\phi_{()}(k) \\
 	\mu_{\cB_{2}}(\phi_{()}(k)) & \text{if}\ \mu_{k}(b)=\mu_{k}(k) \\
  	0 & \text{if}\ \mu_{k}(b)=0
 \end{cases} \\
 & = \begin{cases}
 	\phi_{()}(b) & \text{if}\ \mu_{k}(b)\neq \mu_{k}(k),0\ \text{and}\ \phi_{()}(b)\neq \phi_{()}(k) \\
 	\mu_{\cB_{2}}(\phi_{()}(k)) & \text{if \emph{either}}\ \mu_{k}(b)=\mu_{k}(k)\ \text{\emph{or}}\ (\mu_{k}(b)\neq \mu_{k}(k)\ \text{and}\ \phi_{()}(b)=\phi_{()}(k)) \\
 	0 & \text{if}\ \mu_{k}(b)=0
 \end{cases}. \\
\end{align*}
Correspondingly, if $\phi_{()}(k)=0$,
\begin{align*}
	\phi_{(k)}(\mu_{k}(b)) & = (\phi_{()} \circ (\mu_{k,()}^{0})^{-1})(\mu_{k}(b)) \\	
	& = \begin{cases}
		\phi_{()}(b) & \text{if}\ \mu_{k}(b)\neq \mu_{k}(k),0 \\
		\phi_{()}(k)=0 & \text{if}\ \mu_{k}(b)=\mu_{k}(k) \\
		0 & \text{if}\ \mu_{k}(b)=0
	\end{cases} \\
	& = \begin{cases}
		\phi_{()}(b) & \text{if}\ \mu_{k}(b)\neq \mu_{k}(k),0 \\
		0 & \text{if \emph{either}}\ \mu_{k}(b)=\mu_{k}(k)\ \text{\emph{or}}\ \mu_{k}(b)=0
	\end{cases}.
\end{align*}

Note in particular that $\phi_{(k)}(\mu_{k}(\ex_{1}))\subseteq \mu_{F(k)}(\ex_{2})^{0}$, so that $\phi_{(k)}$ is $\ex$-admissible with respect to its domain and codomain.  This follows from Lemma~\ref{l:mu-admissible} and $\ex$-admissibility of $\phi_{\leq 0}=\phi_{()}=\phi$.

For $\uk=(k_{r},\dotsc ,k_{1})$ admissible, as we saw with respect to $\uk_{\leq 0}$,  we have $F\uk_{\leq 1}=F(k_{1})$ and $\phi_{\uk_{\leq 1}}=\phi_{(k_{1})}$.

With these established as the base cases, we can now make general definitions of $F\uk$ and $\phi_{\uk}$ for $\uk$ admissible, as follows.

First, recall that for $\uk=(k_{r},\dotsc ,k_{1})\in \curly{K}(\ex,\cB)$ we set
\[ \mu_{\uk} = \mu_{k_{r},\uk_{\leq r-1}}\circ \mu_{k_{r-1},\uk_{\leq r-2}}\circ \dotsm \circ \mu_{k_{2},\uk_{\leq 1}}\circ \mu_{k_{1},\uk_{\leq 0}} \]
and $\mu_{\uk}^{0}\colon \cB^{0}\to \mu_{\uk}(\cB)^{0}$ is the extension by zero of this (and is equal to the compositions of the extension by zero of the constituent maps).

\begin{definition}\label{d:Fk-phi-k}
	Let $\cB_{1},\cB_{2}$ be countable sets with distinguished subsets $\ex_{i}\subseteq \cB_{i}$ ($i=1,2$) and let $\phi\colon \cB_{1}^{0}\to \cB_{2}^{0}$ be $\ex$-admissible.  Let $\uk\in \curly{K}(\ex_{1},\cB_{1})$ be $(\ex_{1},\cB_{1})$-admissible and set $r=\card{\uk}$.
	
	Define 
	\[ F\uk =\begin{cases}
		(\phi_{\uk_{\leq r-1}}(k_{r}),F\uk_{\leq r-1}) & \text{if}\ \phi_{\uk_{\leq r-1}}(k_{r})\in \mu_{F\uk_{\leq r-1}}(\ex_{2}) \\
		F\uk_{\leq r-1} & \text{if}\ \phi_{\uk_{\leq r-1}}(k_{r})=0
	\end{cases}\]
	and define $\phi_{\uk}\colon \mu_{\cB_{1}}(\uk)\to \mu_{\cB_{2}}(F\uk)$ by
	\[ \phi_{\uk}=\mu_{F\uk}^{0}\circ \phi_{()} \circ (\mu_{\uk}^{0})^{-1} .\]
\end{definition}

Note that, as $F$ and $\phi_{\bullet}$ appear in each other's definitions, we need to define these simultaneously.  But since the definition of $F\uk$ only requires $\phi_{\uk_{\leq r-1}}$, we can validly make the definitions iteratively, by induction on the length of $\underline{k}$.

\begin{lemma}\label{l:Fk-phi-admissible}
	With notation as in Definition~\ref{d:Fk-phi-k}, we have that
	\begin{enumerate}
		\item $F\uk$ is $(\ex_{2},\cB_{2})$-admissible and
		\item $\phi_{\uk}$ is $\ex$-admissible with respect to its domain and codomain.
	\end{enumerate}
\end{lemma}

\begin{proof}
	The verification of these claims follows by a linked induction along the lines indicated above for the case $\uk=(k)$: the $(\ex_{2},\cB_{2})$-admissibility of $F\uk$ follows from the $\ex$-admissibility of $\phi_{\uk_{\leq r-1}}$ and the $\ex$-admissibility of $\phi_{\uk}$ follows from the admissibility of $F\uk$, the $\ex$-admissibility of $\mu_{\uk}^{0}$ and $\mu_{F\uk}^{0}$ (via Lemma~\ref{l:mu-admissible}) and $\phi$.
\end{proof}

The following is immediate from examining the (many) above definitions, in particular that of $\mu_{\uk}$:

\begin{lemma}\label{l:phi-k-uk}
	Consider the admissible sequence $(k,\uk)\in \curly{K}(\ex_{1},\cB_{1})$.  Then
	\[ F(k,\uk) = \begin{cases} (\phi_{\uk}(k),F\uk) & \text{if}\ \phi_{\uk}(k)\neq 0 \\ F\uk & \text{if}\ \phi_{\uk}(k)=0 \end{cases} \]
	and 
	\[ \phi_{(k,\uk)}=\mu_{\phi_{\uk}(k),F\uk}^{0} \circ \phi_{\uk} \circ (\mu_{k,\uk}^{0})^{-1} .\] \qed
\end{lemma}

Note that in the latter expression we again use the convention that $\mu_{\phi_{\uk}(k),F\uk}^{0}=\id_{\mu_{F\uk}(\cB_{2})^{0}}$ if $\phi_{\uk}(k)=0$, to avoid case splits.  Indeed, if we went one step further and adopted the convention that $(0,\underline{l})=\underline{l}$ for any sequence $\underline{l}$, we could also write $F(k,\underline{k})=(\phi_{\uk}(k),F\uk)$ without cases.

Consequently, from a pair of indexing sets $\cB_{1},\cB_{2}$ with $\ex_{i}\subseteq \cB_{i}$ and an $\ex$-admissible map $\phi\colon \cB_{1}^{0}\to \cB_{2}^{0}$, we have a function $F\colon \curly{K}(\ex_{1},\cB_{1})\to \curly{K}(\ex_{2},\cB_{2})$ sending $(\ex_{1},\cB_{1})$-admissible sequences to $(\ex_{2},\cB_{2})$-admissible sequences and also functions $\phi_{\uk}$ from $\mu_{\cB_{1}}(\uk)$ to $\mu_{\cB_{2}}(F\uk)$.  

Let us examine $F$ more closely.

\begin{lemma}\label{l:Fk-is-a-functor}
	Let $\cB_{1},\cB_{2}$ be countable sets with distinguished subsets $\ex_{i}\subseteq \cB_{i}$ ($i=1,2$) and let $\phi\colon \cB_{1}^{0}\to \cB_{2}^{0}$ be $\ex$-admissible.
	
	There is a functor $\bF\colon \cE(\ex_{1},\cB_{1})\to \cE(\ex_{2},\cB_{2})$ defined on objects $\uk\in \curly{K}(\ex_{1},\cB_{1})$ by $\bF\uk=F\uk$ and on morphisms $\vec{\mu}_{k}\colon \uk\to (k,\uk)$ by
	\[ \bF\vec{\mu}_{k}=\begin{cases}
		\vec{\mu}_{\phi_{\uk}(k)} & \text{if}\ \phi_{\uk}(k)\in \mu_{F\uk}(\ex_{2}) \\
		\id_{F\uk} & \text{if}\ \phi_{\uk}(k)=0
	\end{cases}. \]
\end{lemma}

\begin{proof}
	We have an assignment of objects: $\bF\uk=F\uk$ as defined in Definition~\ref{d:Fk-phi-k} is an object of $\cE(\ex_{2},\cB_{2})$ by Lemma~\ref{l:Fk-phi-admissible}.
	
	For functoriality, it suffices to check that we can extend the given specification of $\bF$ on paths of length one consistently under composition.  So, let $\vec{\mu}_{k}\colon \uk\to (k,\uk)$ and $\vec{\mu}_{l}\colon (k,\uk)\to (l,k,\uk)$ be composable morphisms in $\cE(\ex_{1},\cB_{1})$ corresponding to paths of length one (i.e.\ derived from arrows in $E(\ex_{1},\cB_{1})$).  
	
	There are four cases to consider, depending on the images of $k$ and $l$ under $\phi_{\uk}$ and $\phi_{(k,\uk)}$ respectively.
	\begin{enumerate}
		\item If $\phi_{\uk}(k)\in \mu_{F\uk}(\ex_{2})$ and $\phi_{(k,\uk)}(l)\in \mu_{F(k,\uk)}(\ex_{2})$ then
		\begin{align*}
			\bF(k,\uk) & = (\phi_{\uk}(k),F\uk), \\
			\bF(l,k,\uk) & = (\phi_{(k,\uk)}(l),\phi_{\uk}(k),F\uk), \\
			\bF\vec{\mu}_{k} & = \vec{\mu}_{\phi_{\uk}(k)}\quad \text{and} \\
			\bF\vec{\mu}_{l} & = \vec{\mu}_{\phi_{(k,\uk)}(l)}
		\end{align*}
		so that we may define $\bF(\vec{\mu}_{l}\circ \vec{\mu}_{k})= \vec{\mu}_{(\phi_{(k,\uk)}(l),\phi_{\uk}(k))}=\vec{\mu}_{\phi_{(k,\uk)}(l)}\circ\vec{\mu}_{\phi_{\uk}(k)}=\bF\vec{\mu}_{l}\circ \bF\vec{\mu}_{k}$.
		\item If $\phi_{\uk}(k)\in \mu_{F\uk}(\ex_{2})$ and $\phi_{(k,\uk)}(l)=0$ then
		\begin{align*}
			\bF(k,\uk) & = (\phi_{\uk}(k),F\uk), \\
			\bF(l,k,\uk) & = F(k,\uk)=(\phi_{\uk}(k),F\uk), \\
			\bF\vec{\mu}_{k} & = \vec{\mu}_{\phi_{\uk}(k)}\quad \text{and} \\
			\bF\vec{\mu}_{l} & = \id_{F(k,\uk)}
		\end{align*}
		so that we may define $\bF(\vec{\mu}_{l}\circ \vec{\mu}_{k})= \vec{\mu}_{\phi_{\uk}(k)}=\id_{F(k,\uk)}\circ\vec{\mu}_{\phi_{\uk}(k)}=\bF\vec{\mu}_{l}\circ \bF\vec{\mu}_{k}$.
		\item If $\phi_{F\uk}(k)=0$ and $\phi_{(k,\uk)}(l)\in \mu_{F(k,\uk)}(\ex_{2})$ then
		\begin{align*}
			\bF(k,\uk) & = F\uk, \\
			\bF(l,k,\uk) & = (\phi_{(k,\uk)}(l),(),F\uk)=(\phi_{\uk}(l),F\uk), \\
			\bF\vec{\mu}_{k} & = \id_{F\uk}\quad \text{and} \\
			\bF\vec{\mu}_{l} & = \vec{\mu}_{\phi_{\uk}(l)}
		\end{align*}
		so that we may define $\bF(\vec{\mu}_{l}\circ \vec{\mu}_{k})= \vec{\mu}_{\phi_{\uk}(l)}=\vec{\mu}_{\phi_{\uk}(l)}\circ\id_{F\uk}=\bF\vec{\mu}_{l}\circ \bF\vec{\mu}_{k}$.
		\item If $\phi_{\uk}(k)=\phi_{(k,\uk)}(l)=0$, then  
		\begin{align*} \bF(k,\uk) & =F(k,\uk)=F\uk, \\
		\bF(l,k,\uk) & = F(k,\uk)=F\uk, \\
		\bF\vec{\mu}_{k} & =\id_{F\uk}\quad \text{and} \\
		\bF\vec{\mu}_{l} & =\id_{F\uk}
		\end{align*}
		so that we may define $\bF(\vec{\mu}_{l}\circ \vec{\mu}_{k})=\id_{F\uk}=\id_{F\uk}\circ \id_{F\uk}=\bF\vec{\mu}_{l}\circ \bF\vec{\mu}_{k}$.
	\end{enumerate}
	Since $E(\ex_{1},\cB_{1})$ is a directed tree, hence acyclic, it follows that the specification of $\bF$ on paths of length one can be extended to obtain a functor as claimed.
\end{proof}

We now show that the $\phi_{\uk}$ in fact encode the data of a morphism of functors\footnote{Depending on our preferred perspective, we could also call this a morphism of representations (of categories, valued in $\Setcat$) or a morphism of sheaves of sets.} from $\mu_{\cB_{1}}$ to $\mu_{\cB_{2}}\circ \bF$. Bear in mind that $\bF$ also depends on the choice of $\phi$.

\begin{proposition}\label{p:phi-is-nat-transf}
	Let $\cB_{1},\cB_{2}$ be countable sets with distinguished subsets $\ex_{i}\subseteq \cB_{i}$ ($i=1,2$) and let $\phi\colon \cB_{1}^{0}\to \cB_{2}^{0}$ be $\ex$-admissible.
	
	Then $\phi$ induces a natural transformation $\phi\colon \mu_{\cB_{1}}\to \mu_{\cB_{2}}\circ \bF$.
\end{proposition}

\begin{proof}
	Let us unpack the notation somewhat first. Let $i=1,2$. Recall that $\mu_{\cB_{i}}$ is a functor from $\cE(\ex_{i},\cB_{i})$ to $\Setcat$, where $\cE(\ex_{i},\cB_{i})$ is the (ordinary) path category associated to the directed exchange tree $E(\ex_{i},\cB_{i})$, so that the objects of $\cE(\ex_{i},\cB_{i})$ are the vertices of $E(\ex_{i},\cB_{i})$, namely the $(\ex_{i},\cB_{i})$-admissible sequences $\curly{K}(\ex_{i},\cB_{i})$.  The arrows in $E(\ex_{i},\cB_{i})$ are $\vec{\mu}_{k}\colon \uk\to (k,\uk)$ for $k\in \mu_{\uk_{\leq \card{\uk}-1}}(\ex_{i})$.
	
	The functor $\mu_{\cB_{i}}$ sends an object $\uk\in \curly{K}(\ex_{i},\cB_{i})$ to the set $\mu_{\cB_{i}}(\uk)\defeq \mu_{\uk}(\cB_{i})^{0}$ and is defined on morphisms by $\mu_{\cB_{i}}(\vec{\mu}_{\uk})=\mu_{\uk}^{0}\colon \cB_{i}^{0}\to \mu_{\uk}(\cB_{i})^{0}$.  
	
	Then the composition $\mu_{\cB_{2}}\circ \bF$ sends $\uk$ to $\mu_{\cB_{2}}(F\uk)=\mu_{F\uk}(\cB_{2})^{0}$ and similarly $\vec{\mu}_{\uk}$ is sent to $\mu_{F\uk}^{0}\colon \cB_{2}^{0} \to \mu_{F\uk}(\cB_{2})^{0}$.
	
	It will suffice to consider paths of length one, i.e.\ of the form $\mu_{\cB_{1}}(\vec{\mu}_{k})=\mu_{k,\uk}^{0}\colon \mu_{\uk}(\cB_{1})^{0}\to \mu_{(k,\uk)}(\cB_{1})^{0}$, where $\uk$ is an $(\ex_{1},\cB_{1})$-admissible sequence of length $\card{\uk}=r$.  Specifically, we see that to prove the claim, we need to establish that the following diagram commutes:
		\[\begin{tikzcd}[column sep=25pt]
			\mu_{\cB_{1}}(\uk) \arrow{r}[above]{\phi_{\uk}} \arrow{d}[left]{\mu_{\cB_{1}}(\vec{\mu}_{k})} & \mu_{\cB_{2}}(F\uk) \arrow{d}[right]{(\mu_{\cB_{2}}\circ \bF)(\vec{\mu}_{k})} \\ \mu_{\cB_{1}}((k,\uk)) \arrow{r}[above]{\phi_{(k,\uk)}} & \mu_{\cB_{2}}(F(k,\uk))
		\end{tikzcd} \]
	That is, inserting the aforementioned definitions of the functors, we require
			\[\begin{tikzcd}[column sep=25pt,row sep=25pt]
		\mu_{\uk}(\cB_{1})^{0} \arrow{r}[above]{\phi_{\uk}} \arrow{d}[left]{\mu_{k,\uk}^{0}} & \mu_{F\uk}(\cB_{2})^{0} \arrow{d}[right]{\mu_{F(k,\uk)}^{0}} \\ \mu_{(k,\uk)}(\cB_{1})^{0} \arrow{r}[above]{\phi_{(k,\uk)}} & \mu_{F(k,\uk)}(\cB_{2})^{0}
	\end{tikzcd} \]
	to commute.
	
	But this is the content of Lemma~\ref{l:phi-k-uk}, so we are done.
\end{proof}

\subsection{Induced morphisms}\label{ss:mor-from-RCM}

For $i=1,2$, let $\cC_{i}=\cC(\ex_{i},\cB_{i},\invfrozen_{i},\beta_{i})$ be cluster algebras, as in Section~\ref{ss:QCAs}, and let $\cC(\cC_{i})=\cC(\cE_{i},\cX_{i},\beta_{i},\cA_{i},\ip{\blank}{\blank}_{i})$ be the abstract cluster structures associated to them, as per Theorem~\ref{t:AQCS-from-QCA}.  Note that the graphs $E_{i}=E(\cC_{i})$ underlying $\cE_{i}=\cE(E_{i})$ are rooted, in the sense of Section~\ref{ss:connectedness}, since they are chosen to be the directed exchange trees associated to $\cC_{i}$. The root vertices are those labelled by $()$, the empty sequence in $\curly{K}=\curly{K}(\ex_{i},\cB_{i})$.  
	
Let $\sigma_{i}$ be the corresponding seed in $\cC_{i}$ (that is, $\sigma_{i}=(\ex_{i},\cB_{i},\invfrozen_{i},\beta_{i})$, the data determining $\cC_{i}$); the pair $(\cC_{i},\sigma_{i})$ is also called a \emph{rooted cluster algebra} in the literature\footnote{In effect, by construction and notation, all of our (quantum) cluster algebras are rooted, and as per the discussion of this paragraph, so are their abstract cluster structures.}

Recall that there is a canonical toric frame associated to an initial seed $\sigma=(\ex,\cB,\invfrozen,\beta)$, namely $M\colon \dual{\integ[\cB]}\to \curly{F}(\qtorus{}{}{\cB})$, given by $M(\dual{b})=x^{\dual{b}}$.  By mutation, we obtain the cluster variables, as $\mu_{\uk}M(\dual{b})$ for $\dual{b}\in \dual{\mu_{\uk}(\cB)}$, $\uk\in \curly{K}$.  As a shorthand, let us write $M(\dual{\cB})$ for the set $\{ M(\dual{b}) \mid \dual{b}\in \dual{\cB} \}$; since $\ex \subseteq \cB$, we also have $M(\dual{\ex})\defeq \{ M(\dual{b}) \mid \dual{b}\in \dual{\ex} \}\subseteq M(\dual{\cB})$.  Then $M(\dual{\cB})$ is the set of initial cluster variables and $M(\dual{\ex})$ the subset of these that are mutable.

We recall and make the following definitions, adapting the original corresponding ones in \cite{ADS} to our notation and terminology.

Recall the definition of an $(\ex,\cB)$-admissible sequence (Definition~\ref{d:admissible-mut-seq}) $\underline{k}\in \curly{K}(\ex,\cB)$ as a tuples such that $k_{i}\in \mu_{\uk_{\leq i-1}}(\ex)$ for all $i$.  Recall too that for $\ex_{i}\subseteq \cB_{i}$ ($i=1,2$) and $S\defeq S\disjointunion\{0\}$, a function $\phi\colon \cB_{1}^{0}\to \cB_{2}^{0}$ is said to be $\ex$-admissible if $\phi(0)=0$ and $\phi(\ex_{1})\subseteq \ex_{2}^{0}$ (Definition~\ref{d:phi-ex-adm}).

Then in Definition~\ref{d:Fk-phi-k} and Lemma~\ref{l:Fk-phi-admissible}, we obtained from an $\ex$-admissible function $\phi$ a function $F\colon \curly{K}(\ex_{1},\cB_{2})\to \curly{K}(\ex_{2},\cB_{2})$, sending $(\ex_{1},\cB_{1})$-admissible sequences to $(\ex_{2},\cB_{2})$-admissible sequences.  This definition and lemma are a substitute for the definition of biadmissible sequence given as \cite[Definition~2.1]{ADS}.  That definition refers to sequences of elements of cluster algebras, whereas our version moves this to the corresponding indexing sets.  Moreover, by asking for the $\ex$-admissibility property in $\phi$, every $(\ex_{1},\cB_{1})$-admissible sequence $\uk$ has a counterpart $F\uk$ that is $(\ex_{2},\cB_{2})$-admissible: the latter is no longer something to be checked for each $\uk$.

Next we adjust the key definition of \cite{ADS}, that of \emph{rooted cluster morphism}, to take account of the above.

For $i=1,2$, let $(\cC_{i},\sigma_{i})$ be rooted cluster algebras with seeds $\sigma_i = (\ex_{i},\cB_{i},\invfrozen_{i},\beta_{i})$. Let us assume we have an algebra homomorphism $f \colon \cC_{1}\to \cC_{2}$ such that the following are satisfied:
\begin{enumerate}
	\item $f(M_{1}(\dual{\cB_1})) \subseteq M_{2}(\dual{\cB_2}) \union \integ$ (that is, $f$ maps initial cluster variables of $(\cC_{1},\sigma_{1})$ to initial cluster variables of $(\cC_{2},\sigma_{2})$ or integers);
	\item $f(M_{1}(\dual{\ex_1})) \subseteq M_{2}(\dual{\ex_2}) \union \integ$ (that is, $f$ maps mutable initial cluster variables of $(\cC_{1},\sigma_{1})$ to mutable cluster variables of $(\cC_{2},\sigma_{2})$ or integers).
\end{enumerate}
We say such a homomorphism is \emph{compatible with the initial seeds} $\sigma_{1}$ and $\sigma_{2}$.

\begin{lemma}\label{l:RCM-induces-phi} Let $f\colon \cC_{1}\to \cC_{2}$ be an algebra homomorphism of rooted cluster algebras that is compatible with the initial seeds.  Then $f$ induces an $\ex$-admissible function $\phi\colon \cB_{1}^{0}\to \cB_{2}^{0}$.
\end{lemma}

\begin{proof}
	Consider $\dual{b}\in \dual{\cB}_{1}$.  If $f(M_{1}(\dual{b}))\in M_{2}(\dual{\cB}_{2})$, there exists $\dual{c}\in \dual{\cB}_{2}$ such that $f(M_{1}(\dual{b}))=M_{2}(\dual{c})$.  In this situation, define $\phi(b)=c$.  Otherwise, $f(M_{1}(\dual{b}))\in \integ$; in this case, define $\phi(b)=0$.  Finally, let $\phi(0)=0$.  Hence, from $f$, we obtain a function $\phi\colon \cB_{1}^{0}\to \cB_{2}^{0}$.
	
	Then $\phi$ is $\ex$-admissible: $\phi(0)=0$ by construction and if $\dual{b}\in \dual{\ex}_{1}$, we have that $\phi(b)\in \ex_{2}^{0}$ since $f$ maps mutable initial variables to mutable variables or integers.
\end{proof}

Being compatible with the initial seed data is not, in general, a sufficiently strong condition, however.  It is natural to ask that $f$ commutes with mutation and indeed this is the extra condition imposed in \cite[Definition~2.2]{ADS}.

\begin{definition}[{cf.\ \cite[Definition~2.2]{ADS}}]
For $i=1,2$, let $(\cC_{i},\sigma_{i})$ be rooted cluster algebras with seeds $\sigma_i = (\ex_{i},\cB_{i},\invfrozen_{i},\beta_{i})$. A \emph{rooted cluster morphism} is an algebra homomorphism $f \colon \cC_{1}\to \cC_{2}$
	such that the following are satisfied:
	\begin{enumerate}
		\item $f$ is compatible with the initial seeds, with associated $\ex$-admissible function $\phi$ and 
		\item\label{d:RCM-mut-compat} 
		for all $\uk\in \curly{K}(\ex_{1},\cB_{1})$ and $\mu_{\uk}(b)\in \mu_{\uk}(\cB_{1})$ such that $\phi_{\uk}(\mu_{\uk}(b))\neq 0$, we have
		\[ f(\mu_{\uk}M_{1}(\dual{\mu_{\uk}(b)}))=\mu_{F\uk}M_{2}(\dual{\phi_{\uk}(\mu_{\uk}(b))}) \]
	\end{enumerate}
	for $F\colon \curly{K}(\ex_{1},\cB_{1})\to \curly{K}(\ex_{2},\cB_{2})$ derived from $\phi$.
	
	We say that a rooted cluster morphism as above is \emph{without specialisations} if we have \hfill \hfill \linebreak $f(M_{1}(\dual{\cB}_1)) \subseteq M_{2}(\dual{\cB}_2)$, or equivalently $\phi(\cB_{1})\subseteq \cB_{2}$.
\end{definition}

Here, $\mu_{\uk}M_{1}(\dual{\mu_{\uk}(b)})\in \cC_{1}\subseteq \curly{F}(\qtorus{}{}\cB_{1})$ is a cluster variable, obtained by mutation along $\uk$ from the initial cluster, labelled by $\mu_{\uk}(b)\in \mu_{\uk}(\cB_{1})$, so that the left-hand side of the equation in~ \ref{d:RCM-mut-compat} is the image of this under $f$.  On the other side, we compare this with the mutation in $\cC_{2}$, where we compute the corresponding variable obtained by mutation along $F\uk$, having mapped the label $\mu_{\uk}(b)$ under $\phi_{\uk}$ to the domain of $\mu_{F\uk}M_{2}$.

This varies presentationally from the original definition of \cite{ADS}, usually written as
\[ f(\mu_{x_{l}}\circ \dotsm \circ \mu_{x_{1}}(y))=\mu_{f(x_{l})}\circ \dotsm \circ \mu_{f(x_{1})}(f(y)) \]
because our mutation sequences are not sequences of cluster variables, but in our setup $F\uk$ plays the role of $\mu_{f(x_{l})}\circ \dotsm \circ \mu_{f(x_{1})}$ and $\phi_{\uk}$ that of $f$ in $f(y)$; both $\phi$ and $F$ are ultimately derived from $f$, of course.

The following lemma is needed for our main result and is analogous to results in \cite[Section~3]{ADS}.

\begin{lemma}\label{l:phi-beta-compat}
	For $f\colon \cC_{1}\to \cC_{2}$ a rooted cluster morphism, we have that either
		\[ \dual{\bar{\phi}}(\bplusminus{\beta_{1}(k)}{\dual{\cB_{1}}})=\bplusminus{\beta_{2}(\phi(k))}{\dual{\cB_{2}}} \]	
	or
	\[ \dual{\bar{\phi}}(\bplusminus{\beta_{1}(k)}{\dual{\cB_{1}}})=\bminusplus{\beta_{2}(\phi(k))}{\dual{\cB_{2}}} .\]	
	Hence, either
		\[ \dual{\bar{\phi}}(\beta_{1}(k))=\beta_{2}(\bar{\phi}(k)) \]
	or
	\[ \dual{\bar{\phi}}(\beta_{1}(k))=-\beta_{2}(\bar{\phi}(k)) \]	
	where $\bar{\phi}(b)=\phi(b)$ for all $b\in \ex_{1}$ and $\dual{\bar{\phi}}(\dual{b})=\dual{\phi(b)}$ for all $b\in \cB_{1}$, both extended $\integ$-linearly, with $\phi$ the $\ex$-admissible function induced by $f$.
\end{lemma}

\begin{proof}
	Let us consider condition~\ref{d:RCM-mut-compat} in the definition of rooted cluster morphism for $\uk=(k)$ at $\dual{\mu_{\cB_{1}}(k)}$, with $\phi_{(k)}(\mu_{\cB_{1}}(k))\neq 0$:
		\[ f(\mu_{k}M_{1}(\dual{\mu_{\cB_{1}}(k)})=\mu_{Fk}M_{2}(\dual{\phi_{(k)}(\mu_{\cB_{1}}(k))}). \]
	From Definition~\ref{d:one-step-mut} and the subsequent Lemma, the left-hand side is
	\[ f(\mu_{k}M_{1}(\dual{\mu_{\cB_{1}}(k)}))= f(M_{1}(\bar{\mu}_{k}^{+}(\dual{\mu_{\cB_{1}}(k)})))+f(M_{1}(\bar{\mu}_{k}^{-}(\dual{\mu_{\cB_{1}}(k)}))). \]
	
	On the right-hand side,	$\phi_{(k)}(\mu_{\cB_{1}}(k))=0$ if $\phi(k)=0$, so we must have $\phi(k)\neq 0$ and 
	\begin{align*} \mu_{Fk}M_{2}(\dual{\phi_{(k)}(\mu_{\cB_{1}}(k))}) & = \mu_{\phi(k)}M_{2}(\dual{\phi_{(k)}(\mu_{\cB_{1}}(k))})  \\
		& =\mu_{\phi(k)}M_{2}(\dual{\mu_{\cB_{2}}(\phi(k))}) \\
		&  = M_{2}(\bar{\mu}_{\phi(k)}^{+}(\dual{\mu_{\cB_{2}}(\phi(k))}))+M_{2}(\bar{\mu}_{\phi(k)}^{-}(\dual{\mu_{\cB_{2}}(\phi(k))})) .
	\end{align*}
	
	By algebraic independence, we have either
	\[ f(M_{1}(\bar{\mu}_{k}^{\pm}(\dual{\mu_{\cB_{1}}(k)})))=M_{2}(\bar{\mu}_{\phi(k)}^{\pm}(\dual{\mu_{\cB_{2}}(\phi(k))})) \]	
	or
	\[ f(M_{1}(\bar{\mu}_{k}^{\pm}(\dual{\mu_{\cB_{1}}(k)})))=M_{2}(\bar{\mu}_{\phi(k)}^{\mp}(\dual{\mu_{\cB_{2}}(\phi(k))})) .\]	
	Comparing these monomials, and by the definition of $\phi$ from $f$, we conclude that 
	\[ \dual{\phi}(\bar{\mu}_{k}^{\pm}(\dual{\mu_{\cB_{1}}(k)}))=\bar{\mu}_{\phi(k)}^{\pm}(\dual{\mu_{\cB_{2}}(\phi(k))}) \]	
	or
	\[ \dual{\phi}(\bar{\mu}_{k}^{\pm}(\dual{\mu_{\cB_{1}}(k)}))=\bar{\mu}_{\phi(k)}^{\mp}(\dual{\mu_{\cB_{2}}(\phi(k))}) \]	
	where $\dual{\bar{\phi}}\colon \dual{\integ[\cB_{1}]}\to\dual{\integ[\cB_{2}]}$ is defined on basis elements by $\dual{\bar{\phi}}(\dual{b})\defeq \dual{\phi(b)}$ and extended linearly.
	
	Since $\bar{\mu}_{k}^{\pm}(\dual{\mu_{\cB_{1}}(k)})=\bplusminus{\beta(k)}{\dual{\cB}}-\dual{k}$, these simplify to
	\[ \dual{\bar{\phi}}(\bplusminus{\beta_{1}(k)}{\dual{\cB_{1}}}-\dual{k})=\bplusminus{\beta_{2}(\phi(k))}{\dual{\cB_{2}}}-\dual{\phi(k)} \]	
	or
	\[ \dual{\bar{\phi}}(\bplusminus{\beta_{1}(k)}{\dual{\cB_{1}}}-\dual{k})=\bminusplus{\beta_{2}(\phi(k))}{\dual{\cB_{2}}}-\dual{\phi(k)}. \]	
	
	Using that $\beta_{i}(k)=\bplus{\beta_{i}(k)}{\cB}-\bminus{\beta_{i}(k)}{\cB}$, we conclude that
	\[ \dual{\bar{\phi}}(\beta_{1}(k))=\beta_{2}(\bar{\phi}(k)) \]
	or
	\[ \dual{\bar{\phi}}(\beta_{1}(k))=-\beta_{2}(\bar{\phi}(k)) \]
	and since $k$ was arbitrary, we have the claim.
\end{proof}

Let us say that $f$ is \emph{consistently positive} (respectively, \emph{consistently negative}) if $\dual{\bar{\phi}}(\beta_{1}(k))=\beta_{2}(\bar{\phi}(k))$ (respectively, $\dual{\bar{\phi}}(\beta_{1}(k))=-\beta_{2}(\bar{\phi}(k))$) for all $k$.  Then $f$ being consistently positive means that the diagram
	\[ \begin{tikzcd} \integ[\ex_{1}] \arrow{r}{\bar{\phi}} \arrow{d}[left]{(\beta_{1})_{()}=\beta_{1}} & \integ[\ex_{2}] \arrow{d}[right]{(\beta_{2})_{()}=\beta_{2}} \\ \dual{\integ[\cB_{1}]}  \arrow{r}[below]{\dual{\bar{\phi}}} & \dual{\integ[\cB_{2}]}
	\end{tikzcd} \]
commutes and $f$ consistently negative corresponds to the diagram
	\[ \begin{tikzcd} \integ[\ex_{1}] \arrow{r}{\bar{\phi}} \arrow{d}[left]{(\beta_{1})_{()}=\beta_{1}} & \integ[\ex_{2}] \arrow{d}[right]{-(\beta_{2})_{()}=-\beta_{2}} \\ \dual{\integ[\cB_{1}]}  \arrow{r}[below]{\dual{\bar{\phi}}} & \dual{\integ[\cB_{2}]}
	\end{tikzcd} \]
commuting.  As a shorthand, let us say that $f$ is consistently signed if it is either consistently positive or consistently negative.

The next lemma is based on the first part of \cite[Lemma~3.7]{Gratz}. Denote by $\ex^\phi$ the set \[  \ex^\phi=\ex_{1}\intersection \phi^{-1}(\ex_{2});\] this is the set of exchangeable indices in $\ex_{1}$ such that their image under $\phi$ is also exchangeable.

\begin{lemma}
	Let $f\colon \cC_{1}\to \cC_{2}$ be a rooted cluster morphism and $\phi\colon \cB_{1}^{0}\to \cB_{2}^{0}$ the associated $\ex$-admissible function. If $b\neq c$ and $\phi(b)=\phi(c)$, then $b,c\in \cB_{1}\setminus \ex_{1}$.
\end{lemma}

\begin{proof} We translate the proof given in \cite{Gratz} into our setting. 
	We assume for a contradiction that we have $b,c\in \ex^{\phi}$ with $b\neq c$ and $\phi(b)=\phi(c)$. Then $\mu_{b}M_{1}(\dual{c})=M_{1}(\dual{c})$ and by condition~\ref{d:RCM-mut-compat} of the definition of a rooted cluster morphism $f(\mu_{b}M_{1}(\dual{c}))=\mu_{\phi(b)}M_{2}(\dual{\phi_{(b)}(c)})$, so that
	\begin{align*} f(M_{1}(\dual{c})) & = f(\mu_{b}M_{1}(\dual{c})) \\
		& = \mu_{\phi(b)}M_{2}(\dual{\phi_{(b)}(c)}) \\
		& = \mu_{\phi(c)}M_{2}(\dual{\phi(c)}).
	\end{align*}
Then since $\phi(b)=\phi(c)\in \ex_{2}$, we have 
\[ M_{2}(\dual{\phi(c)})=f(M_{1}(\dual{c}))=\mu_{\phi(c)}M_{2}(\dual{\phi(c)}). \]
But a cluster variable cannot be equal to its own mutation, by algebraic independence, contradicting the assumption.
\end{proof}

We claim that consistently positive rooted cluster morphisms induce morphisms of abstract cluster structures.

\begin{theorem}\label{t:mor-ACS-from-RCM}
	For $i=1,2$, let $(\cC_{i},\sigma_{i})$ be rooted cluster algebras with seeds $\sigma_i = (\ex_{i},\cB_{i},\invfrozen_{i},\beta_{i})$. Let $\cC(\cC_{i})=\cC(\cE_{i},\cX_{i},\beta_{i},\cA_{i},\ip{\blank}{\blank}_{i})$ denote the abstract cluster structure associated to $\cC_{i}$.
	
	Then for each consistently positive rooted cluster morphism $f\colon \cC_{1}\to \cC_{2}$, there exists a morphism of abstract cluster structures \[ \curly{F}=(\bar{\bF},\chi^{f},\alpha^{f})\colon\cC(\cC_{1}) \to \cC(\cC_{2}) \] 
	with $\bar{\bF}$ the functor $\bF\colon \cE(\ex_{1},\cB_{1})\to \cE(\ex_{2},\cB_{1})$ given by
		\[ \bar{\bF}\mu_{k}^{\pm} = \begin{cases} \mu_{\phi_{\uk}(k)}^{\pm} & \text{if}\ \phi_{\uk}(k)\in \mu_{F\uk}(\ex_{2}) \\ \id_{F\uk} & \text{if}\ \phi_{\uk}(k)=0 \end{cases} \]
	and $\alpha^{f}_{\uk}=\dual{\bar{\phi}}_{\uk}$, $\chi^{f}_{\uk}=\bar{\phi}_{\uk}$ natural transformations defined in terms of the $\ex$-admissible function $\phi$ associated to $f$ and its mutations $\phi_{\uk}$.
\end{theorem}

\begin{proof}
	Let us first briefly recall the above constructions leading to the definition of $\bF$.  From the rooted cluster morphism $f$, we obtain the $\ex$-admissible map $\phi\colon \cB_{1}^{0}\to \cB_{2}^{0}$ (Lemma~\ref{l:RCM-induces-phi}) and hence a functor $\bF\colon \cE(\ex_{1},\cB_{1})\to \cE(\ex_{2},\cB_{1})$ defined on objects $\uk\in \curly{K}(\ex_{1},\cB_{1})$ by $\bF\uk=F\uk$ and on morphisms $\vec{\mu}_{k}\colon \uk\to (k,\uk)$ by
	\[ \bF\vec{\mu}_{k}=\begin{cases}
		\vec{\mu}_{\phi_{\uk}(k)} & \text{if}\ \phi_{\uk}(k)\in \mu_{F\uk}(\ex_{2}) \\
		\id_{F\uk} & \text{if}\ \phi_{\uk}(k)=0
	\end{cases} \]
	(Lemma~\ref{l:Fk-is-a-functor}), via the maps $F$ and $\phi_{\uk}$ defined in Definition~\ref{d:Fk-phi-k}.  It is straightforward to see that $\bF$ extends to a functor $\bar{\bF}\colon \cE_{1}\to \cE_{2}$: the category $\cE_{i}$ is the signed path category on the same underlying directed exchange graph $E(\ex_{i},\cB_{i})$ as $\cE(\ex_{i},\cB_{i})$, hence these have the same objects, so we may define $\bar{\bF}\uk=\bF\uk$.  Then on arrows, we act in a similar fashion, respecting the signs: 
	\[ \bar{\bF}\mu_{k}^{\pm} = \begin{cases} \mu_{\phi_{\uk}(k)}^{\pm} & \text{if}\ \phi_{\uk}(k)\in \mu_{F\uk}(\ex_{2}) \\ \id_{F\uk} & \text{if}\ \phi_{\uk}(k)=0 \end{cases} .\]
	
	Then, to satisfy Definition~\ref{d:morphism-in-ACS}, we need to identify natural transformations $\chi^{f}\colon \cX_{1}\to \cX_{2}\op{\bar{\bF}}$ and $\alpha^{f}\colon \cA_{1}\to \cA_{2}\bar{\bF}$ such that $\alpha^{f}\circ \beta_{1}=\beta_{2}\op{\bar{\bF}}\circ \chi^{f}$.  From Theorem~\ref{t:AQCS-from-QCA}, we have that
	\begin{itemize}
		\item $\cX_{i}\colon \op{\cE}_{i}\to \Abcat$ is defined by $\cX_{i} \uk=\integ[\mu_{\uk}(\ex_{i})]$ and $\cX_{i}(\mu_{k}^{\pm}\colon \uk\to (k,\underline{k}))=\bar{\mu}_{k}^{\pm}$ (for $\bar{\mu}_{k}^{\pm}$ the isomorphisms of Lemma~\ref{l:F-isos}) and
		\item $\cA_{i}\colon \op{\cE}_{i}\to \Abcat$ is defined by $\cA_{i} \uk=\dual{\integ[\mu_{\uk}(\cB_{i})]}$ and $\cA_{i}(\mu_{k}^{\pm}\colon \uk\to (k,\underline{k}))=\mu_{k}^{\pm}$ (for $\mu_{k}^{\pm}$ the isomorphisms of Lemma~\ref{l:E-isos}).
	\end{itemize}
	
	Let $\alpha_{()}^{f}\colon \dual{\integ[\cB_{1}]}\to \dual{\integ[\cB_{2}]}$ be given by $\alpha_{()}^{f}(\dual{b})=\dual{\phi(b)}$ on basis elements $\dual{b}\in \dual{\cB_{1}}$ and extended $\integ$-linearly; this is the map $\dual{\bar{\phi}}$ appearing in Lemma~\ref{l:phi-beta-compat}.  That is, $\alpha_{()}^{f}$ is a linearization of $\phi$ combined with taking dual basis elements.  Explicitly,
	\[ \alpha_{()}^{f}(\dual{b}) = \begin{cases} \dual{c} \quad (\text{for}\ f(M_{1}(\dual{b}))=M_{2}(\dual{c})) & \text{if}\ \phi(b)\neq 0,\ \text{so that such a $\dual{c}$ exists} \\ 0 & \text{if}\ \phi(b)=0 \end{cases}. \]
	Here we see the payoff for using $0$ as our additional element disjointly adjoined: we can completely naturally interpret the above definition of $\dual{\phi(b)}$ as being an element of $\dual{\integ[\cB_{2}]}$, with $0$ being the zero element of the group. 
		
	That is, from the map $\phi_{()}=\phi$ on indexing elements obtained from $f$, we can first write down the corresponding function on dual indexing elements (this is notational: we simply add a $*$) and then linearise to obtain a homomorphism of the required free Abelian groups, from the specification on their bases.  In particular, basis elements $\dual{b}$ from $\dual{\integ[\cB_{1}]}$ are either sent to basis elements of $\dual{\integ[\cB_{2}]}$ or specialised to zero.  This corresponds exactly to the function $f$ sending cluster variables either to cluster variables or specialising them to an integer.
	
	Let us similarly define $\chi_{()}^{f}\colon \integ[\ex_{1}]\to \integ[\ex_{2}]$, by $\chi_{()}^{f}(b)=\bar{\phi}(b)=\phi(b)$ on basis elements $b\in \ex_{1}$ and extended $\integ$-linearly.  Indeed, from $\phi_{\uk}$ as in Definition~\ref{d:Fk-phi-k}, we may obtain similarly $\dual{\bar{\phi}}_{\uk}\colon \dual{\integ[\mu_{\uk}(\cB_{1})]}\to \dual{\integ[\mu_{F\uk}(\cB_{2})]}$ and $\bar{\phi}\colon \integ[\mu_{\uk}(\ex_{1})]\to \integ[\mu_{\uk}(\ex_{2})]$, so we may set $\alpha^{f}_{\uk}=\dual{\bar{\phi}}_{\uk}$ and $\chi^{f}_{\uk}=\bar{\phi}_{\uk}$.
	
	To verify that the maps $\alpha^{f}_{\uk}$ define a natural transformation $\alpha^{f}$, we show that the diagram
	\[ \begin{tikzcd}[column sep=25pt,row sep=25pt]
		\cA_{1}\uk \arrow{r}[above]{\alpha_{\uk}^{f}} \arrow{d}[left]{\cA_{1}\mu_{k}^{\pm}} & \cA_{2}\bar{\bF}\uk \arrow{d}[right]{\cA_{2}\bar{\bF}\mu_{k}^{\pm}} \\
		\cA_{1}(k,\uk) \arrow{r}[below]{\alpha_{(k,\uk)}^{f}} & \cA_{2}\bar{\bF}(k,\uk)
	\end{tikzcd} \]
	commutes.
	
	We will do this explicitly for $\uk=()$: the argument is the same in the general case (corresponding to simply shifting our initial clusters) but notationally more complex.  That is, we will show that
	\[ \begin{tikzcd}[column sep=25pt,row sep=25pt]
		\cA_{1}()=\dual{\integ[\cB_{1}]} \arrow{r}[above]{\alpha_{()}^{f}=\dual{\bar{\phi}}} \arrow{d}[left]{\cA_{1}\mu_{k}^{\pm}=\mu_{k}^{\pm}} & \dual{\integ[\cB_{2}]}=\cA_{2}\bar{\bF}() \arrow{d}[right]{\cA_{2}\bar{\bF}\mu_{k}^{\pm}=\mu_{Fk}^{\pm}} \\
		\cA_{1}(k)=\dual{\integ[\mu_{k}(\cB_{1})]} \arrow{r}[below]{\alpha_{(k)}^{f}=\dual{\bar{\phi}}_{(k)}} & \dual{\integ[\mu_{Fk}(\cB_{2})]}=\cA_{2}\bar{\bF}(k)
	\end{tikzcd} \]
	where we have expanded the various definitions.

	Now, for a basis element $\dual{b}\in \dual{\mu_{k}(\cB_{1})}$ of $\cA_{1}(k)$, we have that
	\[ \mu_{k}^{\pm}(\dual{b})= \begin{cases} \dual{b} & \text{if}\ \dual{b}\neq \dual{k} \\
	\bplusminus{\beta_{1}(k))}{\dual{\cB_{1}}}-\dual{\mu_{\cB_{1}}(k)}	& \text{if}\ \dual{b}=\dual{k} 
	\end{cases}. \]
	Then 
	\begin{align*} 
		(\alpha_{(k)}^{f}\circ \mu_{k}^{\pm})(\dual{b}) & =\begin{cases} \dual{\bar{\phi}}_{(k)} & \text{if}\ \dual{b}\neq \dual{k} \\ \dual{\bar{\phi}}_{(k)}(\bplusminus{\beta_{1}(k))}{\dual{\cB_{1}}})-\dual{\bar{\phi}}_{(k)}(\dual{\mu_{\cB_{1}}(k)}) & \text{if}\ \dual{b}=\dual{k} \end{cases}  \\
			& = \begin{cases} \dual{\phi(b)} & \text{if}\ \dual{b}\neq \dual{k} \\
			\dual{\bar{\phi}}(\bplusminus{\beta_{1}(k))}{\dual{\cB_{1}}})-\dual{\mu_{\cB_{2}}(\phi(k))} & \text{if}\ \dual{b}=\dual{k}
			\end{cases}.
	\end{align*}
	On the other hand,
	\begin{align*} 
		(\mu_{Fk}^{\pm}\circ \alpha_{()}^{f})(\dual{b}) & = \begin{cases} \mu_{(\phi(k))}^{\pm}(\dual{\phi(b)}) & \text{if}\ \phi(k)\neq 0 \\ \dual{\phi(b)} & \text{if}\ \phi(k)=0 \end{cases}  \\
		& = \begin{cases} \dual{\phi(b)} & \text{if}\ \phi(k)\neq 0, \dual{\phi(b)}\neq \dual{\phi(k)} \\
			\bplusminus{\beta_{2}(\phi(k))}{\dual{\cB_{2}}}-\dual{\mu_{\cB_{2}}(\phi(k))} & \text{if}\ \phi(k)\neq 0, \dual{\phi(b)}=\dual{\phi(k)} \\ \dual{\phi(b)} & \text{if}\phi(k)=0
		\end{cases}.
	\end{align*}
	By Lemma~\ref{l:phi-beta-compat} and that $f$ is consistently positive, we have that $\alpha^{f}_{(k)}\circ \mu_{k}^{\pm}=\mu_{Fk}^{\pm}\circ \alpha^{f}_{()}$ as required.
	
	Similarly, we have
	\begin{align*} 
		(\chi_{()}^{f}\circ \bar{\mu}_{k}^{\pm})(\mu_{k}(b)) & =\begin{cases} \phi(b)+\evform{\bplusminus{\beta_{1}(k)}{\dual{\cB_{1}}}}{b}\phi(k) & \text{if}\ \mu_{k}(b)\neq \mu_{k}(k) \\ -\phi(k) & \text{if}\ \mu_{k}(b)=\mu_{k}(k) \end{cases}
	\end{align*}
	and
	\begin{align*} 
		(\mu_{Fk}^{\pm}\circ \chi_{(k)}^{f})(\mu_{k}(b)) & = \begin{cases}
			\phi(b)+\evform{\bplusminus{\beta_{2}(k)}{\dual{\cB_{1}}}}{\phi(b)}\phi(k) & \text{if}\ \mu_{k}(b)\neq \mu_{k}(k) \\
			-\phi(k) & \text{if}\ \mu_{k}(b)=\mu_{k}(k), \phi(k)\neq 0 \\
			0=-\phi(k) & \text{if}\ \phi(k)=0
		\end{cases}.
	\end{align*}
	Hence $\chi_{()}^{f}\circ \bar{\mu}_{k}^{\pm}=\mu_{Fk}^{\pm}\circ \chi_{(k)}^{f}$.
	
	That is, $\{ \alpha^{f}_{\uk} \}$ and $\{ \chi^{f}_{uk} \}$ define natural transformations $\alpha^{f}$ and $\chi^{f}$ as required.
	
	It remains to check that $\alpha^{f}\circ \beta_{1}=\beta_{2}\op{\bar{\bF}}\circ \chi^{f}$.  In fact, this follows essentially immediately from the definitions, that $\alpha^{f}$ and $\chi^{f}$ are factorizations and that the $\beta_{i}$ are factorizations, by gluing together the relevant diagrams.
	
	To be more precise, by the remarks after Lemma~\ref{l:phi-beta-compat} defining consistently positive rooted cluster morphisms, we have $\dual{\bar{\phi}}\circ \beta_{1}=\beta_{2}\circ \bar{\phi}$ at the initial cluster.  Then since $\beta_{1}$ and $\beta_{2}$ are factorizations, we see that  $(\beta_{i})_{(k)}=\cA_{i}\mu_{k}^{\pm}\circ (\beta_{i})_{()} \circ \cX_{i}\mu_{k}^{\pm}$.  Then $\alpha^{f}_{(k)}\circ (\beta_{1})_{(k)}=(\beta_{2})_{\op{\bar{\bF}(k)}}\circ \chi^{f}_{(k)}$ follows from $\alpha^{f}_{()}\circ (\beta_{1})_{()}=(\beta_{2})_{\op{\bar{\bF}}()}\circ \chi^{f}_{()}$, and the general case from this, by induction on the length of $\uk$.
	
	Then $\curly{F}$ is a morphism of abstract cluster structures, as claimed.
\end{proof}

\begin{remark}
	We expect that the case of $f$ consistently negative is essentially identical, except that $\curly{F}$ should be a morphism to $\cC(\cE_{2},\cX_{2},-\beta_{2},\cA_{2},\canform{\blank}{\blank}{i})$, i.e.\ with $\beta_{2}$ replaced by its negative.  Then this sign propagates through and via $\bplusminus{-\beta(k)}{}=\bminusplus{\beta(k)}{}$, we will obtain the desired conclusion.
\end{remark}

\begin{remark}
	It is not true, in this level of generality, that any rooted cluster morphism is consistently signed.  Indeed, the claim in Lemma~\ref{l:phi-beta-compat} is local to each $k$ and we have made no such assumption nor connectedness assumptions on our initial data (equivalently, we have not assumed that our abstract cluster structures are indecomposable with respect to the categorical product $\dsum$ of Section~\ref{ss:props-of-ACS-cat} which, as we noted there, models the disjoint union of clusters).  In particular, one would expect that a rooted cluster morphism can be consistently positive or negative on connected components of the exchange quiver but not be consistent over the union. 
	
	An earlier result in this direction is given in \cite{ADS}, when the rooted cluster morphism is bijective, and the claim follows in general by translating the argument of \cite[Lemma~4]{Gratz} to this setting. That is, indeed, rooted cluster morphisms are consistently signed on connected  components. By keeping track of the signs appropriately, one should then be able to remove the hypothesis from Theorem~\ref{t:mor-ACS-from-RCM}.
		
	We suggest that it would be reasonable to investigate other similar properties previously identified for rooted cluster morphisms at the level of morphisms of abstract cluster structures: as the discussion in Section~\ref{s:cat-of-ACS} indicates, the category $\ACScat$ has some advantages over other candidates and, of course, the results then apply in much greater generality to abstract cluster structures coming from geometric, categorical or other types of representation.
\end{remark}

\subsection{Classification via morphisms}\label{ss:classifying-mor}
	
We conclude with a short, sketched worked example of an isomorphism of abstract cluster structures coming from representations of two different types.  We will briefly recall the cluster structure on the coordinate algebra $\curly{O}(\mathrm{Gr}(2,6))$ and write down its associated abstract cluster structure, following the recipe of Section~\ref{ss:AQCS-from-QCA}.  We will then show that this is isomorphic to the abstract cluster structure obtained from the surface model of an unpunctured disk with six marked points (i.e.\ the usual hexagon model), as seen in Section~\ref{ss:triangulations}.
	
As described in \cite{Scott-Grassmannians} (see also \cite{Gr2nSchubertQCA}), the homogeneous coordinate ring $\curly{O}(\mathrm{Gr}(2,6))$ has a cluster algebra structure.  The clusters are collections of Pl\"{u}cker coordinates $\Delta_{ij}$ labelled by pairs $ij$ with $i,j\in \{1,\dotsc ,6\}$; indeed, they are exactly the collections of Pl\"{u}cker coordinates such that their indexing pairs form families of maximally weakly separated sets in the sense of \cite{Leclerc-Zelevinsky}. The Pl\"{u}cker coordinates $\Delta_{i,i+1}$ ($i=1,\dotsc ,6$) are frozen variables and in every cluster.

Furthermore, one can show that there are exchange matrices with skew-symmetric principal part for these clusters that precisely encode three-term Pl\"{u}cker relations and that are related by Fomin--Zelevinsky matrix mutation for adjacent clusters.  It follows that the set of all cluster variables is the set of Pl\"{u}cker coordinates, so that $\curly{O}(\mathrm{Gr}(2,6))$ is a cluster algebra of finite type; in fact, type $A_{3}$.

In what follows, let $c$ denote a cluster of $\curly{O}(\mathrm{Gr}(2,6))$ and write $\curly{P}(c)$ for the set of pairs $ij$ such that $\Delta_{ij}\in c$.  Let $\overline{\curly{P}}(c)$ be the set of pairs $ij$ in $\curly{P}(c)$ such that $|i-j|>1$, i.e.\ such that the associated Pl\"{u}cker coordinates are mutable and not frozen variables.

By the application of Theorem~\ref{t:AQCS-from-QCA}, we obtain an abstract cluster structure $\cC=\curly{G}(2,6)=(\cE,\cX,\beta,\cA,\canform{\blank}{\blank}{})$.  Here, $\cA c=\integ[\curly{P}(c)]$ and $\cX c=\integ[\overline{\curly{P}}(c)]$ and $\beta$ is the $\integ$-linear map such that the associated form $\canform{\blank}{\blank}{\cX}$ has Gram matrix the exchange matrix of $c$.

As is well known, the Grassmannian cluster structure above has a surface model via triangulations of the (unpunctured) hexagon (i.e.\ disk with six marked points).  From the correspondence of the indexing pairs $ij$ with arcs and by examining the construction in Section~\ref{ss:triangulations}, we see that there are bijections between (i) the clusters of $c$ and triangulations of the hexagon, (ii) the sets $\curly{P}(c)$ and $\curly{A}(c)$ (the latter being the arcs) and (iii) the sets $\overline{\curly{P}}(c)$ and $\curly{Q}(c)$ (the latter being the quadrilaterals having pairs $ij\in \overline{\curly{P}}$ as unique non-boundary arc).

One may check---for it is true but not automatic---that the rule \eqref{eq:beta-for-surfaces} for $\beta_{c}$ does indeed coincide with the three-term Pl\"{u}cker relation; an example of this is given in \cite[\S3.1]{Gr2nSchubertQCA} (for $n=8$ rather than $6$, but the computation is local to the four indices involved in the Ptolemy relation).  In the three-term Pl\"{u}cker relation for a quadrilateral whose vertices are sequentially numbered from $1$ to $4$,
\[ \Delta_{13}\Delta_{24}=\Delta_{12}\Delta_{34}+\Delta_{14}\Delta_{23} \]
the left-hand side corresponds to the product of the variables associated to the two diagonals of the quadrilateral and the right-hand side to the two products of variables associated to opposite pairs of sides.

Then, as noted in Section~\ref{ss:triangulations}, the formul\ae\ for quadrilateral and arc mutation, $\mu_{q}^{\pm}(p)$ and $\mu_{q}^{\pm}(a)$, are exactly those occurring in Section~\ref{ss:toric-frames}.

Therefore, using the underlying bijection of the indexing sets for all the data involved, we conclude that there is a morphism $\curly{F}$ of abstract cluster structures from $\curly{G}(2,6)$ to the abstract cluster structure obtained from the hexagon surface model, where the components are an isomorphism $F$ of the path categories coming from the identification of the exchange trees and the natural transformations $\chi$ and $\alpha$ are the natural isomorphisms induced by the bijections of the sets $\overline{\curly{P}}(c)$ (respectively $\curly{P}(c)$) and $\curly{Q}(c)$ (respectively $\curly{A}(c)$).  As all maps involved are essentially ``the identity'', the required equations hold.

That is, $\curly{F}$ is an isomorphism of abstract cluster structures from $\curly{G}(2,6)$ to the abstract cluster structure obtained from the hexagon surface model, as claimed.

We finish by recognising that this is not a surprising result: indeed, were it not as claimed, one ought to challenge our definitions.  The much more interesting question, ripe for exploration, is whether or not there exist ``non-trivial'' isomorphisms among abstract cluster structures and what these might be.  Such a question of course requires a framework in which it can be posed, and we hope that this work establishes such a framework.
	
	\bibliographystyle{halpha}
	\bibliography{biblio}\label{references}
	
\end{document}